\documentclass[11pt,a4paper,twoside]{article}
\usepackage{hyperref}
\usepackage{latexsym,amsfonts,amsmath, amsthm, amssymb}
\usepackage{fullpage}
\usepackage{xcolor,bm, bbm}
\usepackage[utf8x]{inputenc}
\usepackage{array}
\usepackage{mathrsfs}
\usepackage{faktor}
\usepackage{fourier}
\newcommand{\R}{\mathbb R}
\newcommand{\N}{\mathbb N}

\newcommand{\E}{\mathbb E}
\newcommand{\Pro}{\mathbb P}

\newcommand{\Var}{\mathrm{Var}}
\newcommand{\Uni}{\mathrm{Unif}}

\newcommand{\vol}{\mathrm{vol}}

\def\dint{\textup{d}}
\newcommand{\SSS}{\ensuremath{{\mathbb S}}}

\newcommand{\B}{\ensuremath{{\mathbb B}}}

\newcommand{\eps}{\varepsilon}

\newcommand{\ii}{{\rm{i}}}

\DeclareMathOperator{\id}{Id}

\DeclareMathOperator{\rate}{\mathbb I}
\DeclareMathOperator{\Tr}{Tr}

\renewcommand{\Re}{\operatorname{Re}}  %Realteil
\newtheorem{thm}{Theorem}[section]

\newtheorem{cor}[thm]{Corollary}
\newtheorem*{cor*}{Corollary}
\newtheorem{lemma}[thm]{Lemma}
\newtheorem{df}[thm]{Definition}
\newtheorem{proposition}[thm]{Proposition}

\newtheorem{thmalpha}{Theorem}

{
\theoremstyle{definition}

\newtheorem{rmk}[thm]{Remark}

}
\allowdisplaybreaks
\begin{document}

\title
%[Large deviations for orthogonal matrices and random projections]
{\bf Large deviations for random matrices in the orthogonal group and  Stiefel manifold with applications to random projections of product distributions}
%Large deviations for orthogonal matrices and random projections of product distributions

\medskip

\author{Zakhar Kabluchko and Joscha Prochno}

%\thanks{~}

%\keywords{}
%\subjclass{}
%% NB There should be only one primary classification, and zero or
%more secondary classifications.

\date{}

\maketitle

\begin{abstract}
\small

We prove large deviation principles (LDPs) for random matrices in the orthogonal group and Stiefel manifold, determining both the speed and  good convex rate functions that are explicitly given in terms of certain log-determinants of trace-class operators and are finite on the set of Hilbert-Schmidt operators $M$ satisfying $\|MM^*\|<1$. As an application of those LDPs, we determine the precise large deviation behavior of $k$-dimensional random projections of high-dimensional product distributions using an appropriate interpretation in terms of point processes, also characterizing the space of all possible deviations. The case of uniform distributions on $\ell_p$-balls, $1\leq p \leq \infty$, is then considered and reduced to appropriate product measures. Those applications generalize considerably the recent work [Johnston, Kabluchko, Prochno: \emph{Projections of the uniform distribution on the cube -- a large deviation perspective}, Studia Mathematica 264 (2022), 103-119].
%Our proofs resemble a lively interplay of functional analytic, geometric and probabilistic ideas.
%Fix $k\in \N$ and consider a random orthonormal $k$-frame of vectors in $\mathbb R^n$ distributed according to the Haar measure on the Stiefel manifold $\mathbb V_{k,n}$ of all such frames, $n\geq k$. Let $A_n$ be a random $k\times n$ matrix  whose rows are the vectors of the frame.  We show that $(A_n)_{n\geq k}$ satisfies a large deviation principle on the space $[-1,1]^{k\times \infty}$ at speed $n$ with an explicit good rate function. The proof is based on a reduction to Gaussian random matrices and a projective limit approach. As an application, we show that random $k$-dimensional projections of the sequence of uniform distributions on cubes $[-1,1]^n$, $n\geq k$, satisfy a large deviation principle at speed $n$ with an explicit good rate function. The latter generalizes the main result in [Johnston, Kabluchko, Prochno: \emph{Projections of the uniform distribution on the cube -- a large deviation perspective}, Studia Mathematica (2021+)] to multi-dimensional random projections. In fact, we also prove a corresponding result for the uniform distribution on $\ell_p$-balls, when $1\leq p<\infty$, and for random projections of product measures on $\mathbb R^n$. Last but not least, we obtain a large deviation principle for random Haar orthogonal matrices on the space $[-1,1]^{\infty \times \infty}$, determining an explicit good rate function which is finite on the unit ball in the  Hilbert-Schmidt norm.
\medspace
\end{abstract}

\tableofcontents

% % % % % % % % % % % % % % % % % % % % % % % % % % % % % % % % %
\section{Introduction and main results}
% % % % % % % % % % % % % % % % % % % % % % % % % % % % % % % % %

% % % % % % % % % % % % % % % % % % % % % % %
\subsection{Introduction}
% % % % % % % % % % % % % % % % % % % % % % %

The systematic study of large random matrices goes back to Wishart and his work on the statistical analysis of large samples \cite{W1928}. Ever since, random matrices have entered numerous areas of mathematics and applied sciences beyond probability theory and statistics, among others asymptotic geometric analysis, combinatorics, number theory, operator theory, nuclear physics, quantum field theory or theoretical neuroscience, just to mention a few. In the last  decades, significant effort has been made to understand their spectral statistics (particularly from a local point of view) and determining various universality properties; we refer the reader to \cite{ABDF2015} and \cite{AGZ2010} for more information.

In this work we are not interested in universality properties or the typical behavior of random matrices, but rather in their atypical behavior that is described by the theory of large deviations. There is quite some work in this direction, for instance dealing with large deviations for empirical measures of random matrices and we refer the reader to \cite{AGZ2010,guionnet_surv} and the references cited therein. The focus in this paper is actually a different, two-fold one. In the first part of the paper, we study the large deviation behavior of random matrices in the orthogonal group and Stiefel manifold and then apply our results in the second part to understand the atypical behavior of $k$-dimensional projections of high-dimensional product distributions and uniform distributions on $\ell_p$-balls. Those applications are motivated by problems concerning the asymptotic theory of convex bodies as studied in geometric functional analysis and high-dimensional probability theory. Indeed, the asymptotic theory of normed spaces, and the study of high-dimensional convex bodies in particular, has attracted considerable attention in the last decades due to its intimate connection to other mathematical and applied disciplines, and depicts a lively interplay of combinatorics, discrete and convex geometry, functional analysis and probability theory. Numerous powerful ideas, tools, and results have a probabilistic flavor and the study of random objects and random geometric quantities is a central part of the theory with important applications, for instance, in approximation theory, information-based complexity, or compressed sensing \cite{ALLPT2011, CGLP2012, FR2013, HKNPV2021, HPS2021, HPU2019, KS2020, MU2020}. Random matrices and random matrix techniques play a pivotal role in many of those applications and laws of large numbers or central limit theorems belong to the classical body of research
%, in particular in view of the study of concentration of measure phenomena,
and have been obtained in various situations, for instance,
\cite{AR2020, APT2019,ABP2003,ABBN2004,KPT2019_I,KPT2019_cube,K2007,MM2007,PPZ14,R2005,SS1991,Schmu2001} to name a few. The study of large deviations, more precisely of large deviation principles (LDPs) -- intensely investigated in probability theory and statistical mechanics since the $1960$s -- has only recently attracted attention in the asymptotic theory of convex bodies. Contrary to the universality in a central limit theorem, which restricts the information that can be retrieved, for instance, from lower-dimensional projections (recall that Klartag's central limit theorem \cite{K2007} says that most lower-dimensional marginals are close to being Gaussian), a large deviation principle is typically distribution dependent and therefore still carries information about the underlying distribution. In geometric terms, this allows one to distinguish between high-dimensional convex bodies via their lower dimensional projections; in the setting of $\ell_p$-balls, this was shown by Gantert, Kim, and Ramanan \cite{GKR2017} and Alonso-Guti\'errez, Prochno, and Th\"ale \cite{APT2018}. Moreover, an interesting connection between the study of large (and moderate) deviations for log-concave distributions and the famous Kannan--Lov\'asz--Simonovits conjecture was established in \cite{APT2020}. Other than that a variety of large deviation results have been obtained in the last five years, among others, \cite{GKR2016,KPT2019_sanov, KPT2019_II,KPT2019_cube,K2021, KLR2019, KR2018, KR2021, LR2020}. Beyond that, in \cite{KP2021} and the subsequent works \cite{AGP2020,JP2020}, it has been demonstrated how ideas and methods from large deviation theory, such as the maximum entropy principle, its relation to Gibbs measures, and Gibbs conditioning, allow one to lift classical results for $\ell_p$-balls to more general symmetric Banach spaces (similar ideas have recently been used by Barthe and Wolff \cite{BW2021} studying Orlicz spaces). In particular, this puts the frequently used Schechtman-Zinn probabilistic representation of the uniform distribution on an $\ell_p$-ball \cite{SchechtmanZinn} into a new perspective. Quite recently, Dadoun, Fradelizi, Gu\'edon, and Zitt used large deviation techniques to obtain a strong version of the variance conjecture for Schatten $p$-balls for $p>3$ \cite{DFGZ2021}.

Before being more specific and precise about what we prove in this paper, let us describe the work that motivated our study of large deviations for random matrices in the orthogonal group and Stiefel manifold. In \cite{JKP2021}, the authors studied the large deviation behavior of random projections of uniform distributions on cubes $[-1,1]^n$ as $n\to\infty$. To be more precise, let $\Theta^{(n)}$ be uniformly distributed on the Euclidean unit sphere $\SSS^{n-1}$. Consider the random probability measure $\mu_{\Theta^{(n)}}$ on $\R$ obtained by projecting the uniform distribution on $[-1,1]^n$ to the line spanned by $\Theta^{(n)}$. More precisely, we put
\[
\mu_{\Theta^{(n)}}(A) := \frac 1 {2^n} \int_{[-1,1]^n}  \mathbb 1_{\{\langle u, \Theta^{(n)} \rangle \in A\}}
\,\dint u
\]
for Borel subsets $A$ of $\mathbb{R}$. Then, as a simple application of Lindeberg's CLT, $\mu_{\Theta^{(n)}}$ converges weakly to a Gaussian distribution of variance $1/3$ as $n\to\infty$. In fact, Klartag's celebrated central limit theorem~\cite{K2007} states that the Gaussian behavior is universal and holds if the cube is replaced by a general $n$-dimensional isotropic convex body, as $n\to\infty$. On the other hand, the large deviation behavior  is not universal. Indeed, the main result in~\cite[Theorem A]{JKP2021} establishes an LDP for the sequence $\mu_{\Theta^{(n)}}$, $n\in\N$, at speed $n$ and with an explicit good rate function $\rate$ defined on the space $\mathcal M_1(\R)$ of probability measures on $\R$ by
\[
\rate\big(\nu(\alpha)\big) := - \frac{1}{2} \log \big( 1 - ||\alpha||_2^2\big),
\]
where $\nu(\alpha)$ is the law of the random variable
\begin{align*}
\sqrt{1 - ||\alpha||_2^2 } \frac{Z}{\sqrt 3} + \sum_{ k = 1}^\infty \alpha_k U_k,
\end{align*}
with $Z$ being a standard Gaussian independent of  $U_1,U_2,\ldots$ which are i.i.d.\ $\Uni[-1,1]$, and $\alpha=(\alpha_{i})_{i\in\N}$ is a non-increasing sequence of non-negative reals with $||\alpha||_2<1$. Whenever $\nu$ is not of the form $\nu(\alpha)$, we have $\rate(\nu)= +\infty$. Replacing the uniform distribution on the cube $[-1,1]^n$ by the uniform distribution on its vertices leads to a different rate function, as was shown in~\cite[Theorem~B]{JKP2021}, thus demonstrating the lack of universality in the large deviations behavior.

In this paper, we put this into a much wider perspective, significantly generalizing the results obtained in~\cite{JKP2021} by proving, on the one hand, a multivariate version of the main result in \cite{JKP2021} for product distributions and, on the other, by proving a  corresponding result for the whole class of uniform distributions on $\ell_p$-balls. As already pointed out, this requires understanding the precise large deviation behavior of random orthogonal matrices and we describe our main results below.

Let us point out that large deviation principles for (random) projections of random points in $\ell_p$-balls accompanied by  LDPs for random Stiefel matrices have been investigated by Kim and Ramanan in~\cite{KR2021}. Their setting differs from ours. To be more specific, Kim and Ramanan~\cite[Theorems~2.4, 2.6, 2.7]{KR2021} essentially consider a random vector $X^{(n)}$ distributed uniformly on an $\ell_p^n$-ball and prove an LDP on $\R^k$ for the projection of the \textit{random point} $X^{(n)}$ onto a uniform, random $k$-dimensional subspace, as $n\to\infty$. We consider random projections of the \textit{probability distribution} of $X^{(n)}$ and prove an LDP on the space of probability measures on $\R^k$ rather than on the space $\R^k$ itself. Although both questions sound similar, the techniques used in the proofs are completely different (as are the rate functions) and it seems that neither LDP implies the other by contraction or any other elementary transformation argument. In any case, both questions require the study of certain types of LDPs for random Stiefel matrices. Kim and Ramanan~\cite[Theorem~2.8]{KR2021} prove an LDP for the empirical measure of the columns of a  random Stiefel $k\times n$-matrix on the space of probability measures on $\R^k$ (as $n\to\infty$), while we prove an LDP for the random Stiefel matrix itself on the space of $k\times \infty$-matrices. In some sense, Kim and Ramanan~\cite[Theorem~2.8]{KR2021} study atypical empirical measures of the columns of a matrix (in the spirit of Sanov's theorem), while we study atypical matrices themselves.

% % % % % % % % % % % % % % % % % % % % % %
\subsection{Main results}
% % % % % % % % % % % % % % % % % % % % % %

We shall denote by $\mathbb V_{k,n}$ the Stiefel manifold of orthonormal $k$-frames in $\R^n$, for $n\in \N$ and $k\in \{1,\ldots,n\}$.  The elements of $\mathbb V_{k,n}$ are orthonormal $k$-tuples $(u_1,\ldots,u_k)$ of vectors from $\R^n$. We agree to identify any such tuple with a matrix $V\in \R^{k\times n}$ whose rows are the vectors $u_1,\ldots,u_k$. Then, the orthonormality condition can be expressed as $VV^* = \id_{k\times k}$. The Stiefel manifold $\mathbb V_{k,n}$ is endowed  with its natural Haar (or uniform) probability measure $\mu_{k,n}$; see Section~\ref{subsec:stiefel_manifold} for more details. The notion of a large deviation principle is introduced in Definition \ref{def:ldp}. In what follows, the following set of matrices is of special interest,
\[
\mathcal R_2^{k\times \infty} := \Big\{A=(A_{ij})_{i,j=1}^{k,\infty}\in\R^{k\times\infty}\,:\, (A_{ij})_{j\in\N}\in\ell_2, \,i=1,\dots,k \Big\},
\]
i.e., $\mathcal R_2^{k\times \infty}$ is the set of all $k\times \infty$ matrices whose rows belong to the space $\ell_2$ of square summable sequences.

%\subsubsection*{LDP's for orthonormal frames}

\medskip
\noindent
\textbf{LDP for Stiefel matrices.}
Our first result is a large deviation principle for a random matrix $V_{k,n}$ distributed according to $\mu_{k,n}$ which is valid in the regime when $k\in \N$ is fixed and $n\to\infty$. We shall abuse notation and identify $V_{k,n}$ with an infinite $k\times \infty$-matrix obtained from $V_{k,n}$ by adding infinitely many zero columns. Therefore, we can consider $V_{k,n}$ as a random element with values in the space $[-1,1]^{k\times \infty}$ of all $k\times \infty$-matrices whose entries have absolute values $\leq 1$. Endowed with the topology of entrywise convergence, the space  $[-1,1]^{k\times \infty}$ becomes compact according to Tikhonov's theorem.

\begin{thmalpha}[LDP for Stiefel matrices]\label{thm:ldp for stiefel manifold}
Fix $k\in\N$. For $n\geq k$ let $V_{k,n}$ be chosen at random from the Stiefel manifold $\mathbb V_{k,n}$ with respect to the uniform distribution $\mu_{k,n}$. Then the sequence $V_{k,n}$, $n\geq k$, satisfies an LDP on $[-1,1]^{k\times \infty}$ at speed $n$ with good convex rate function $\rate:[-1,1]^{k\times\infty} \to [0,+\infty]$ given by
\[
\rate(A) :=
\begin{cases}
-\frac{1}{2} \log \det\big(\id_{k\times k} - AA^*\big) &:\, A\in \mathcal R_2^{k\times\infty} \text{ and } \|AA^*\| <1,\\
+\infty &:\, \text{otherwise}.
\end{cases}
\]
%where $\mathcal R_2^{k\times \infty}$ is the set of all $k\times \infty$ matrices whose rows belong to the space $\ell_2$ of square summable sequences.
\end{thmalpha}
Here, $\|AA^*\|$ denotes the operator norm of the square matrix $AA^*\in \R^{k\times k}$, which is the Gram matrix of the $k$ rows of $A$ and is well defined provided $A\in \mathcal R_2^{k\times \infty}$.  Observe that $\|AA^*\|<1$ if and only if the matrix $\id_{k\times k}-AA^*$ is positive definite.  In this case, $\id_{k\times k}-AA^*$ has positive determinant, implying that the rate function above is indeed well-defined. An LDP closely related to the special case $k=1$ of Theorem~\ref{thm:ldp for stiefel manifold} can be found in~\cite[Theorem~3.7]{barthe_etal}.

\medskip
\noindent
\textbf{LDP for orthogonal matrices.}
In the special case $k=n$, we can view $O^{(n)}:=V_{n,n}$ as a random orthogonal matrix distributed according to the Haar measure on the orthogonal group $\mathcal O(n)$. As $n\to\infty$, the entries of $O^{(n)}$ converge to independent standard normal random variables after multiplication by $\sqrt n$; see~\cite{jiang1,jiang2} for much stronger results. The maximal entry is known to be of order $2\sqrt{(\log n)/n}$; see~\cite{jiang_maximal}.  In the following theorem, we characterize the atypical behavior of the entries of large orthogonal matrices. We denote by $\mathcal S_2$ the class of Hilbert--Schmidt operators on the (real) Hilbert space $\ell_2$. Any  operator $T\in \mathcal S_2$ can be identified with a $\infty\times \infty$-matrix $(t_{i,j})_{i,j\in \N}$ such that the Hilbert-Schmidt norm (Frobenius norm) is finite, i.e., $\sum_{i,j\in \N} t_{i,j}^2 <\infty$.
%For more information we refer the reader to Section \ref{sec:prelim}. ZK: I removed this because section 2 does not contain information on Hilbert-Schmidt operators.

\begin{thmalpha}[LDP for orthogonal matrices]\label{thm:ldp orthogonal group}
For $n\in\N$, let $O^{(n)}$ be an $n\times n$ matrix chosen uniformly at random  from the orthogonal group $\mathcal O(n)$ with respect to the Haar probability measure. We identify $O^{(n)}$ with an element from $[-1,1]^{\N\times \N}$ by filling up $O^{(n)}$ with zeros. Then the sequence $O^{(n)}$, $n\in\N$, satisfies an LDP on the compact space $[-1,1]^{\N\times \N}$ endowed with the product topology with good convex rate function $\rate:[-1,1]^{\N\times \N}\to[0,\infty]$ given by
\[
\rate(T) :=
\begin{cases}
- \frac{1}{2}\log \det \big(\id_{\infty\times \infty} - TT^*\big) & : T\in\mathcal S_2 \text{  and  } \,\|TT^*\| <1\\
+\infty &: \text{ otherwise}.
\end{cases}
\]
\end{thmalpha}
In the previous theorem, the determinant is to be understood in the sense explained, for instance, in \cite[p.61]{GGK2000}. This means that if $S$ is a trace-class operator (i.e., an operator in the Schatten $1$-class $\mathcal S_1$) and $S_n$, $n\in\N$, are finite rank operators such that $\|S-S_n\|_{\mathcal S_1}\to 0$ as $n\to\infty$, then
\[
\det\big(\id_{\infty\times \infty} - S\big) := \lim_{n\to\infty}\det\big(\id_{\infty\times \infty}-S_n\big).
\]
Alternatively, $T$ is Hilbert--Schmidt if and only if  $TT^*$ is a trace-class operator (that is, $TT^*\in\mathcal S_1$). In this case,  by Lidskii's trace theorem~\cite[p.~63]{GGK2000}, we have
$$
\det\big(\id_{\infty\times \infty} - TT^*\big)
=
\prod_{i=1}^\infty (1-\lambda_i(TT^*))
=
\prod_{i=1}^\infty (1-\lambda_i(T^*T))
=
\det\big(\id_{\infty\times \infty} - T^*T\big),
$$
where $\lambda_i(TT^*)$ are the eigenvalues of $TT^*$ taken with multiplicities (which coincide with the eigenvalues of $T^*T$ by  Corollary 2.2 on p.~51 of~\cite{GGK2000}).
If, additionally, $\|TT^*\| <1$, then the product on the right-hand side converges to a number in $(0,1]$ since $\sum_{i=1}^\infty |\lambda_i(TT^*)|<\infty$ by the trace-class property of $TT^*$.
%The rate function in Theorem \ref{thm:ldp orthogonal group} is finite for $T\in\mathcal S_2$, because  $TT^*\in\mathcal S_1$) and of course from $$. The details are explained in Section \ref{sec: Proof of Theorem B}.
Let us mention that there are many known LDP's for random matrices~\cite{guionnet_surv}, but these usually deal with the empirical eigenvalue distribution and are of different type than Theorem~\ref{thm:ldp orthogonal group}.

\medskip
\noindent
\textbf{LDP for random projections of uniform distribution on $\ell_p^n$-balls.}
The next result deals with the case of random $k$-dimensional projections of the uniform distribution on an $\ell_p^n$-ball for $1\leq p <\infty$. Recall that the unit ball in $\ell_p^n$ is given by
\[
\B_p^n := \Bigg\{ x=(x_i)_{i=1}^n \,:\, \|x\|_p = \Big(\sum_{i=1}^n|x_i|^p\Big)^{\frac{1}{p}} \leq 1 \Bigg\}.
\]
Let $X^{(n)}$ be a random vector uniformly distributed on $n^{1/p}\B_p^n$. Given some element $V\in \mathbb V_{k,n}$ of the Stiefel manifold, which we view as a linear operator $V:\R^n \to\R^k$, we consider the projection of the uniform distribution on $n^{1/p}\B_p^n$ by $V$, which is a probability measure on $\R^k$ defined by
$$
\mu_V(A)
=
\Pro[V X^{(n)} \in A]
=
\Pro\Big[\big(\langle R_1(V),X^{(n)} \rangle,\dots,\langle R_k(V),X^{(n)} \rangle\big) \in A \Big],
\qquad
A\subset \R^k \text{ Borel},
$$
where $R_1(V),\dots,R_k(V)$ are the rows of $V$.
Let now $V_{k,n} : \R^n \to\R^k$ be a random element in the Stiefel manifold $\mathbb V_{k,n}$ chosen with respect to $\mu_{k,n}$.
%Then, if $x\in\R^n$, $V_{k,n}x = \big(\langle R_1(V_{k,n}),x \rangle,\dots,\langle R_k(V_{k,n}),x \rangle\big)\in\R^k$ is a random projection of $x$ onto $\R^k$, .  the probability measure
%\[
%\mu_{V_{k,n}}(A) := \Pro\Big[\big(\langle R_1(V_{k,n}),X^{(n)} \rangle,\dots,\langle R_k(V_{k,n}),X^{(n)} \rangle\big) \in A \Big], %\qquad A\subset \R^k \text{ Borel}.
%\]
%{\color{red} Z: Why these scalar products? Can we write $V_{k,n} X^{(n)}$?}
The sequence $\mu_{V_{k,n}}$, $n\geq k$, is a sequence of random probability measures consisting of the random $k$-dimensional projections of the uniform distributions on the balls $n^{1/p}\B_p^n$. We view each $\mu_{V_{k,n}}$ as a random element with values in the space of probability measures on $\R^k$, denoted by $\mathcal M_1(\R^k)$ and endowed with the topology of weak convergence.

\begin{thmalpha}[LDP for random projections of uniform distribution on $\ell_p^n$-balls]\label{thm:ldp multidimensional projection unif distr pball}
%\textcolor{red}{ldp multidimensional random projections of uniform distr. on $\ell_p$ balls}
Fix some $1\leq p < \infty$ with $p\neq 2$ and some $k\in \N$. Then, the sequence of random probability measures $\mu_{V_{k,n}}$, $n\geq k$, satisfies an LDP on $\mathcal M_1(\R^k)$ with speed $n$ and a good rate function $\rate: \mathcal M_1(\R^k) \to [0,+\infty]$ defined as follows:
$$
\rate(\nu) = -\frac{1}{2} \log \det\big(\id_{k\times k} - AA^*\big),
$$
if $\nu\in \mathcal M_1(\R^k)$  admits a representation of the form
\begin{equation}\label{eq:rep_law_sum_p_gauss}
\nu = \text{Law}\left(\sum_{j=1}^\infty C_j(A) Z_j + \sigma_p (\id_{k\times k}-AA^*)^{1/2} N_{k}\right)
\end{equation}
for some matrix $A\in \mathcal R_2^{k\times\infty}$ such that $\|AA^*\| < 1$, where $C_1(A), C_2(A),\ldots$ are the columns of $A$,  the random variables $Z_1,Z_2,\ldots$ are i.i.d.\ with the  generalized $p$-Gaussian density
$$
f_p(x):= {1\over 2p^{1/p}\Gamma(1+{1\over p})}\,e^{-|x|^p/p},
\qquad
%p\neq \infty,
%\qquad
%f_\infty(x):= \frac 12 \mathbb 1_{[-1,1]}(x),
%\qquad
x\in \R,
$$
and variance
$$
\sigma_p^2 := \E[Z_1^2] = p^{2/p}\frac{\Gamma(3/p)}{\Gamma(1/p)},
%\qquad
%p\neq \infty,
%\qquad
%\sigma_\infty^2 = \frac 13,
$$
and, independently, $N_k$ is a $k$-dimensional standard Gaussian random vector. If $\nu$ does not admit such a representation, then $\rate(\nu) = +\infty$.
\end{thmalpha}

We shall show that the set $\mathcal K_{k, p}$ of probability measures on $\R^k$ admitting a representation of the form~\eqref{eq:rep_law_sum_p_gauss} with some matrix $A\in \mathcal R_2^{k\times\infty}$ such that $\|AA^*\| \leq  1$  (where we allow equality) is compact in the topology of weak convergence on $\mathcal M_1(\R^k)$. Moreover, such a representation of $\nu$, if it exists, is unique up to a signed permutation of the columns of the matrix $A$. Since the matrix $AA^*$ does not change under signed permutations of the columns of $A$, the rate function $\rate$ is well-defined. It is the uniqueness of the representation~\eqref{eq:rep_law_sum_p_gauss} which forces us to exclude the case $p=2$. Observe that this case is anyway not interesting because then $\mu_{V_{k,n}}$, which does not depend on the direction of the projection,  becomes deterministic. The function $\rate(\nu)$ vanishes if and only if $\nu= \mathcal N(0, \sigma_p^2 \id_{k\times k})$ is the isotropic Gaussian law on $\R^k$ with variance $\sigma_p^2$, which corresponds to the choice $A=0$ in~\eqref{eq:rep_law_sum_p_gauss}. This leads to the following consequence.
\begin{cor*}
In the setting of Theorem~\ref{thm:ldp multidimensional projection unif distr pball}, the following holds with probability $1$: the sequence $\mu_{V_{k,n}}$ converges to  $\mathcal N(0, \sigma_p^2 \id_{k\times k})$  in the weak topology of $\mathcal M_1(\R^k)$ as $n\to\infty$.
\end{cor*}
\begin{proof}
Let $O$ be any weak neighborhood of $\mathcal N(0, \sigma_p^2 \id_{k\times k})$ in $\mathcal M_1(\R^k)$. The lower semi-continuous  function $\rate$ does not vanish on the compact set $\mathcal K_{k,p}\backslash O$, and hence its minimum $m$ there satisfies $m>0$. The LDP stated in Theorem~\ref{thm:ldp multidimensional projection unif distr pball} implies that $\Pro [\mu_{V_{k,n}}\notin O] = O(e^{-mn/2})$. An application of the lemma of Borel--Cantelli completes the proof.
\end{proof}

Note that a similar claim holds for general isotropic convex bodies by Klartag's central limit theorem for convex bodies \cite[Theorem 1.3]{K2007} using a standard Borel-Cantelli argument.
Returning to Theorem~\ref{thm:ldp multidimensional projection unif distr pball}, we observe that the compact set $\mathcal K_{k,p}$ %(after excluding all elements corresponding to $A$'s with $\|AA^*\|=1$)
encodes all possible deviations of the projection $\mu_{V_{k,n}}$ which have probabilities decaying exponentially in $n$ (while the probabilities of deviations not belonging to $\mathcal K_{k,p}$ decay superexponentially).
 In Section~\ref{sec:proof_theorem_D} we shall construct a natural homeomorphism  from $\mathcal K_{k,p}$ to a certain space of locally finite point configurations on $[-1,1]^k \backslash\{0\}$ endowed with the vague topology.
We conjecture that the space $\mathcal K_{k,p}$ is homeomorphic to the Hilbert cube $[0,1]^{\N}$ endowed with the product topology.

\medskip
\noindent
\textbf{LDP for random projections of product measures.}
Theorem~\ref{thm:ldp multidimensional projection unif distr pball} will be deduced from an LDP for random $k$-dimensional projections of product measures. To state it, let $Y^{(n)}=(Y_1,\ldots,Y_n)$ be a random vector whose components are i.i.d.\ random variables $Y_1,\ldots, Y_n$. Given some deterministic Stiefel matrix $V\in \mathbb V_{k,n}$, we define the projection of the distribution of $Y^{(n)}$ by $V$ as follows:
$$
\widetilde \mu_V(A)
=
\Pro[V Y^{(n)} \in A]
=
\Pro\Big[\big(\langle R_1(V),Y^{(n)} \rangle,\dots,\langle R_k(V),Y^{(n)} \rangle\big) \in A \Big],
\qquad
A\subset \R^k \text{ Borel}.
$$
As before, we are interested in random projections which are obtained by letting $V=V_{k,n}$ be random with the uniform distribution $\mu_{k,n}$ on $\mathbb V_{k,n}$.
\begin{thmalpha}[LDP for random projections of product measures]\label{thm:ldp multidimensional projection_product_measures}
Consider a random vector $Y^{(n)}=(Y_1,\ldots,Y_n)$, where $Y_1,Y_2,\ldots$ are non-Gaussian i.i.d.\ random variables with symmetric distribution (meaning that $Y_1$ has the same law as $-Y_1$) and $\E[|Y_1|^p] <+\infty$ for all $p\in \N$. Let $\sigma^2 := \E[Y_1^2]>0$ be the variance of $Y_1$.
%Assume additionally that the moment generating function $\E \eee^{tY_1}$ is finite for all $t$ in a sufficiently small interval $(-\eps_0,\eps_0)$.
Then, the sequence of random probability measures $\widetilde \mu_{V_{k,n}}$, $n\geq k$, satisfies an LDP on $\mathcal M_1(\R^k)$ with speed $n$ and a good rate function $\rate: \mathcal M_1(\R^k) \to [0,+\infty]$ defined as follows:
$$
\rate(\nu) = -\frac{1}{2} \log \det\big(\id_{k\times k} - AA^*\big),
$$
if $\nu\in \mathcal M_1(\R^k)$  admits a representation of the form
\begin{equation}\label{eq:rep_law_sum_p_gauss_product_meas}
\nu = \text{Law}\left(\sum_{j=1}^\infty C_j(A) Y_j + \sigma (\id_{k\times k}-AA^*)^{1/2} N_{k}\right)
\end{equation}
for some $A\in \mathcal R_2^{k\times \infty}$ such that $\|AA^*\| < 1$. Here, $N_k$ is a $k$-dimensional standard Gaussian random vector independent of $Y_1,Y_2,\ldots$. If $\nu$ does not admit such a representation, then $\rate(\nu) = +\infty$.
\end{thmalpha}

The previous theorem contains as special cases large deviation principles for a sequence of random $k$-dimensional projections of the  uniform distributions on the cubes $[-1,1]^n$ and the discrete cubes $\{-1,+1\}^n$, as $n\to\infty$.
The former example (which is the case $p=\infty$ missing in Theorem~\ref{thm:ldp multidimensional projection unif distr pball}) generalizes the main result in \cite{JKP2021} to multi-dimensional random projections. We shall show that the representation in~\eqref{eq:rep_law_sum_p_gauss_product_meas}, if it exists, is unique up to a signed permutation of the columns of $A$, which implies that the function $\rate$ is well defined. The assumption of non-Gaussianity is crucial for the uniqueness, as is the finiteness of \textit{all} moments of $Y_1$ (which cannot be replaced by the finiteness of \textit{some fixed} moment; see the discussion in Remark~\ref{rem:finite_moments_assumption}. The symmetry assumption on $Y_1$ will be discussed in Remark~\ref{rem:symmetry_assumption}. We shall prove that the set $\mathcal K_{k, Y_1}$ of probability measures $\nu$ admitting representation~\eqref{eq:rep_law_sum_p_gauss_product_meas} with $\|A A^*\|\leq 1$ is compact in the weak topology of $\mathcal M_1(\R^k)$.
%It is interesting to note that the uniqueness of the representation~\eqref{eq:rep_law_sum_p_gauss_product_meas} fails under the sole assumption of the finite second moment (and excluding the Gaussian distribution); see the comments after ????? for more details.

% % % % % % % % % % % % % % % % % % % % % %
\subsection*{Organization of the paper}
% % % % % % % % % % % % % % % % % % % % % %

The remainder of this manuscript is organized as follows. In Section \ref{sec:prelim} we introduce notation and some of the fundamental results and concepts used throughout the paper. The proof of the LDP on the Stiefel manifold (Theorem \ref{thm:ldp for stiefel manifold}) is presented in Section \ref{sec:proof of stiefel ldp}. In Section \ref{sec: Proof of Theorem B}, we prove the Theorem \ref{thm:ldp orthogonal group} on the LDP for the orthogonal group. The application to $k$-dimensional random projections of uniform distributions on $\ell_p$-balls is presented in Section \ref{sec:k-dimensional projections lp}, where we reduce this LDP to the one for projections of high-dimensional product distributions. We deal with the latter in the last part of the paper, Section \ref{sec:proof_theorem_D}.

% % % % % % % % % % % % % % % % % % % %
\section{Notation and preliminaries}\label{sec:prelim}
% % % % % % % % % % % % % % % % % % % %

Let us briefly recall (and complement) the basic notation used throughout this paper.
If $n\in\N$, then $\mathbb{S}^{n-1}:=\{x\in\R^n\,:\,\|x\|_2=1 \}$ is the Euclidean unit sphere, and the cube in $\R^n$ is denoted by $\B_\infty^n:=[-1,1]^n$. The standard inner product on $\R^n$ is denoted by $\langle \cdot, \cdot\rangle$. For a Borel measurable set $A\subset \R^n$, we denote by $\vol_n(A)$ its $n$-dimensional Lebesgue measure. For a set $A\subset\R^n$, we denote by $A^\circ$ and $\overline{A}$ its interior and closure, respectively. The group of orthogonal $n\times n$ matrices will be denoted by $\mathcal{O}(n)$ and we shall write $A^*$ for the adjoint of a matrix $A$.
Also, by $\id_{k\times k}$ denote the $k\times k$ identity matrix.
%and by $\id_{\R^k}:\R^k\to\R^k$ the identity map.
For a linear mapping $T:\R^d\to\R^n$ (for some $n\in\N$) we let $T^*$ be the adjoint operator satisfying $\langle Tx,y\rangle=\langle x,T^*y\rangle$ for all $x\in\R^d$ and $y\in\R^n$.
%We denote by $\mathcal M(m\times n,\R)$ the space of real $m\times n$ matrices.

% % % % % % % % % % % % % % % % % % % % % % % % % % %
\subsection{Elements from large deviation theory}
% % % % % % % % % % % % % % % % % % % % % % % % % % %
Let us continue with some notions and results from large deviation theory. For a thorough introduction to this topic, we refer the reader to \cite{DZ2010}.

\begin{df}\label{def:ldp}
Let $(\xi_n)_{n\in\N}$ be a sequence of random elements taking values in some metric space $M$. Further, let $(s_n)_{n\in\N}$ be a sequence of positive reals
%positive sequence
with $s_n\uparrow\infty$ and $\mathcal{I}:M\to[0,+\infty]$ be a lower semi-continuous function.
We say that $(\xi_n)_{n\in\N}$ satisfies a (full) large deviations principle (LDP) with speed $s_n$ and a rate function $\mathcal{I}$ if
\begin{equation}\label{eq:LDPdefinition}
\begin{split}
-\inf_{x\in A^\circ}\mathcal{I}(x)
\leq\liminf_{n\to\infty}{1\over s_n}\log\Pro\left[\xi_n \in A \right]
\leq\limsup_{n\to\infty}{1\over s_n}\log\Pro\left[\xi_n \in A \right]\leq -\inf_{x\in\overline{A}}\mathcal{I}(x)
\end{split}
\end{equation}
for all Borel sets $A\subset M$. The rate function $\mathcal I$ is called good if its lower level sets
%level sets
$\{x\in M\,:\, \mathcal{I}(x) \leq \alpha \}$ are compact for all finite $\alpha\geq 0$.
We say that $(\xi_n)_{n\in\N}$ satisfies a weak LDP with speed $s_n$ and rate function $\mathcal{I}$ if the rightmost upper bound in \eqref{eq:LDPdefinition} is valid only for compact sets $A\subset M$.
\end{df}

%\textcolor{red}{What separates a weak from a full LDP is the so-called exponential tightness of the sequence of random variables (see, e.g., \cite[Lemma 1.2.18]{DZ2010}).

%\begin{proposition}\label{prop:equivalence weak and full LDP}
%Let $(\xi_n)_{n\in\N}$ be a sequence of random elements taking values in a metric space $M$. Suppose that it satisfies a weak LDP with speed $s_n$ %and rate function $\mathcal{I}$. Then $(\xi_n)_{n\in\N}$ satisfies a full LDP if and only if the sequence is exponentially tight, that is, if and %only if for every $C\in(0,\infty)$ there exists a compact set $K_C\subset M$ such that
%$$
%\limsup_{n\to\infty}{1\over s_n}\log \Pro\left[\xi_n\notin K_C\right]<-C\,.
%$$
%In this case, the rate function $\mathcal I$ is good.
%\end{proposition}
%}

In our setting, all LDPs occur effectively on compact spaces, so that the notions of weak  and full LDP fall together.
The following result (see \cite[Theorems 4.1.11 and 4.1.18]{DZ2010}) shows that to prove a weak LDP it is sufficient (and necessary) to consider a  base of the underlying topology on a metric space.
\begin{proposition}\label{prop:basis topology}
Let $\mathcal T$ be a base of the topology in a metric space $M$. Let $(\xi_n)_{n\in\N}$ be a sequence of $M$-valued random elements and assume $s_n\uparrow\infty$.  If for every $w\in M$,
\[
\mathcal I(w)
:=
- \inf_{A\in\mathcal T:\, w\in A} \limsup_{n\to\infty} \frac 1 {s_n} \log \Pro \left[\xi_n \in A\right]
=
- \inf_{A\in\mathcal T:\, w\in A} \liminf_{n\to\infty} \frac 1 {s_n} \log \Pro \left[\xi_n \in A\right],
\]
then $(\xi_n)_{n\in\N}$ satisfies a weak LDP with speed $s_n$ and rate function $\mathcal{I}$. Conversely, if $(\xi_n)_{n\in\N}$ satisfies a weak LDP with speed $s_n$ and rate function $\mathcal{I}$, then the above identities hold.
\end{proposition}
In the present paper we will apply Proposition~\ref{prop:basis topology} to compact spaces $M$ only, so that each closed subset is compact and there is no difference between a weak LDP and a full LDP.

% % % % % % % % % % % % % % % % % % % % % % % % % % % % % % % % % % % % % % %
\subsection{Projective limits and the Dawson--G\"artner theorem}
% % % % % % % % % % % % % % % % % % % % % % % % % % % % % % % % % % % % % % %

Recall that a projective system $(\mathcal Y_j,p_{ij})_{i\leq j}$ consists of Hausdorff topological spaces $\mathcal Y_j$, $j\in\N$, and continuous mappings $p_{ij}:\mathcal Y_j \to\mathcal Y_i$ such that
\[
\forall i\leq j \leq k:\qquad p_{ik} = p_{ij}\circ p_{jk},
\]
where $p_{jj}$, $j\in\N$, is the identity mapping of $\mathcal Y_j$. Then the projective limit $\mathcal X$ of this system is given by
\[
\mathscr X:= \lim_{\longleftarrow} \mathcal Y_j :=\Bigg\{  y=(y_j)_{j\in\N} \in \mathcal Y:= \prod_{j\in\N}\mathcal Y_j\,:\, y_i = p_{ij}(y_j)\,\,\forall i<j \Bigg\},
\]
i.e., it is the subset of the topological product space $\mathcal Y = \prod_{j\in\N}\mathcal Y_j$, consisting of all elements $x=(y_j)_{j\in\N}$ for which $y_i = p_{ij}y_j$ whenever $i\leq j$ equipped with the topology induced by $\mathcal Y$. For $j\in\N$, we shall denote by $p_j:\mathcal X\to\mathcal Y_j$ the canonical projections of $\mathcal X$, which are the restrictions of the coordinate maps from $\mathcal Y$ to $\mathcal Y_j$. Note that each $p_j$, $j\in\N$, is a continuous mapping.

The following theorem is due to Dawson and G\"artner \cite{DG1987}, which we use as formulated in \cite[Theorem 4.6.1]{DZ2010}, but restrict ourselves to the case of sequences of probability measures rather than families.

\begin{proposition}[Dawson--G\"artner theorem]\label{prop:dawson-gaertner}
Let $\mathscr X$ be the projective limit of the system $\mathcal Y_\ell$, $\ell\in\N$, of Hausdorff topological spaces. Assume that $(\mu_n)_{n\in\N}$ is a sequence of probability measures on $\mathscr X$ such that for any $\ell\in\N$, the pushforward-sequence $(\mu_{n}\circ p_\ell^{-1})_{n\in\N}$ satisfies an LDP on $\mathcal Y_\ell$ at speed $n$ with good rate function $\rate_\ell$. Then $(\mu_n)_{n\in\N}$ satisfies an LDP at speed $n$ with the good rate function $\rate:\mathcal X \to [0,+\infty]$
\[
\rate(x) :=\sup_{\ell\in\N} \rate_\ell\big(p_\ell(x)\big).
\]
\end{proposition}

% % % % % % % % % % % % % % % %
\subsection{Uniform distribution on Stiefel manifolds} \label{subsec:stiefel_manifold}
% % % % % % % % % % % % % % % %
For $n,k\in\N$ with $k\leq n$, the Stiefel manifold $\mathbb V_{k,n}$ (over $\R$) is defined as the set of all orthonormal $k$-frames in $\R^n$, i.e., the set of all ordered $k$-tuples of orthonormal vectors in Euclidean space $\R^n$. Alternatively, the Stiefel manifold can be thought of as the set of $k\times n$ matrices and a $k$-frame $u_1,\ldots,u_k$ is represented as a matrix with the $k$ rows $u_1,\ldots,u_k\in\R^n$. Formally, this means that
\[
\mathbb V_{k,n} = \big\{ V \in\R^{k\times n}\,:\, VV^* = \id_{k\times k}  \big\}.
\]

We denote by $\mu_{k,n}$ the Haar probability measure on the Stiefel manifold $\mathbb V_{k,n}$, i.e., the unique probability measure on $\mathbb V_{k,n}$ which is invariant under the two-sided action of the product of the orthogonal groups $\mathcal{O}(n)\times \mathcal O(k)$. We refer to this measure as the uniform distribution on the Stiefel manifold. So, if $V_{k,n}$ is a random matrix uniformly distributed on $\mathbb V_{k,n}$, then $U'V_{k,n} U''$ has the same distribution as $V_{k,n}$ for every $U'\in\mathcal O(n)$ and $U''\in\mathcal O(k)$. In order to generate a random matrix $V_{k,n}$ distributed according to $\mu_{k,n}$ one can proceed as follows: generate the first row $u_1$ of $V_{k,n}$ according to the uniform distribution on the unit sphere $\mathbb S^{n-1}$, then generate the  second row $u_2$ according to the uniform distribution on $\mathbb S^{n-1} \cap u_1^\bot$, then generate the third row according to the uniform distribution on $\mathbb S^{n-1} \cap u_1^\bot \cap u_2^\bot$, and so on.

Various characterizations of the uniform distribution on $\mathbb V_{k,n}$ can be found in~\cite[Section~8.2]{GN2000}. In particular, the next result shows how to generate it using Gaussian random matrices; see~\cite[Theorem~8.2.5]{GN2000}, \cite[Lemma~5]{Khatri1970a} or~\cite[Lemma 2.1]{KPT2019_cube}.

\begin{lemma}\label{lem:uniform distribution Stiefel manifold}
Let $k,n\in\N$ and assume that $n \geq k$. Consider a Gaussian random matrix $G_{k,n}=(g_{ij})_{i,j=1}^{k,n}:\R^n\to\R^k$ with independent standard normal entries. Then the random matrix
$$
V_{k,n} = (G_{k,n}G_{k,n}^*)^{-1/2}G_{k,n}:\R^n\to\R^k
$$
is uniformly distributed on the Stiefel manifold $\mathbb V_{k,n}$.
%Moreover, the radial part $(GG^*)^{1/2}$ and $G^*(GG^*)^{-1/2}$ are independent \textcolor{red}{(J: Not sure the latter is really true! But we don't need it anyway!)}.
\end{lemma}
%\begin{rmk}
%The polar decomposition states that for any linear operator $T:\R^n\to\R^k$, $n\geq k$ there exists a linear isometry $J :\R^k\to\R^n$ (i.e., an %isometric embedding, or, equivalently, an isometric isomorphism onto its image) such that
%\[
%T^* \,=\, J (TT^*)^{1/2}.
%\]
%We will be working in the setting where $T=G_{k,n}$ is a random standard Gaussian matrix.
%In our setting, taking $T=G$, $J$ is the isometric embedding associated with the operator $G$.
%\end{rmk}

% % % % % % % % % % % % % % % % % % % % % % % % % % % % % % % % % % % % % %
\subsection{Inverted matrix variate \texorpdfstring{$t$}{t}-distribution \& Wishart distribution}
% % % % % % % % % % % % % % % % % % % % % % % % % % % % % % % % % % % % % %
Next we recall some facts on matrix-variate distributions from~\cite{GN2000}.
A random $k\times m$-matrix is said to have an inverted matrix variate $t$-distribution~\cite[Definition~4.4.1]{GN2000} with $n$ degrees of freedom if it has density
\begin{equation}\label{eq:density inverted matrix variate t-distr}
A \mapsto  \frac{\Gamma_k(\frac{n+m+k-1}{2})}{\pi^{\frac{mk}{2}}\Gamma_k(\frac{n+k-1}{2})} \det\Big(\id_{k\times k} - AA^*\Big)^{\frac{n-2}{2}},
\qquad
A\in \R^{k\times m},
\qquad \|AA^*\| <1.
\end{equation}
%Note that condition $\|AA^*\| < 1$ holds if and only if $\id_{k\times k}- AA^*$ is positive definite.
Here, $\Gamma_k$ denotes the multivariate gamma function, which is defined as
%\begin{equation}\label{eq:density inverted matrix variate t-distr}
%\R^{p\times m}\ni A \mapsto  \frac{\Gamma_p(\frac{n+m+p-1}{2})}{\pi^{\frac{mp}{2}}\Gamma_p(\frac{n+p-1}{2})} %\det(\Sigma_1)^{-\frac{m}{2}}\det(\Sigma_2)^{-\frac{p}{2}} \det\Big(\id_{p\times p} - %\Sigma_1^{-1}(A-M)\Sigma_2^{-1}(A-M)^*\Big)^{\frac{n-2}{2}},
%\end{equation}
%where $\id_{p\times p} - \Sigma_1^{-1}(A-M)\Sigma_2^{-1}(A-M)^*$ is positive definite. Here, $\Gamma_p$ denotes the multivariate gamma function, which is defined as
\begin{equation}\label{eq:gamma_function_generalized}
\Gamma_k(x) = \pi^{\frac{k(k-1)}{4}}\prod_{i=1}^{k}\Gamma\Big(x-\frac{i-1}{2}\Big),
\end{equation}
where $\Re(x)>\frac{1}{2}(k-1)$; see~\cite[Theorem~1.4.1]{GN2000}.
This distribution, denoted by $\mathrm{IT}_{k,m}(n,0,\id_{k\times k},\id_{m\times m})$, is a special case of a more general family derived by Khatri in~\cite{K1959}, see~\cite[Definition~4.4.1]{GN2000}, with parameters $M = 0\in\R^{k\times m}$, $\Sigma = \id_{k\times k}$, and $\Omega = \id_{m\times m}$.

%In the following, we shall need only the special case when $M=0$ and $\Sigma_1,\Sigma_2$ are identity matrices.  We refer the reader to \cite[Section 4]{GN2000} for more information.

The inverted matrix variate $t$-distribution is related to the Wishart distribution. Recall~\cite[Definition~3.2.1]{GN2000} that a $k\times k$ random symmetric positive definite matrix $S$ is said to have a Wishart distribution with parameters $k$,$n\in\N$ with $n\geq k$, and $\Sigma\in\R^{k\times k}$ a symmetric positive definite matrix, we write $S\sim W_k(n,\Sigma)$, if it has density
\[
A \mapsto \frac{\det(A)^{\frac{n-k-1}{2}}}{2^{\frac{nk}{2}}\Gamma_k(\frac{n}{2})\det(\Sigma)^{\frac{n}{2}}} e^{-\frac{1}{2}\Tr(\Sigma^{-1}A)}.
\]
on the set of symmetric positive definite $k\times k$  matrices. It is known~\cite[Theorem~3.2.2]{GN2000} that if $n\geq k$ and $H_{k,n}$ is a $k\times n$ Gaussian random matrix such that each column $(H_{k,n}(i,j))_{i=1}^k$, $j\in\{1,\dots,n\}$ has $k$-variate Gaussian distribution $\mathcal N_k(0,\Sigma)$ with $k\times k$ covariance matrix $\Sigma$, then
\begin{equation}\label{eq:wishart and gauss}
S=H_{k,n}H_{k,n}^* \sim W_k(n,\Sigma).
\end{equation}

The following result is due to Dickey \cite{D1967} and relates the Wishart distribution to the inverted matrix variate $t$-distribution; see~\cite[Theorem 4.4.1]{GN2000}.

\begin{proposition}\label{prop:relation inverse t and wishart}
Assume that $S\sim W_k(N+k-1,\id_{k\times k})$ and $G_{k,m}$ is a standard Gaussian random $k\times m$-matrix independent from $S$, where $k,m,N\in \N$.  Then,
\[
T:= (S+G_{k,m}G_{k,m}^*)^{-\frac{1}{2}}G_{k,m} \sim \mathrm{IT}_{k,m}(N,0,\id_{k\times k},\id_{m\times m}).
\]
\end{proposition}

% % % % % % % % % % % % % % % % % % % % % % % %
\subsection{Probability on \texorpdfstring{$\ell_p^n$}{lpn}-balls}
% % % % % % % % % % % % % % % % % % % % % % % %

The proof of Theorem~\ref{thm:ldp multidimensional projection unif distr pball} regarding the LDP for projections of uniform distributions on $\ell_p^n$-balls $\B_p^n := \{x\in\R^n \,:\, \|x\|_p\leq 1 \}$ relies on the following probabilistic representation for the uniform distribution on $\B_p^n$, which is due to Schechtman and Zinn \cite{SchechtmanZinn}.

\begin{proposition}\label{prop:schechtman zinn}
Let $n\in\N$ and $p\in[1,\infty)$. Suppose that $Z_1,\ldots,Z_n$ are independent $p$-generalized Gaussian random variables whose distribution has density
$$
f_p(x):= {1\over 2p^{1/p}\Gamma(1+{1\over p})}\,e^{-|x|^p/p}
$$
with respect to the Lebesgue measure on $\R$. Then, if $U$ is a random variable uniformly distributed on $[0,1]$ and independent of $Z:=(Z_1,\ldots,Z_n)\in\R^n$, the random vector $U^{1/n} Z/\|Z\|_p$ is uniformly distributed on $\B_p^n$.
\end{proposition}

%
%\[
%\Gamma_p(x) := \int_{A>0} \text{etr}(-A)\det(A)^{x-\frac{p+1}{2}} \,\dint A,
%\]
%where $\Re(x)>\frac{1}{2}(p-1)$, and the integral is over the space of $p\times p$ symmetric positive definite matrices.

% % % % % % % % % % % % % % % % % % % % % % % %
\subsection{L\'evy-Prokhorov distance}
% % % % % % % % % % % % % % % % % % % % % % % %

Let $(M,d)$ be some metric space and denote by $\mathscr B(M)$ its Borel $\sigma$-field. For $\varepsilon\in(0,\infty)$ and $A\subset M$, let
\[
A_{\varepsilon} := \big\{ x\in M\,:\, \exists a\in A: \, d(a,x)<\varepsilon \big\}
\]
be the $\varepsilon$-neighborhood of the set $A$. If $\mathcal M_1(M)$ denotes the collection of probability measures on $M$, then the L\'evy-Prokhorov distance $\rho_{\textrm{LP}}:\mathcal M_1(M) \times \mathcal M_1(M)\to [0,+\infty)$ is defined by
\[
\rho_{\textrm{LP}}(\mu,\nu) := \inf\Big\{ \varepsilon\in(0,\infty)\,:\, \forall A\in\mathcal B(M):\,\, \mu(A) \leq \nu(A_\varepsilon)+\varepsilon \text{ and }  \nu(A) \leq \mu(A_\varepsilon)+\varepsilon \Big\}.
\]
Note that $(\mathcal M_1(M),\rho_{\textrm{LP}})$ is a Polish space if and only if $(M,d)$ is a Polish space. Moreover, in that case convergence in L\'evy-Prokhorov distance is equivalent to weak convergence. We refer the reader to \cite{B1999} for more information.

% % % % % % % % % % % % % % % % % % % % % % % % % % % % % % % % % % % % % % %
\section{Proof of Theorem \ref{thm:ldp for stiefel manifold} -- the LDP on the Stiefel manifold}\label{sec:proof of stiefel ldp}
% % % % % % % % % % % % % % % % % % % % % % % % % % % % % % % % % % % % % % %

We shall now present the proof of Theorem \ref{thm:ldp for stiefel manifold}, which we split into several parts. The general idea is to reduce the problem to a setting of Gaussian random matrices that allows one to partition $V_{k,n}$ into two blocks, where one can then obtain an LDP for the first block in the space of $k\times \ell$ matrices. This LDP is then lifted to an LDP on the space $[-1,1]^{k\times \infty}$ via a projective limit argument.

% % % % % % % % % % % % % % % % % % % % % % % % % % % %
\subsection{Step 1 -- Reduction to Gaussian random matrices}
% % % % % % % % % % % % % % % % % % % % % % % % % % % %

For $n,k\in\N$ with $n\geq k$, let us denote by $G_{k,n} \in\R^{k\times n}$ a random Gaussian $k\times n$-matrix with i.i.d.\ standard normal entries. Then $G_{k,n}G_{k,n}^*$ is a symmetric $k\times k$ random matrix (having a Wishart distribution, see \eqref{eq:wishart and gauss}) and by Lemma \ref{lem:uniform distribution Stiefel manifold}, the rows of the $k\times n$ matrix
\[
V_{k,n} := (G_{k,n}G_{k,n}^*)^{-\frac{1}{2}}G_{k,n}
\]
form an orthonormal  $k$-frame uniformly distributed on the Stiefel manifold $\mathbb V_{k,n}$. %Let us write the rows of $V_{k,n}$ as
%\[
%V_{k,n} =
%\left(
%\begin{array}{ccc}
% U^{(n)}_{1}\\
% \vdots\\
% U^{(n)}_{k}
%\end{array}
%\right)
%\]
%with $U^{(n)}_i:=(U^{(n)}_{i,1},\dots,U^{(n)}_{i,n})\in\R^n$ denoting the $i$th vector of the $k$-frame,
%$i\in\{1,\dots,k\}$.
Let us fix some $\ell\in\N$ and assume that $n > \ell$; later we consider $n\to\infty$. We shall now decompose each of the random Gaussian matrices $G_{k,n}$ and $G_{k,n}^*$ into two blocks as follows:
\[
G_{k,n} =
\left(
B \,|\, C_n
\right)
\qquad\text{and}\qquad
G_{k,n}^* =
\left(\begin{array}{c}
 % \begin{matrix}
  B^*
 %\end{matrix}
\cr
\hline
 % \begin{matrix}
  C_n^*
 % \end{matrix}
\end{array}\right),
\]
where $B \in\R^{k \times \ell}, C_n\in \R^{k\times(n-\ell)}$ and  $B^*\in\R^{\ell\times k},C_n^*\in\R^{(n-\ell)\times k}$. Therefore, we can write %we can express $G_{k,n}G_{k,n}^*$ as follows,
\[
G_{k,n}G_{k,n}^* = BB^* + C_nC_n^* \in\R^{k\times k}.
\]
In particular, the $k\times n$ matrix $V_{k,n}$ takes the form
\begin{equation}\label{eq:V_k_n_decomposition}
V_{k,n} = \left(
\begin{array}{cc}
(BB^*+C_nC_n^*)^{-\frac{1}{2}} B \quad|\quad (BB^*+C_nC_n^*)^{-\frac{1}{2}} C_n
\end{array}
\right)
\end{equation}
with blocks $(BB^*+C_nC_n^*)^{-1/2} B \in \R^{k\times \ell}$ and $(BB^*+C_nC_n^*)^{-1/2} C_n\in \R^{k\times(n-\ell)}$.

% % % % % % % % % % % % % % % % % % % % % % % % % % % % % % % % % % % % % % %
\subsection{Step 2 -- Inverted matrix variate \texorpdfstring{$t$}{t}-distribution and LDP for \texorpdfstring{$k\times \ell$}{k*l}-submatrices}\label{subsec:inverted_t_and_LDP_for_submatrix}
% % % % % % % % % % % % % % % % % % % % % % % % % % % % % % % % % % % % % % %
Recall that $\ell \in \N$ is fixed. Let us denote the first $k\times \ell$-block of $V_{k,n}$ in~\eqref{eq:V_k_n_decomposition} by
$$
A_{\ell; n} := (BB^*+C_nC_n^*)^{-1/2} B\in \R^{k\times \ell},
\qquad
n > \ell.
$$
It follows directly from Proposition \ref{prop:relation inverse t and wishart} and \eqref{eq:wishart and gauss} (with the choice $N= n-\ell - k + 1$,  $m=\ell$, $S=C_nC_n^*$, $G_{k,m}=B$) that $A_{\ell;n}$ has inverted matrix variate $t$-distribution with $n-\ell-k+1$ degrees of freedom, i.e.,
\[
A_{\ell;n} = (BB^*+C_nC_n^*)^{-1/2} B \sim \mathrm{IT}_{k,\ell}(n-\ell-k+1,0,\id_{k\times k},\id_{\ell\times \ell}),
\]
provided $n\geq \ell+k$.
This means that by \eqref{eq:density inverted matrix variate t-distr} it has density $f_n:[-1,1]^{k\times \ell}\to[0,\infty)$ given by
\begin{equation}\label{eq:V_k_n_submatrix_density}
f_n(A) = \frac{\Gamma_k(\frac{n}{2})}{\pi^{\frac{k\ell}{2}}\Gamma_k(\frac{n-\ell}{2})}\det\Big(\id_{k\times k} - AA^*\Big)^{\frac{n-\ell-k-1}{2}} \mathbb 1_{\{\|AA^*\| <1\}},
\qquad
A\in [-1,1]^{k\times \ell}.
\end{equation}
This fact could have been deduced from the work of Khatri~\cite[Lemma~2]{Khatri1970a} which also describes the joint distribution of both blocks in~\eqref{eq:V_k_n_decomposition}. The same formula can be found in~\cite[Proposition 2.1]{diaconis_etal}, \cite[Proposition 7.3]{eaton} and \cite[Lemma 2.5]{jiang2}. We will now prove the following large deviation result for the block of $V_{k,n}$.

\begin{lemma}\label{lem:ldp_stiefel_kl}
As $n\to\infty$, the sequence
$(A_{\ell;n})_{n\geq\ell+k}$ satisfies an LDP on $[-1,1]^{k\times \ell}$ at speed $n$ with good rate function $\rate_{\ell}:[-1,1]^{k\times \ell} \to [0,+\infty]$ given by
\[
\rate_{\ell}(A):=
\begin{cases}
-\frac{1}{2} \log \det\big(\id_{k\times k} - AA^*\big) & : \|AA^*\|<1 \\
+\infty &: \text{otherwise}.
\end{cases}
\]
\end{lemma}
\begin{proof}
The strategy is to employ Proposition~\ref{prop:basis topology} to obtain a weak LDP, which implies the full LDP since the space $[-1,1]^{k\times \ell}$ is compact. Thus, we need to show that for every matrix $A\in [-1,1]^{k\times \ell}$,
\begin{align}
&\lim_{r\to 0 }\limsup_{n\to\infty} \frac{1}{n} \log \Pro[A_{\ell;n} \in B_r(A)] \leq -\rate_{\ell}(A),\label{eq:proof_kl_upper_bound}\\
&\lim_{r\to 0 }\liminf_{n\to\infty} \frac{1}{n} \log \Pro[A_{\ell;n} \in B_r(A)] \geq -\rate_{\ell}(A),\label{eq:proof_kl_lower_bound}
\end{align}
where $B_r(A)$ denotes the open (Euclidean) ball in $\R^{k\times \ell}$ of radius $r>0$ around $A$.

\medskip
\noindent
\textit{Proof of the upper bound~\eqref{eq:proof_kl_upper_bound}.}
Let $\mathbb B:= \{M\in [-1,1]^{k\times \ell}\,:\,\| MM^* \|<1\}$ and write $\bar{\mathbb B}:= \{M\in [-1,1]^{k\times \ell}\,:\,\| MM^* \|\leq 1\}$ for the closure of $\mathbb B$. Note that $A_{\ell; n} \in \mathbb B$ a.s.
%Although it is not essential  for the following, one can check that $\bar{\mathbb B}\subset [-1,1]^{k\times \ell}$.
If $A\notin \bar{\mathbb B}$, then the probability in~\eqref{eq:proof_kl_upper_bound} vanishes for sufficiently small $r>0$, while at the same time $-\rate_{\ell}(A)= -\infty$, so that~\eqref{eq:proof_kl_upper_bound} holds.  Next, we  take some matrix $A\in \bar{\mathbb B}$.
%Then,
%\[
%\Pro[A_{\ell;n} \in B_r(A)] = \int_{B_r(A)} f_n(D) \,\dint D.
%\]
Then, using that $A_{\ell;n}$ has inverted matrix variate $t$-distribution,
\begin{align*}
\frac{1}{n} \log \Pro[A_{\ell;n} \in B_r(A)]
& =
\frac{1}{n} \log \int_{B_r(A)} f_n(D) \,\dint D \cr
& = \frac{1}{n} \log \int_{B_r(A)\cap \mathbb B} e^{n \frac{1}{n}\log f_n(D)} \,\dint D \cr
& \leq \frac{1}{n} \log \int_{B_r(A)\cap \mathbb B} e^{n \sup_{C\in B_r(A)\cap \mathbb B}\frac{1}{n}\log f_n(C)} \,\dint D \cr
& = \frac{1}{n}\log \vol_{k\times \ell}\big(B_r(A)\cap \mathbb B\big) + \sup_{C\in B_r(A)\cap \mathbb B}\frac{1}{n}\log f_n(C).
\end{align*}
%Since $A\in\B$, we have $\overline{B_r(A)} \subset \B$ if $r>0$ is sufficiently small. In particular, $\det(\id_{k\times k}-CC^*)>0$ on $\overline{B_r(A)}$.
As $n\to\infty$, the first summand on the right-hand side converges to $0$. Let us look at the second one. For any $C\in B_r(A)\cap \mathbb B$ we have
\begin{align*}
\frac{1}{n} \log f_n(C) & = -\frac{k\ell}{2n}\log(\pi) + \frac{1}{n}\log\Bigg(\frac{\Gamma_k(\frac{n}{2})}{\Gamma_k(\frac{n-\ell}{2})}\Bigg) + \frac{n-\ell-k-1}{2n} \log \det\big(\id_{k\times k} - CC^*\big).
\end{align*}
Observe that
\begin{equation}\label{eq:asymptotic multivariate gamma quotient}
\frac 1n \log \Bigg(\frac{\Gamma_k(\frac{n}{2})}{\Gamma_k(\frac{n-\ell}{2})} \Bigg)
=
\frac 1n \log \Bigg(\prod_{j=1}^k \frac{\Gamma(\frac{n}{2}+\frac{1-j}{2})}{\Gamma(\frac{n}{2}+\frac{1-j}{2} - \frac{\ell}{2})}\Bigg)
=
\frac 1n \log \left( (n/2)^{\frac{k\ell}{2}}(1+o(1))\right)
\to 0.
\end{equation}
It follows that
%on $\overline{B_r(A)}$ the uniform convergence
%$$
%\lim_{n\to\infty} \frac{1}{n} \log f_n(C) = -\rate_{\ell}(C).
%$$
%It follows that
\begin{equation}\label{eq:uniform convergence}
\lim_{n\to\infty} \sup_{C\in B_r(A)\cap \mathbb B} \frac{1}{n} \log f_n(C) =  \sup_{C\in B_r(A)\cap \mathbb B} \frac{1}{2}\log \det\big(\id_{k\times k} - CC^*\big) = \sup_{C\in B_r(A)\cap \bar{\mathbb B}} -\rate_{\ell}(C).
\end{equation}
%In particular, because of the lower-semicontinuity of $\rate$ (and thus upper-semicontinuity of $-\rate$), we have that
The function $C\mapsto \det (\id_{k\times k}-CC^*)$ is continuous on $[-1,1]^{k\times \ell}$. Hence, $C\mapsto -\rate_{\ell}(C)$ is a continuous function from $\bar{\mathbb B}$ to the compactified space $[-\infty,0]$. It follows that
\begin{equation}\label{eq: upper semicont of -J}
\lim_{r\to 0}\sup_{C\in B_r(A)\cap \bar{\mathbb B}} -\rate_{\ell}(C)
%= -\lim_{r\to 0}\inf_{C\in B_r(A)\cap \bar{\mathbb B}} \rate_{\ell}(C)
= -\rate_{\ell}(A).
%\geq \limsup_{C\to A}  -\rate_{\ell}(C) = .
\end{equation}
Note that the case when $A$ is on the boundary of $\bar{\mathbb B}$ is included.
Putting the pieces together yields~\eqref{eq:proof_kl_upper_bound}.
%Now, since
%we obtain from \eqref{eq:uniform convergence} that
%\[
%\limsup_{n\to\infty} \frac{1}{n} \log \Pro[A_{\ell;n} \in B_r(A)] \leq \limsup_{n\to\infty} \sup_{C\in B_r(A)}\frac{1}{n} %\log f_n(C) =  \sup_{C\in B_r(A)} -\rate_{\ell}(C).
%\]
%Combining this with \eqref{eq: upper semicont of -J}, we get

\medskip
\noindent
\textit{Proof of the lower bound~\eqref{eq:proof_kl_lower_bound}.} If $A\notin \mathbb B$, then $-\rate_{\ell}(A)=-\infty$ and~\eqref{eq:proof_kl_lower_bound} is trivially satisfied. Hence, we can assume that $A\in  \mathbb B$. For sufficiently small $r>0$, the ball $B_r(A)$ is  contained in $\mathbb B$.  Estimating the density $f_n$ by its infimum over $B_r(A)$, we obtain
$$
\frac{1}{n} \log \Pro[A_{\ell;n} \in B_r(A)]
\geq
\frac{1}{n}\log \vol_{k\times \ell}\big(B_r(A)\big) + \inf_{C\in B_r(A)\cap \mathbb B}\frac{1}{n}\log f_n(C).
$$
Then one can argue as in the proof of the upper bound.
\end{proof}

\subsection{Step 3 -- The LDP for random Stiefel matrix} \label{subsec:dawson gaertner approach}
% % % % % % % % % % % % % % % % % % % % % % % % % % % % % % % % %

We shall now combine the projective limit approach and the Dawson--G\"artner theorem (see Proposition \ref{prop:dawson-gaertner}) with the LDP for the first block of $V_{k,n}$ obtained in Lemma~\ref{lem:ldp_stiefel_kl}. In order to do this, we consider the space $[-1,1]^{k\times
\infty}$ of $k\times \infty$-matrices with entries from $[-1,1]$  (endowed with topology of coordinatewise convergence and Borel $\sigma$-field) as a projective limit of the spaces  $\mathcal Y_\ell:= [-1,1]^{k\times \ell}$, $\ell\in \N$. For $\ell\leq m$, the mapping $p_{\ell m}:\mathcal Y_m \to \mathcal Y_\ell$ is defined as
$$
p_{\ell m }: [-1,1]^{k\times m} \to [-1,1]^{k\times \ell},
\qquad
(a(i,j))_{i,j = 1}^{k,m} \mapsto (a(i,j))_{i,j = 1}^{k,\ell}
$$
i.e., $p_{\ell m}$ just removes the columns $\ell+1$ up to $m$ of a $k\times m$-matrix. % $(a_{ij})_{i,j = 1}^{k,m}$.
%\[
%p_{\ell m }(A) = p_{\ell m } \Big(\sum_{i,j=1}^{k,m} A_{ij}E^{(i,j)} + \sum_{i=1}^{k}\sum_{j=m+1}^\infty 0\cdot E^{(i,j)}\Big)= %\sum_{i,j=1}^{k,\ell} A_{ij}E^{(i,j)} + \sum_{i=1}^{k}\sum_{j=\ell+1}^\infty 0\cdot E^{(i,j)} \in\mathcal Y_\ell,
%\]
%matrices as a subspace of $\R^{k\times \infty}$ (endowed with the topology of coordinatewise convergence); we just fill all columns with column index greater than or equal to $\ell+1$ with zeros. More precisely, define
%\[
%\mathcal Y_\ell := \Big\{ A=(A_{ij})_{i,j=1}^{k,\infty} \in\R^{k\times \infty}\,:\, A_{ij}=0\,\,\,\forall i\in\{1,\dots,k\} \, %\forall j\geq \ell+1  \Big\}.
%\]
%For $s\in\{1,\dots,k\}$ and $t\in\N$, let us write $E^{(s,t)}$ for the matrix in $\R^{k\times \infty}$ with $E^{(s,t)}_{st} = 1$ and $E^{(s,t)}_{ij}=0$ whenever $(s,t)\neq (i,j)$. Then each matrix $A\in\R^{k\times \infty}$ has a representation of the form
%\[
%A = \sum_{i,j=1}^{k,\infty} A_{ij}E^{(i,j)}.
%\]
%For $\ell<m$, the mapping $p_{\ell m}:\mathcal Y_m \to \mathcal Y_\ell$ is defined as
%\[
%p_{\ell m }(A) = p_{\ell m } \Big(\sum_{i,j=1}^{k,m} A_{ij}E^{(i,j)} + \sum_{i=1}^{k}\sum_{j=m+1}^\infty 0\cdot E^{(i,j)}\Big)= %\sum_{i,j=1}^{k,\ell} A_{ij}E^{(i,j)} + \sum_{i=1}^{k}\sum_{j=\ell+1}^\infty 0\cdot E^{(i,j)} \in\mathcal Y_\ell,
%\]
%i.e., $p_{\ell m}$ replaces the columns $\ell+1$ up to $m$ of $A$ by zero columns.
In that setting, $[-1,1]^{k\times \infty}$ is the projective limit of the projective system $(\mathcal Y_\ell,p_{i,\ell})_{i\leq \ell}$. The projection  $p_{\ell}: [-1,1]^{k\times \infty} \to [-1,1]^{k\times \ell}$ maps a  matrix $A\in [-1,1]^{k\times \infty}$ to the matrix $p_\ell(A) =: A_\ell\in [-1,1]^{k\times \ell}$ consisting of the first $\ell$ columns of $A$.

On the compact space $[-1,1]^{k\times \infty}$  we consider, for each $n\in\N$, a probability measure $\mu_n$ which is the distribution of the matrix $V_{k,n}$ which is extended to a $k\times \infty$-matrix by filling the columns $n+1,n+2,\ldots$ with zeroes. For every fixed $\ell\in\N$, Lemma~\ref{lem:ldp_stiefel_kl} shows that $(\mu_n\circ p_{\ell}^{-1})_{n\geq \ell}$ satisfies a large deviation principle on $[-1,1]^{k\times \ell}$ at speed $n$ and with the good rate function $\rate_{\ell}:[-1,1]^{k\times \ell} \to[0,+\infty]$ given by
$$
\rate_\ell (A_\ell)
=
\begin{cases}
-\frac{1}{2} \log \det\big(\id_{k\times k} - A_\ell A_\ell^*\big) & : \|A_\ell A_\ell^*\|<1 \\
+\infty &: \text{otherwise}.
\end{cases}
$$
By the Dawson-Gärtner theorem (Proposition~\ref{prop:dawson-gaertner}), it remains to check that for every matrix $A\in [-1,1]^{k\times \infty}$, we have
\begin{equation}\label{eq:rate_stiefel_supremum}
\sup_{\ell\in \N} \rate_\ell \big(p_{\ell} (A)\big)
 =
\rate(A)
=
  \begin{cases}
-\frac{1}{2} \log \det\big(\id_{k\times k} - A A^*\big) & : A\in \mathcal R_2^{k\times \infty} \text{ and } \|A A^*\|<1 \\
+\infty &: \text{otherwise}.
\end{cases}
\end{equation}
Recall that $A\in \mathcal R_2^{k\times \infty}$ means that all rows of $A$ are in $\ell_2$. If some row of $A$ is not square summable, then for sufficiently large $\ell$, the corresponding row of $A_\ell (= p_\ell(A))$ has Euclidean norm $>1$, implying that the condition $\|A_\ell A_\ell^*\|<1$ fails and hence $\rate_\ell(A_\ell) = +\infty$. From now on, let all rows of $A$ be square summable. Then, we have that $A_\ell A_\ell^* = p_\ell (A) p_\ell(A)^* \to AA^*$ entry-wise. If $\|AA^*\|>1$, then $\|A_\ell A_\ell^*\|>1$, where $\ell$ is sufficiently large, and we again have $\rate_\ell(A_\ell) = \infty$. If $\|AA^*\|=1$, then $\|A_\ell A^*_\ell\|\to 1$ and hence $\rate_\ell(A_\ell)\to\infty$ as $\ell\to\infty$. Finally, if $\|AA^*\|<1$, then
$$
\rate_\ell(A_\ell) = -\frac{1}{2} \log \det\big(\id_{k\times k} - A_\ell A_\ell^*\big) \stackrel{\ell\to\infty}{\longrightarrow}  -\frac{1}{2} \log \det\big(\id_{k\times k} - A A^*\big) =\rate(A).
$$
Given the last display, it is only left to prove that $\rate_\ell(A_\ell)$ is non-decreasing in $\ell$. This is done in the following lemma using the Courant--Fischer min-max principle.
\begin{lemma}\label{lem:monotone_determinant_in_ell}
Let $A\in\mathcal R_2^{k\times \infty}$ such that $\|AA^*\|\leq 1$. Then, the function
$\ell \mapsto \det(\id_{k\times k}- A_{\ell}A_{\ell}^*)$
is non-increasing in $\ell\in \N$.
\end{lemma}
\begin{proof}
 First, we observe that, for any $\ell\in\N$, a direct computation yields that
\[
A_{\ell+1}A_{\ell+1}^* = A_\ell A_{\ell}^* + C_{\ell+1}C_{\ell+1}^*,
%\big(A_{\ell+1}(i,\ell+1)\big)_{i=1}^k \otimes \big(A_{\ell+1}(i,\ell+1)\big)_{i=1}^k,
\]
where $C_{\ell+1}\in\R^k$ denotes the $(\ell+1)$st column of $A$.
%That is, the matrix $A_{\ell+1}A_{\ell+1}^*$ is just the matrix $A_\ell A_{\ell}^*$ plus the $k\times k$ matrix $C_{\ell+1}C_{\ell+1}^*$ formed by taking the tensor product of the $(\ell+1)$st column of $A_{\ell+1}$ with itself.
Hence, for any $z\in\R^{k}$,
\[
\Big\langle z, \big(A_{\ell+1}A_{\ell+1}^* - A_{\ell}A_{\ell}^*\big)z\Big\rangle
=
% \Big\langle z, \Big(\big(A_{\ell+1}(i,\ell+1)\big)_{i=1}^k \otimes \big(A_{\ell+1}(i,\ell+1)\big)_{i=1}^k\Big)z\Big\rangle.
 \Big\langle z, C_{\ell+1}C_{\ell+1}^* z\Big\rangle
=
\Big\langle C_{\ell+1}^* z, C_{\ell+1}^* z\Big\rangle\geq 0.
\]
%since $C_{\ell+1}C_{\ell+1}^*$ is positive semi-definite, which means that $\langle z,C_{\ell+1}C_{\ell+1}^* z \rangle \geq 0$, implying
%\[
%\Big\langle z, \big(A_{\ell+1}A_{\ell+1}^* - A_{\ell}A_{\ell}^*\big)z\Big\rangle \geq 0
%\]
It follows that $A_{\ell+1}A_{\ell+1}^* - A_{\ell}A_{\ell}^*$  is  positive semi-definite. Therefore, using that by the Courant-Fischer min-max principle (see, e.g., \cite{D2013}),
%the variational characterization of eigenvalues (see min-max theorem \cite{}),
we have for any two symmetric and positive semi-definite matrices $M_1$ and $M_2$ that $\det(M_1+M_2)\geq \det(M_1)$ (because the eigenvalues, arranged decreasingly,  satisfy $\lambda_i(M_1+M_2) \geq \lambda_i(M_1)$), we obtain
\[
\det\big(\id_{k\times k}- A_{\ell}A_{\ell}^*\big) = \det\big(\id_{k\times k}-A_{\ell+1}A_{\ell+1}^*+A_{\ell+1}A_{\ell+1}^*- A_{\ell}A_{\ell}^*\big) \geq
\det\big(\id_{k\times k}- A_{\ell+1}A_{\ell+1}^*\big),
\]
which completes the proof.
%Thus, the determinant decreases as $\ell$ increases, which implies that $\rate_\ell(A_\ell)$ is increasing in $\ell$.
\end{proof}

%\subsection{Step 5 -- Convexity of the rate function in Theorem \ref{thm:ldp for stiefel manifold}}

\subsection{Step 4 -- Convexity of the rate function}\label{subsec:convexity_rate_function}
To complete the proof of Theorem~\ref{thm:ldp for stiefel manifold} it remains to establish the convexity of the rate function $\rate:[-1,1]^{k\times\infty} \to [0,+\infty]$ given by~\eqref{eq:rate_stiefel_supremum}.
%\[
%\rate(A) :=
%\begin{cases}
%-\frac{1}{2} \log \det\big(\id_{k\times k} - AA^*\big) &:\, A\in \mathcal R_2^{k\times\infty} \text{ and } \|AA^*\| <1,\\
%+\infty &:\, \text{otherwise},
%\end{cases}
%\]
%where $\mathcal R_2^{k\times \infty}$ is the set of all $k\times \infty$ matrices whose rows belong to the space $\ell_2$ of square summable sequences.
%In order to do this we use the matrix logarithm in form of the identity
%\[
%\log \det(CC^*) = \text{tr} \log (CC^*),
%\]
%for $C\in \mathcal R_2^{k\times\infty} \text{ and } \|CC^*\| <1$.
To this end, we need to prove that for any two matrices $A,B\in \mathcal R_2^{k\times\infty}$ (meaning that the rows of $A$ and $B$ are square summable) with $\|AA^*\|<1$ and $\|BB^*\| <1$ the following inequality holds:
$$
\rate\Big(\frac{A+B}{2}\Big) \leq \frac {\rate(A) + \rate(B)}{2}.
$$
Denote by $\lambda_1,\dots,\lambda_k\in[0,\infty)$ the eigenvalues of the  symmetric and positive semi-definite $k\times k$-matrix  $\Big(\frac{A+B}{2}\Big)\Big(\frac{A+B}{2}\Big)^*$ and by $s_1=\sqrt{\lambda_1},\dots, s_k= \sqrt{\lambda_k}$ the singular values of $\frac{A+B}{2}$. Then we obtain
\begin{align}\label{eq:proof_I_convex_I_AB2}
\rate\Big(\frac{A+B}{2}\Big) & = -\frac{1}{2} \log \det\Bigg(\id_{k\times k} - \Big(\frac{A+B}{2}\Big)\Big(\frac{A+B}{2}\Big)^*\Bigg)
 = -\frac{1}{2} \log \prod_{j=1}^k (1 - \lambda_j)
% = -\frac{1}{2} \log \prod_{j=1}^k (1 - s_j^2)
  = -\frac{1}{2} \sum_{j=1}^k\log(1-s_j^2).
\end{align}
The singular values of $A$ are denoted by $s_j(A)$, $j=1,\ldots, k$. The Schatten $p$-norm of $A$ is defined as $\|A\|_{\mathscr S_{p}} = (\sum_{j=1}^k s_j^p(A))^{1/p}$, $p\geq 1$ (see, e.g., \cite{K1986}). Similar notation is used for the matrix $B$.
Using the Taylor expansion
\[
\log(1-x) = -\sum_{i=1}^\infty \frac{x^i}{i},\qquad |x|<1,
\]
we can rewrite~\eqref{eq:proof_I_convex_I_AB2} in terms of Schatten norms, namely,
\begin{align*}
\rate\Big(\frac{A+B}{2}\Big)
  = -\frac{1}{2} \sum_{j=1}^k\log(1-s_j^2)
   = \frac{1}{2} \sum_{j=1}^k \sum_{i=1}^\infty \frac{s_j^{2i}}{i}
%  & = \frac{1}{2} \sum_{i=1}^\infty \frac{1}{i} \sum_{j=1}^k s_j^{2i}
    = \frac{1}{2} \sum_{i=1}^\infty \frac{1}{i} \Big\|\frac{A+B}{2} \Big\|_{\mathscr S_{2i}}^{2i}
    \leq \frac{1}{2} \sum_{i=1}^\infty \frac{1}{i} \Big(\frac{\|A\|_{\mathscr S_{2i}} +\|B\|_{\mathscr S_{2i}}}{2}\Big)^{2i}.
\end{align*}
Using convexity of $x\mapsto x^{2i}$, we see that
\begin{align*}
\rate\Big(\frac{A+B}{2}\Big)
  & \leq \frac{1}{2} \sum_{i=1}^\infty \frac{1}{i} \Big(\frac{1}{2}\|A\|_{\mathscr S_{2i}}^{2i} + \frac{1}{2}\|B\|_{\mathscr S_{2i}}^{2i}\Big)
   = \frac{1}{4} \sum_{i=1}^\infty \frac{1}{i} \Big(\|A\|_{\mathscr S_{2i}}^{2i} + \|B\|_{\mathscr S_{2i}}^{2i}\Big) \cr
  & = \frac{1}{4} \sum_{i=1}^\infty \frac{1}{i} \sum_{j=1}^k \Big(s_j(A)^{2i} + s_j(B)^{2i}\Big)
%  & = \frac{1}{4} \sum_{i=1}^\infty \frac{1}{i} \sum_{j=1}^k s_j(A)^{2i} + \frac{1}{4} \sum_{i=1}^\infty \frac{1}{i} \sum_{j=1}^k  s_j(B)^{2i} \cr
%  & = \frac{1}{4}  \sum_{j=1}^k \sum_{i=1}^\infty \frac{1}{i} s_j(A)^{2i} + \frac{1}{4}  \sum_{j=1}^k \sum_{i=1}^\infty \frac{1}{i}  s_j(B)^{2i} \cr
   = -\frac{1}{4} \sum_{j=1}^k \log (1-s_j(A)^2) - \frac{1}{4} \sum_{j=1}^k \log(1-s_j(B)^2).
\end{align*}
Using the same argument as in~\eqref{eq:proof_I_convex_I_AB2}, the latter can be rewritten as $\frac 12 \rate(A) + \frac 12 \rate(B)$, thus establishing the convexity of $\rate$. Alternatively, it would be also possible to deduce the convexity of $\rate$ from~\eqref{eq:proof_I_convex_I_AB2} and a classical inequality of Ky Fan (see, e.g., \cite[Theorem 4.1]{GK1969}).

% $\rate:[-1,1]^{k\times\infty} \to [0,\infty]$.

%\begin{rmk}
%For the sake of completeness, we include a direct proof of the convexity in the appendix.?????
%\end{rmk}
% % % % % % % % % % % % % % % % % % % % % % % % % % % % % % % % % % % % % % %
\section{Proof of Theorem \ref{thm:ldp orthogonal group} -- the LDP on the orthogonal group} \label{sec: Proof of Theorem B}
% % % % % % % % % % % % % % % % % % % % % % % % % % % % % % % % % % % % % % %

We shall now present the proof of the large deviation principle for random matrices chosen uniformly at random from the orthogonal group with respect to the Haar probability measure. Several elements from the proof of Theorem \ref{thm:ldp for stiefel manifold} in Section \ref{sec:proof of stiefel ldp} will enter. Moreover, due to compactness of the underlying space $[-1,1]^{\N\times \N}$ it will be enough to establish a weak large deviation principle. The latter will be achieved through Proposition \ref{prop:basis topology} on a base of the topology.

% % % % % % % % % % % % % % % % % % % % % % % % % % % % %
\subsection{Step 1 -- Reduction to row-submatrices}
% % % % % % % % % % % % % % % % % % % % % % % % % % % % %

Let $n\in\N$ and assume that $O^{(n)}$ is chosen uniformly at random from the orthogonal group $\mathcal O(n)$ with respect to the Haar probability measure. Note that if $k\in\N$ with $k\leq n$, then the first $k$ rows of $O^{(n)}$ are uniformly distributed on the Stiefel manifold $\mathbb V_{k,n}$; this will allow us to use some results obtained in Section \ref{sec:proof of stiefel ldp}. Let $M\in[-1,1]^{\N\times \N}$ be an infinite matrix with entries bounded in modulus by $1$. For $k\in\N$, $\ell\in\N$, denote by $M_{k,\ell}$ the $k\times\ell$ upper left corner submatrix of $M$ defined by setting
\[
M_{k,\ell}(i,j)
:=
%\begin{cases}
M(i,j), \qquad \text{ for } 1\leq i \leq k,\, 1\leq j \leq \ell.
%0 &: \text{otherwise}.
%\end{cases}
\]
Similarly, for $k\in \N$ and $\ell = \infty$, we define $M_{k,\infty}$ to be the $k\times \infty$-matrix consisting of the first $k$ rows of $M$.
Note that a base of the product topology on $[-1,1]^{\N\times \N}$ is given by the collection of balls of the form
\[
B_{k,\ell}(M,r)
:= \Big\{ \widetilde{M}\in[-1,1]^{\N\times \N}\,:\, \big\|\widetilde{M}_{k,\ell} - M_{k,\ell}\big\|_2 \leq r \Big\},\qquad M\in [-1,1]^{\N\times \N}, r\in(0,\infty), k,\ell\in\N,
\]
where $\|\cdot\|_2$ is the Euclidean norm on $\R^{k\times \ell}$.
Lemma~\ref{lem:ldp_stiefel_kl} shows that for every fixed $k,\ell\in\N$, the sequence $O^{(n)}_{k,\ell}$,  with $n\geq \max\{k,\ell\}$, satisfies a large deviation principle on $[-1,1]^{k\times \ell}$ at speed $n$ and with good rate function
$$
\rate_{k,\ell} (M_{k,\ell})
=
\begin{cases}
-\frac{1}{2} \log \det\big(\id_{k\times k} - M_{k,\ell} M_{k,\ell}^*\big) & : \|M_{k,\ell} M_{k,\ell}^*\|<1 \\
+\infty &: \text{otherwise}.
\end{cases}
$$
Moreover, the monotonicity of $\rate_{k,\ell}$ in the parameter $\ell$, which we established in Lemma~\ref{lem:monotone_determinant_in_ell}, together with the Dawson-G\"artner argument (see Proposition \ref{prop:dawson-gaertner}), then guarantee (similar to the argument presented in Subsection \ref{subsec:dawson gaertner approach}) that for every fixed $k\in \N$, the sequence $O^{(n)}_{k,\infty}$, $n\geq k$, satisfies an LDP at speed $n$ with a good rate function
$$
\rate_{k,\infty}(M_{k,\infty})
:=
\sup_{\ell\in \N} \rate_{k,\ell} \big(M_{k,\ell}\big)
=
  \begin{cases}
-\frac{1}{2} \log \det\big(\id_{k\times k} - M_{k,\infty} M_{k,\infty}^*\big) & : M_{k,\infty}\in \mathcal R_2^{k\times \infty},  \|M_{k,\infty} M_{k,\infty}^*\|<1 \\
+\infty &: \text{otherwise},
\end{cases}
$$
where $\mathcal R_2^{k\times \infty}$ denotes the set of infinite $k\times\infty$-matrices for which each of the $k$ rows is in $\ell_2$.

% % % % % % % % % % % % % % % % % % % % % % % % % %
\subsection{Step 2 -- Properties of submatrices}
% % % % % % % % % % % % % % % % % % % % % % % % % %

While monotonicity for $\rate_{k,\ell}$ in the parameter $\ell$ has been established in Lemma~\ref{lem:monotone_determinant_in_ell}, we shall now show the monotonicity of $\rate_{k,\infty}$ in the parameter $k$.
The argument is based on Cauchy's interlacing theorem for eigenvalues.

\begin{lemma}\label{lem:log-determinant decreasing in k}
Consider a matrix $M\in \R^{\N\times \N}$ with square summable rows and $\|M_{k,\infty}M_{k,\infty}^*\|\leq 1$ for all $k\in \N$.  Then the mapping $k\mapsto \det(\id_{k\times k} - M_{k,\infty}M_{k,\infty}^*)$
is well-defined and non-increasing in $k\in \N$.
\end{lemma}
\begin{proof}
First observe that $M_{k,\infty}M_{k,\infty}^*$, which is the Gram matrix of the first $k$ rows of $M$, is a $k\times k$ submatrix of $M_{k+1,\infty}M_{k+1,\infty}^*$, which is the Gram matrix of the first $k+1$ rows of $M$.
It follows from Cauchy's interlacing theorem (see, e.g., \cite[Theorem 4.3.17, p. 242]{HJ1994}) that the eigenvalues $\lambda_i$, which we assume in non-increasing order, are interlacing, i.e.,
\begin{align*}
\lambda_1(M_{k+1,\infty}M_{k+1,\infty}^*) & \geq \lambda_1(M_{k,\infty}M_{k,\infty}^*) \geq \lambda_2(M_{k+1,\infty}M_{k+1,\infty}^*) \geq \lambda_2(M_{k,\infty}M_{k,\infty}^*) \\
& \geq \dots \geq \lambda_k(M_{k,\infty}M_{k,\infty}^*) \geq \lambda_{k+1}(M_{k+1,\infty}M_{k+1,\infty}^*) \geq 0.
\end{align*}
Let us note here that in reference \cite{HJ1994} the order of eigenvalues is opposite to ours. Since
 $\lambda_1(M_{\ell,\infty}M^*_{\ell,\infty})=\|M_{\ell,\infty}M^*_{\ell,\infty}\| \leq  1$
for all $\ell\in\N$, we have
\begin{align*}
0\leq  1-\lambda_1(M_{k+1,\infty}M_{k+1,\infty}^*) & \leq 1-\lambda_1(M_{k,\infty}M_{k,\infty}^*) \\
%\geq
% \lambda_2(M_{k+1,\infty}M_{k+1,\infty}^*) \geq \lambda_2(M_{k,\infty}M_{k,\infty}^*) \\
& \leq \dots \leq 1-\lambda_k(M_{k,\infty}M_{k,\infty}^*) \leq 1-\lambda_{k+1}(M_{k+1,\infty}M_{k+1,\infty}^*) \leq 1.
\end{align*}
Therefore, we obtain that
\[
\det\big(\id_{k\times k} -  M_{k,\infty}M_{k,\infty}^*\big) \geq \det\big(\id_{(k+1)\times (k+1)} -  M_{k+1,\infty}M_{k+1,\infty}^*\big).
\]
This completes the proof.
\end{proof}

% % % % % % % % % % % % % % % % % % % % % % % % % % % % %
\subsection{Step 3 -- The weak large deviation principle}
% % % % % % % % % % % % % % % % % % % % % % % % % % % % %

We are now ready to extract from Steps 1 and 2, in combination with Proposition \ref{prop:basis topology}, the weak large deviation principle for the sequence $O^{(n)}$, $n\in\N$, of random Haar orthogonal matrices in $\mathcal O(n)$, $n\in\N$. Recall that each $O^{(n)}$ takes values in the space $[-1,1]^{\N\times \N}$ of infinite matrices and that, as explained in Step 1, a base of the product topology on $[-1,1]^{\N\times \N}$ is given by the collection of balls of the form
\[
B_{k,\ell}(M,r)
:= \Big\{ \widetilde{M}\in[-1,1]^{\N\times \N}\,:\, \big\|\widetilde{M}_{k,\ell} - M_{k,\ell}\big\|_2 \leq r \Big\},\qquad M\in [-1,1]^{\N\times \N}, r\in(0,\infty), k,\ell\in\N.
\]
By Proposition \ref{prop:basis topology} it is enough to work on this base of the topology. So let $M\in[-1,1]^{\N\times \N}$, $r\in(0,
\infty)$, and $k,\ell\in \N$. In the following, we let $n\to\infty$, so that $n\geq \max\{k,\ell\}$ may be assumed. By the ``converse'' part of Proposition \ref{prop:basis topology} and Step~1, for each fixed $k\in\N$,
\begin{align}
%\rate_{k,\infty}(M_{k,\infty})
%&=
&-\inf_{\ell\in\N}  \inf_{r\in(0,\infty)}\limsup_{n\to\infty} \frac{1}{n} \log \Pro\Big[O^{(n)}_{k,\ell} \in B_{k,\ell}(M,r)\Big] \notag \\
& =
-\inf_{\ell\in\N} \inf_{r\in(0,\infty)} \liminf_{n\to\infty} \frac{1}{n} \log \Pro\Big[O^{(n)}_{k,\ell} \in B_{k,\ell}(M,r)\Big] \notag \\
& = \sup_{\ell\in\N }\rate_{k,\ell}(M_{k,\ell})\notag\\
& = \rate_{k,\infty}(M_{k,\infty}) \notag\\
& =
  \begin{cases}
-\frac{1}{2} \log \det\big(\id_{k\times k} - M_{k,\infty} M_{k,\infty}^*\big) & : M_{k,\infty}\in \mathcal R_2^{k\times \infty} \text{ and } \|M_{k,\infty} M_{k,\infty}^*\|<1 \\
+\infty &: \text{otherwise}.
\end{cases}
\label{eq:I_k_infty_tech}
\end{align}
Let first $M$ be such that $M_{k,\infty}\in \mathcal R_2^{k\times \infty}$ and $\|M_{k,\infty} M_{k,\infty}^*\|<1$ for all $k\in\N$.
Now it follows from  Lemma~\ref{lem:log-determinant decreasing in k} that the log-determinant is non-increasing in $k$. This, in combination with the previous display, yields
\begin{align*}
 -\inf_{r\in(0,\infty)}\inf_{k, \ell\in\N} \limsup_{n\to\infty} \frac{1}{n} \log \Pro\Big[O^{(n)}_{k,\ell} \in B_{k,\ell}(M,r)\Big]
& =
-\inf_{r\in(0,\infty)}\inf_{k,\ell\in\N} \liminf_{n\to\infty} \frac{1}{n} \log \Pro\Big[O^{(n)}_{k,\ell} \in B_{k,\ell}(M,r)\Big] \\
& =
-\frac{1}{2} \lim_{k\to\infty} \log \det\big(\id_{k\times k} - M_{k,\infty} M_{k,\infty}^*\big)
\end{align*}
%for any matrix $M\in[-1,1]^{\N\times \N}$ whose rows are all in $\ell_2$ and for which $\|M_{k,\infty}M^*_{k,\infty}\|<1$ for all %$k\in\N$.
For all matrices $M\in[-1,1]^{\N\times \N}$ violating either one of the assumptions $M_{k,\infty}\in \mathcal R_2^{k\times \infty}$ or $\|M_{k,\infty}M^*_{k,\infty}\|<1$ for all $k\in \N$, the expressions on the left-hand side of the previous display are equal to $+\infty$ by~\eqref{eq:I_k_infty_tech}.  By Proposition~\ref{prop:basis topology}, we obtain that $O^{(n)}$ satisfies a weak (and thus full due to compactness) LDP at speed $n$ with good rate function $\rate_1: [-1,1]^{\N \times \N}\to [0,+\infty]$ given by
\begin{align}\label{eq:I_preliminary}
\rate_1(M)
& :=
  \begin{cases}
-\frac{1}{2} \lim_{k\to\infty} \log \det\big(\id_{k\times k} - M_{k,\infty} M_{k,\infty}^*\big) & : \text{rows of $M$ are in $\ell_2$, $\|M_{k,\infty} M_{k,\infty}^*\|<1$  $\forall\,k\in\N$} \\
+\infty &: \text{otherwise}.
\end{cases}
\end{align}
This function is convex as a pointwise limit of convex functions; see Subsection~\ref{subsec:convexity_rate_function}.

% % % % % % % % % % % % % % % % % % % % % % % % % % % % % % % % % % %
\subsection{Step 4 -- A reformulation via Hilbert--Schmidt class operators}
% % % % % % % % % % % % % % % % % % % % % % % % % % % % % % % % % % %

The LDP for $O^{(n)}$, $n\in\N$, which we obtained in the previous subsection, can be formulated in a more elegant way, namely in terms of Hilbert--Schmidt operators as presented in Theorem \ref{thm:ldp orthogonal group}. In particular, in this subsection we complete the proof of this theorem. For general background on classes of compact operators as well as traces and determinants of operators, we refer to the monographs \cite{GGK2000,K1986,P1980,P1987}. Our aim is to show that the information function $\rate_1(M)$, originally defined by~\eqref{eq:I_preliminary},  equals
\begin{align*}
\rate(M)
& :=
  \begin{cases}
-\frac{1}{2} \log \det\big(\id_{\infty \times \infty} - M M^*\big) & : \text{$M\in\mathcal S_2\,$ and $\,\|MM^*\|<1$}\\
+\infty &: \text{otherwise}.
\end{cases}
\end{align*}
We start by assuming that  $M=(m_{ij})_{i,j\in\N}\in\mathcal S_2$, i.e., $M$ is a Hilbert--Schmidt operator meaning that $\sum_{i,j\in\N}m_{ij}^2 <\infty$.   Recall that $M\in\mathcal S_2$ is equivalent to $MM^*$ being a trace class (or nuclear) operator, i.e., $MM^*\in\mathcal S_1$.  For $k\in\N$, we view $M_{k,\infty}$ as an $\infty\times \infty$-matrix by filling up this $k\times \infty$-matrix with zero rows.
Our first task is to show that
\begin{equation}\label{eq:log_det_M_k_converges}
\lim_{k\to\infty}\det\big(\id_{\infty\times \infty} - M_{k,\infty}M_{k,\infty}^*\big) = \det(\id_{\infty\times \infty} - MM^*).
\end{equation}
It is known that $\det(\id_{\infty \times \infty} - M M^*) = \det(\id_{\infty \times \infty} - M^*M)$; see Corollary 2.2 on p.~51 and use Lidskii's theorem~\cite[Theorem~6.1 on p.~63]{GGK2000}. Hence, \eqref{eq:log_det_M_k_converges} is equivalent to
\begin{equation}\label{eq:log_det_M_k_converges_*}
\lim_{k\to\infty}\det\big(\id_{\infty\times \infty} - M_{k,\infty}^*M_{k,\infty}\big)
=
\det(\id_{\infty\times \infty} - M^*M).
\end{equation}
We claim that $M^* M - M_{k,\infty}^*M_{k,\infty}$ is positive semi-definite. Indeed, for every vector $x\in \ell_2$, we have
\begin{align*}
\big\langle (M^* M -  M_{k,\infty}^*M_{k,\infty})x, x \big\rangle
&=
\big\langle Mx, Mx \big\rangle -  \big\langle M_{k,\infty} x, M_{k,\infty}x\big\rangle\\
&=
\| Mx\|^2 - \| M_{k,\infty} x\|^2
=
\sum_{j = 1}^\infty \langle R_j(M), x\rangle^2
-
\sum_{j = 1}^k \langle R_j(M), x \rangle^2
\geq
0.
\end{align*}
We shall now show that the  operator $M_{k,\infty}^*M_{k,\infty}$
converges to $M^*M$ in $\mathcal S_1$. Indeed, since the $\mathcal S_1$-norm
of a positive semi-definite matrix coincides with its trace, we have
\begin{align}\label{eq:S_1_convergence_M_k_to_M}
\|M^*M - M_{k,\infty}^*M_{k,\infty} \|_{\mathcal S_1}
=
\sum_{i\in\N} (M^*M - M_{k,\infty}^*M_{k,\infty})_{ii}
=
\sum_{i,j=1}^\infty m_{ij}^2 - \sum_{i=1}^k \sum_{j=1}^\infty m_{ij}^2
\stackrel{k\to\infty}{\longrightarrow} 0,
\end{align}
because $M$ is a Hilbert--Schmidt operator. For trace class operators $A_1,A_2,\ldots$ and $A$ it is known that $A_n \to A$ in the $\mathcal S_1$-norm implies that $\det (1-A_n) \to \det (1-A)$; see Equation (5.14) on p.~60 in~\cite{GGK2000}. In our setting, this yields~\eqref{eq:log_det_M_k_converges_*} and thus~\eqref{eq:log_det_M_k_converges}.

Now we are ready to prove that $\rate_1(M) = \rate(M)$ for all $M\in [-1,1]^{\N \times \N}$.
First, we assume that $\|MM^*\|<1$. Denote by $P_k$ the $\infty\times \infty$-matrix obtained from the $k\times k$-identity matrix by filling it up with zeroes. Then, we have $M_{k,\infty} = P_k M$ and hence $\|M_{k,\infty} M_{k,\infty}^*\| = \|P_k M M^* P_k\| \leq \|MM^*\|<1$ for all $k\in \N$ .
Hence, in this case $\rate_1(M) = \rate(M)$ by~\eqref{eq:log_det_M_k_converges}.
%For completeness, let us mention that $\det(\id_{\infty \times \infty} - M M^*)=\det(\id_{\infty \times \infty} - M^*M)$ by~\cite[p.~61]{GGK2000}.

Next, assume that  $M=(m_{ij})_{i,j\in\N}\in\mathcal S_2$ but $\|MM^*\|\geq 1$. Then, $\rate(M)= +\infty$ and we need to show that $\rate_1(M)= \infty$. Let first $\|MM^*\| > 1$. As $k\to\infty$, we have $M_{k,\infty}^*M_{k,\infty} \to M^*M$ in the operator norm (because the convergence holds even in the Schatten $\mathcal S_1$-norm, as we have shown in~\eqref{eq:S_1_convergence_M_k_to_M} above). Hence,
$$
\|M_{k,\infty}M_{k,\infty}^*\| = \|M_{k,\infty}\|^2 = \|M_{k,\infty}^*M_{k,\infty}\|
\stackrel{k\to\infty}{\longrightarrow}
\|M^*M\| =\|M\|^2 = \|MM^*\|>1.
$$
Consequently $\|M_{k,\infty}M_{k,\infty}^*\|>1$ for sufficiently large $k$ and we have $\rate_1(M)= +\infty = \rate(M)$.

Next assume that $M\in\mathcal S_2$ and $\|MM^*\|= 1$. Then, $\rate(M) = +\infty$ and we want to show that $\rate_1(M)= +\infty$.
We may assume that $\|M_{k,\infty}M_{k,\infty}^*\|<1$
for all $k\in \N$ since otherwise $\rate_1(M)= +\infty = \rate(M)$ by definition.
Denote by $(\lambda_i^{(k)})_{i=1}^k$ the eigenvalues of the $k\times k$-matrix $M_{k,\infty}M_{k,\infty}^*$ ordered decreasingly. We have that $0\leq \lambda_i^{(k)}<1$ for all $i\in\{1,\dots,k\}$.
Since $\id_{k\times k}-M_{k,\infty}M_{k,\infty}^*$ has eigenvalues $\big(1-\lambda_i^{(k)}\big)_{i=1}^k$, we have
\begin{align*}
\det\big(\id_{k\times k} - M_{k,\infty}M_{k,\infty}^*\big) & = \prod_{i=1}^k (1-\lambda_i^{(k)}) \leq 1-\lambda_1^{(k)}
\stackrel{k\to\infty}{\longrightarrow} 0
\end{align*}
since $\lambda_1^{(k)} = \|M_{k,\infty}M_{k,\infty}^*\| \to \|MM^*\| = 1$ as we have shown above. Together with~\eqref{eq:I_preliminary} this implies that $\rate_1(M) = +\infty$, as claimed.

Now, assume that $M\notin \mathcal S_2$. Then, $\rate(M)= +\infty$ and we need to show that $\rate_1(M)= +\infty$.   Assuming that the rows of $M$ are in
$\ell_2$ and $\|M_{k,\infty} M_{k,\infty}^*\|<1$  for all $k\in\N$ (since otherwise $\rate_1(M) = +\infty$)
it suffices to  show that
\[
\lim_{k\to\infty} \det\big(\id_{k\times k} - M_{k,\infty}M_{k,\infty}^*\big) = 0.
\]
Using the same notation for eigenvalues as before and the estimate $1-\lambda_i^{(k)} \leq e^{-\lambda_i^{(k)}}$, we obtain
\begin{align*}
\det\big(\id_{k\times k} - M_{k,\infty}M_{k,\infty}^*\big)
& =
\prod_{i=1}^k (1-\lambda_i^{(k)}) \leq \prod_{i=1}^k e^{-\lambda_i^{(k)}}
= e^{-\sum_{i=1}^k \lambda_i^{(k)}}
\\
&= e^{-\Tr(M_{k,\infty}M_{k,\infty}^*)}
 = e^{-\sum_{i=1}^k \sum_{j=1}^\infty m_{ij}^2}
%& = e^{-\|M_{k,\infty}\|_{\mathcal S_2}}
 \stackrel{k\to\infty}{\longrightarrow} 0,
\end{align*}
because if $M\notin\mathcal S_2$, then $\sum_{i=1}^\infty \sum_{j=1}^\infty m_{ij}^2=+\infty$.  This shows that $\rate_1(M) = +\infty$.
%\[
%-\frac{1}{2}\log \det\big(\id_{k\times k} - M_{k,\infty}M_{k,\infty}^*\big)\stackrel{k\to\infty}{\longrightarrow} \infty.
%\]

Wrapping up, this means that $\rate_1= \rate$ and hence
$O^{(n)}$, $n\in\N$, satisfies a full LDP at speed $n$ with good rate function $\rate:[-1,1]^{\N\times \N}\to[0,+\infty]$ as defined above. This completes the proof of Theorem \ref{thm:ldp orthogonal group}.

%defined as
%\begin{align*}
%\rate(M)
%& :=
%  \begin{cases}
%-\frac{1}{2} \log \det\big(\id - M M^*\big) & : \text{$M\in\mathcal S_2\,$ and $\,\|MM^*\|<1$}\\
%\infty &: \text{otherwise}.
%\end{cases}
%\end{align*}
%good rate function $\rate:[-1,1]^{\N\times \N}\to[0,\infty]$ defined as????

% % % % % % % % % % % % % % % % % % % % % % % % % % % % % % % % % % % % % % %
\section{Proof of Theorem \ref{thm:ldp multidimensional projection unif distr pball} -- the LDP for \texorpdfstring{$k$}{k}-dimensional projections of the uniform distribution on \texorpdfstring{$\ell_p^n$}{lpn}-balls} \label{sec:k-dimensional projections lp}
% % % % % % % % % % % % % % % % % % % % % % % % % % % % % % % % % % % % % % %

In this section we will prove Theorem \ref{thm:ldp multidimensional projection unif distr pball} on random projections of uniform distributions on $\ell_p^n$-balls. To be more precise, we shall reduce it to Theorem~\ref{thm:ldp multidimensional projection_product_measures} on random projections of product measures which we then prove in Section~\ref{sec:proof_theorem_D}.
The proof requires some preparation. Fix $k\in\N$. For $n\in\N$, we recall that for a random element $V_{k,n} : \R^n \to\R^k$ in the Stiefel manifold $\mathbb V_{k,n}$, chosen with respect to the uniform distribution $\mu_{k,n}$, and a random vector $X^{(n)}$ uniformly distributed on $n^{1/p}\B_p^n$ (which is independent of $V_{k,n}$), we are interested in the large deviation behavior of the random probability measure
\[
\mu_{V_{k,n}}(A) := \Pro\Big[\big(\langle R_1(V_{k,n}),X^{(n)} \rangle,\dots,\langle R_k(V_{k,n}),X^{(n)} \rangle\big) \in A \Big], \qquad A\subset \R^k \text{ Borel},
\]
where $R_1(V_{k,n}),\dots,R_k(V_{k,n})\in\R^n$ are the rows of $V_{k,n}$.
In this section we show that it is enough to prove an LDP for simpler variants of the measures $\mu_{V_{k,n}}$, denoted by $\widetilde{\mu}_{V_{k,n}}$, in which $X^{(n)}$ is replaced by a random vector $(Z_1,\ldots, Z_n)$ with independent $p$-Gaussian components. This can be done because those two measures are, as we will see, close in the L\'evy--Prokhorov distance for any realization.

% % % % % % % % % % % % % % % % % % % % % % % % % % % % % % % % % % % % % % % %
\subsection{Step 1 -- Reduction to projections of product measures}
% % % % % % % % % % % % % % % % % % % % % % % % % % % % % % % % % % % % % % % %

It follows from the result of Schechtman and Zinn (see Proposition \ref{prop:schechtman zinn}) that if $X^{(n)}$ is uniformly distributed on $n^{1/p}\B_p^n$, we have
\begin{align}\label{eq:probab representation}
X^{(n)} &\stackrel{\dint}{=} U^{1/n} n^{1/p}\frac{Z}{\|Z\|_p},
\end{align}
where $U\sim\Uni([0,1])$ is independent of $Z:=(Z_1,\dots,Z_n)$, which is a random vector whose coordinates $Z_1,\dots,Z_n$ are independent $p$-generalized Gaussian random variables, i.e., they have Lebesgue density
\[
f_p(x):= {1\over 2p^{1/p}\Gamma(1+{1\over p})}\,e^{-|x|^p/p}, \qquad x\in\R.
\]
%The normalization $n^{1/p}\B_p^n$ is chosen in such a way that the ball essentially has volume $e^{O(n)}$. {\color{red}???}

Now, using the probabilistic representation in \eqref{eq:probab representation}, we obtain that for every Stiefel matrix $V\in \mathbb V_{k,n}$ the distribution $\mu_V$ of $VX^{(n)}$ can be represented as
\begin{equation}\label{eq:mu_V_def}
\mu_{V}(A)
:=
\Pro\Bigg[ VX^{(n)} \in A \Bigg]
%&=
%\Pro\Bigg[\Big(\Big\langle R_1(V),U^{1/n} \frac{n^{1/p}Z}{\|Z\|_p} \Big\rangle,\dots,\Big\langle R_k(V),U^{1/n} \frac{n^{1/p}Z}{\|Z\|_p} \Big\rangle\Big) \in \cdot \Bigg] \cr
=
\Pro\Bigg[\sum_{j=1}^n U^{1/n}n^{1/p}\frac{Z_j}{\|Z\|_p} C_j(V) \in A \Bigg],
\qquad A\subset \R^k \text{ Borel},
\end{equation}
where $C_j(V):= (V(i,j))_{i=1}^k$ is the $j$th column of the $k\times n$ matrix $V$. We shall now compare $\mu_V$ to the following simplified variant:
\begin{equation}\label{eq:mu_V_tilde_def}
\widetilde{\mu}_{V}(A) := \Pro\Bigg[\sum_{j=1}^nZ_j C_j(V) \in A \Bigg],\qquad A\subset \R^k \text{ Borel}.
\end{equation}
Recall that $\mathcal M_1(\R^k)$ denotes the space of probability measures on $\R^k$. The next lemma shows that the L\'evy--Prokhorov distance $\rho_{\textrm{LP}}$ on $\mathcal M_1(\R^k)$ of the measures $\mu_{V}$ and $\widetilde{\mu}_{V}$ converges to $0$ uniformly over all $V\in \mathbb V_{k,n}$, as $n\to\infty$.

\begin{lemma}\label{lem:levy-prokhorov lemma}
Let $k\in\N$ fixed. Then,
\[
\lim_{n\to\infty} \sup_{V \in\mathbb V_{k,n}}\rho_{\textrm{LP}}\big(\widetilde{\mu}_{V}, \mu_{V}\big) = 0.
\]
\end{lemma}
\begin{proof}
For $n\in\N$, let $A\subset\R^k$ be a Borel set, $V\in\mathbb V_{k,n}$, and $\varepsilon\in(0,\infty)$. Then, denoting by $C_1=C_1(V),\dots,C_n=C_n(V)\in\R^k$ the columns of $V$, we obtain that
\begin{align}\label{eq:levy prokhorov bound}
\widetilde{\mu}_{V}(A) & = \Pro\Bigg[\sum_{j=1}^nZ_j C_j \in A \Bigg] \cr
& \leq \Pro\Bigg[\sum_{j=1}^n U^{1/n}n^{1/p}\frac{Z_j}{\|Z\|_p} C_j \in A_{\varepsilon} \Bigg] + \Pro\Bigg[\Big\| \sum_{j=1}^nZ_jC_j - \sum_{j=1}^nZ_jC_j U^{1/n}\frac{n^{1/p}}{\|Z\|_p} \Big\|_2 \geq \varepsilon\Bigg] \cr
& \leq \mu_V(A) + \Pro\Bigg[\Big\| \sum_{j=1}^nZ_jC_j - \sum_{j=1}^nZ_jC_j U^{1/n}\frac{n^{1/p}}{\|Z\|_p} \Big\|_2 \geq \varepsilon\Bigg].
\end{align}
It is thus left to prove that the second summand on the right-hand side is bounded above by $\varepsilon$ for sufficiently large $n$. From Markov's inequality, we obtain
\begin{align*}
\Pro\Bigg[\Big\| \sum_{j=1}^nZ_jC_j - \sum_{j=1}^nZ_jC_j U^{1/n}\frac{n^{1/p}}{\|Z\|_p} \Big\|_2 \geq \varepsilon\Bigg] & \leq \varepsilon^{-1} \E\Bigg[\Big\| \sum_{j=1}^nZ_jC_j - \sum_{j=1}^nZ_jC_j U^{1/n}\frac{n^{1/p}}{\|Z\|_p} \Big\|_2\Bigg].
\end{align*}
Now, the Cauchy-Schwarz inequality shows that
\begin{align}\label{eq:estiamte expectation via cs}
   \E\Bigg[\Big\| \sum_{j=1}^nZ_jC_j - \sum_{j=1}^nZ_jC_j U^{1/n}\frac{n^{1/p}}{\|Z\|_p} \Big\|_2\Bigg]
   & =
   \E\Bigg[\Big\|\sum_{j=1}^n Z_jC_j\Big\|_2\cdot\Big| 1 - U^{1/n}\frac{n^{1/p}}{\|Z\|_p}  \Big|\Bigg] \cr
   & \leq \sqrt{\E\Bigg[\Big\|\sum_{j=1}^n Z_jC_j\Big\|_2^2\Bigg]} \cdot  \sqrt{\E\Bigg[\Big| 1 - U^{1/n}\frac{n^{1/p}}{\|Z\|_p}  \Big|^2\Bigg]}.
\end{align}
For the first factor in \eqref{eq:estiamte expectation via cs}, we have
$$
\E\Bigg[\Big\|\sum_{j=1}^n Z_jC_j\Big\|_2^2\Bigg] = \E \Bigg[\sum_{i,j=1}^n Z_i Z_j \langle C_i, C_j\rangle\Bigg]
=
\E [Z_1^2] \sum_{j=1}^n  \langle C_j, C_j\rangle
=
\E [Z_1^2] \sum_{i=1}^k \| C_j \|_2^2
=
\E [Z_1^2] k,
$$
where we used that $Z_1,\ldots,Z_n$ are i.i.d.\ and have zero mean, and that $C_1,\ldots,C_n$ are the columns of a Stiefel matrix $V$.

Next we take care of the second factor in \eqref{eq:estiamte expectation via cs}. First, we observe that by the strong law of large numbers and because $\E[|Z_i|^p]=1$ (since $Z_1,\dots,Z_n$ are independent $p$-generalized Gaussian random variables),
\[
\xi_n:=\Big( 1 - U^{1/n}\frac{n^{1/p}}{\|Z\|_p}  \Big)^2 \stackrel{\textrm{a.s.}}{\longrightarrow} 0,\qquad \text{for $n\to\infty$}.
\]
In particular, this means that we have convergence to $0$ in probability. Now, the goal is to show that $\E[\xi_n^2]\leq C$ for some constant $C\in(0,\infty)$ independent of $n$, because if this holds, then the collection $(\xi_n)_{n\in\N}$ is uniformly integrable. And since $L_1$ convergence is equivalent to convergence in probability together with uniform integrability, we obtain the $L_1$ convergence to $0$ of $(\xi_n)_{n\in\N}$. But this means that the right-hand side of \eqref{eq:estiamte expectation via cs} converges to $0$ (even uniformly on $\mathbb V_{k,n}$).

Since by standard estimates and the independence of $U$ and $Z$, we have
\[
\E[\xi_n^2] \leq C + C\,\E[U^{4/n}n^{4/p}\|Z\|_p^{-4}] = C + C\,\E[U^{4/n}]\cdot\E[n^{4/p}\|Z\|_p^{-4}] \leq C+C\,\E[n^{4/p}\|Z\|_p^{-4}]
\]
for some absolute constant $C\in(0,\infty)$, it is enough to prove the existence of
a constant $C_p\in(0,\infty)$ such that for all (sufficiently large) $n\in\N$,
\[
\E\Bigg[\Bigg(\frac{1}{\frac{1}{n}\sum_{i=1}^n|Z_i|^p}\Bigg)^{\frac{4}{p}}\Bigg] \leq C_p.
\]
First, note that since $Z_1,\dots,Z_n$ are independent $p$-generalized Gaussian random variables, each $|Z_i|^p$ satisfies $\E[|Z_i|^p]=1$ and has a Lebesgue density of the form
\begin{align}\label{eq:density modulus p gaussian to power p}
\R \ni x \mapsto \gamma_p x^{\frac{1}{p}-1}e^{-\frac{x}{p}}\, \mathbb 1_{[0,\infty)}(x),
\end{align}
with normalizing constant $\gamma_p:= \big(p^{1/p}\Gamma(1/p)\big)^{-1}\in(0,\infty)$. Since we are dealing with non-negative random variables,
\begin{align*}
\E\Bigg[\Bigg(\frac{1}{\frac{1}{n}\sum_{i=1}^n|Z_i|^p}\Bigg)^{\frac{4}{p}}\Bigg]
  & =
  \frac{4}{p}\int_0^\infty t^{\frac{4}{p}-1}\Pro\Big[\frac{1}{n}\sum_{i=1}^n|Z_i|^p < \frac{1}{t}\Big] \,\dint t \cr
  & =
  \frac{4}{p} \int_0^2 t^{\frac{4}{p}-1}\Pro\Big[\frac{1}{n}\sum_{i=1}^n|Z_i|^p < \frac{1}{t}\Big] \,\dint t + \frac{4}{p} \int_2^\infty t^{\frac{4}{p}-1}\Pro\Big[\frac{1}{n}\sum_{i=1}^n|Z_i|^p < \frac{1}{t}\Big]
  \,\dint t\cr
  & \leq
  \frac{4}{p} \int_0^2 t^{\frac{4}{p}-1}\,\dint t + \frac{4}{p} \int_2^\infty t^{\frac{4}{p}-1}\Pro\Big[\frac{1}{n}\sum_{i=1}^n|Z_i|^p < \frac{1}{t}\Big]\,\dint t \cr
  & =
  2^{\frac{4}{p}} + \frac{4}{p} \int_2^\infty t^{\frac{4}{p}-1}\Pro\Big[\frac{1}{n}\sum_{i=1}^n|Z_i|^p < \frac{1}{t}\Big] \,\dint t,
\end{align*}
where the first summand is trivially bounded by $16$ since $p\geq 1$.
It is thus left to estimate the second summand in the previous line. We again split the integral and write
\begin{align}\label{eq:separation in two parts}
\frac{4}{p}\int_2^\infty t^{\frac{4}{p}-1}\Pro\Big[\frac{1}{n}\sum_{i=1}^n|Z_i|^p < \frac{1}{t}\Big] \,\dint t
&\nonumber = \frac{4}{p}\int_2^{n^6} t^{\frac{4}{p}-1}\Pro\Big[\frac{1}{n}\sum_{i=1}^n|Z_i|^p < \frac{1}{t}\Big] \,\dint t \\
&\quad + \frac{4}{p}\int_{n^6}^\infty t^{\frac{4}{p}-1}\Pro\Big[\frac{1}{n}\sum_{i=1}^n|Z_i|^p < \frac{1}{t}\Big] \,\dint t
\end{align}
and estimate both summands separately. We start with the first, which is handled by means of a Cram\'er or Chernoff-type bound (see, e.g., \cite{DH2000}). Indeed, for any $0\leq a < 1$, it follows from Markov's inequality and the i.i.d.~property of the random variables that
\begin{align*}
\Pro\Big[\frac{1}{n}\sum_{i=1}^n|Z_i|^p \leq a \Big]
\leq
e^{-nI(a)},
% & = \Pro\Big[e^{-tn\frac{1}{n}\sum_{i=1}^n|Z_i|^p} \geq e^{-tna }\Big]
%   =
%  \Pro\Big[e^{-t\sum_{i=1}^n|Z_i|^p} \geq e^{-tna}\Big] \cr
%  & \leq
%  e^{tna}\, \E\Big[\prod_{i=1}^n e^{-t|Z_i|^p}\Big]
%   =
%  e^{n\big(ta + \log \E\big[e^{-t|Z_1|^p}\big]\big)}.
\end{align*}
%Optimizing in the parameter $t>0$, we obtain the upper bound
%\begin{align}\label{eq:chernoff bound}
%\Pro\Big[\frac{1}{n}\sum_{i=1}^n|Z_i|^p & \leq a \Big] \leq e^{n\inf_{t>0}\big(ta + \log \E\big[e^{-t|Z_1|^p}\big]\big)} = %e^{-n\varphi(a)},
%\end{align}
where $I(a):=\sup_{t\in\R} \big( ta  - \log \E\big[e^{t|Z_1|^p}\big] \big) = \sup_{t<1/p} \big( ta  - \log \E\big[e^{t|Z_1|^p}\big] \big)$ is the Cram\'er information function of the random variable $|Z_1|^p$.  Note that $I(a) \geq 0$, the only zero of $I(a)$ is $a=\E |Z_1|^p = 1$ (see, e.g., \cite[Lemma I.14]{DH2000}) and that in particular $I(1/2)>0$ . It follows from the Chernoff bound that
\begin{align*}
\frac{4}{p}\int_2^{n^6} t^{\frac{4}{p}-1}\Pro\Big[\frac{1}{n}\sum_{i=1}^n|Z_i|^p < \frac{1}{t}\Big] \,\dint t
\leq
\frac{4}{p}\int_2^{n^6} t^{\frac{4}{p}-1}\Pro\Big[\frac{1}{n}\sum_{i=1}^n|Z_i|^p < \frac{1}{2}\Big] \,\dint t
& \leq
\frac{4}{p}e^{-n I(1/2)} \int_2^{n^6} t^{\frac{4}{p}-1} \,\dint t \cr
& \leq
n^{\frac{24}{p}}e^{-nI(1/2)}
\stackrel{n\to\infty}{\longrightarrow} 0,
\end{align*}
where in the last step we used that $I(1/2)>0$.
%Hence, using the previous bounds together with the fact that $\varphi(1/2)>0$, we obtain
%\[
%\frac{4}{p}\int_2^{n^6} t^{\frac{4}{p}-1}\Pro\Big[\frac{1}{n}\sum_{i=1}^n|Z_i|^p < \frac{1}{t}\Big] \,\dint t \leq %n^{\frac{24}{p}}e^{-n\varphi(1/2)} \stackrel{n\to\infty}{\longrightarrow} 0.
%\]
We now estimate the second summand on the right-hand side of \eqref{eq:separation in two parts}. Thereto, we observe that,
 %for any $t>0$,
%\[
%\Pro\Big[\frac{1}{n}\sum_{i=1}^n|Z_i|^p < \frac{1}{t}\Big] \leq \Pro\Big[\frac{|Z_1|^p}{n} < \frac{1}{t}\Big] = \Pro\Big[|Z_1|^p < %\frac{n}{t}\Big]
%\]
since $|Z_1|^p$ has density on $\R$ given by \eqref{eq:density modulus p gaussian to power p}, we obtain, for every $t>0$,
\begin{align*}
\Pro\Big[|Z_1|^p < \frac{n}{t}\Big] = \gamma_p \int_0^{n/t}x^{\frac{1}{p}-1}e^{-\frac{x}{p}}\,\dint x \leq \gamma_p \int_0^{n/t}x^{\frac{1}{p}-1}\,\dint x = p\gamma_p \Big(\frac{n}{t}\Big)^{\frac{1}{p}}.
\end{align*}
In particular, because of positivity, the i.i.d.~property, and the previous estimate, for any $n\geq 5$,
\begin{align*}
\Pro\Big[\frac{1}{n}\sum_{i=1}^n|Z_i|^p < \frac{1}{t}\Big]
  & \leq
  \Pro\Big[\frac{|Z_1|^p}{n}+\dots+\frac{|Z_5|^p}{n} < \frac{1}{t}\Big] \cr
  & \leq \Pro\Big[\frac{|Z_1|^p}{n} < \frac{1}{t}, \dots,\frac{|Z_5|^p}{n} < \frac{1}{t}\Big] = \Pro\Big[\frac{|Z_1|^p}{n} < \frac{1}{t}\Big]^5 \leq p^5\gamma_p^5 \Big(\frac{n}{t}\Big)^{\frac{5}{p}}.
\end{align*}
Putting things together, for any sufficiently large $n\in\N$,
\begin{align*}
\frac{4}{p}\int_{n^6}^\infty t^{\frac{4}{p}-1}\Pro\Big[\frac{1}{n}\sum_{i=1}^n|Z_i|^p < \frac{1}{t}\Big] \,\dint t
   \leq
   4p^4\gamma_p^5 n^{\frac{5}{p}}\int_{n^6}^\infty t^{\frac{4}{p}-1}  t^{-\frac{5}{p}}\,\dint t
  = 4p^5\gamma_p^5 n^{\frac{5}{p}}n^{-\frac{6}{p}} \stackrel{n\to\infty}{\longrightarrow} 0.
\end{align*}
The proof is now completed by combining the bound in \eqref{eq:levy prokhorov bound} together with the one in \eqref{eq:estiamte expectation via cs} with the previous estimates.
\end{proof}

\subsection{Step 2 -- Equivalence of LDP's \& LDP lifting}

We now let $V_{k,n}$ be a random Stiefel matrix uniformly distributed on $\mathbb V_{k,n}$ and consider  $\widetilde{\mu}_{V_{k,n}}$ and $\mu_{V_{k,n}}$ as \textit{random} elements with values in $\mathcal M_1(\R^k)$ obtained by taking $V=V_{k,n}$ in the definitions~\eqref{eq:mu_V_def} and~\eqref{eq:mu_V_tilde_def} of $\widetilde{\mu}_{V}$  and  $\mu_V$.
The next lemma shows that if the modified sequence $\widetilde{\mu}_{V_{k,n}}$, $n\in\N$ satisfies an LDP, then so does $\mu_{V_{k,n}}$, $n\in\N$, with the same speed and at the same rate.

\begin{lemma}\label{lem:weak ldp for tilde mu implies weak ldp for mu}
Let $k\in\N$ be fixed. For $n\geq k$ let $V_{k,n}$ be a matrix chosen uniformly at random from $\mathbb V_{k,n}$ and, for $p\in[1,\infty]$, let $Z=(Z_1,\dots,Z_n)$ be a random vector with independent $p$-Gaussian entries such that $Z$ is independent of $V_{k,n}$.
%Denote by $C_j:=C_j(V_{k,n}):= (V_{k,n}(i,j))_{i=1}^k$ the $j$th column of $V_{k,n}$.
If the sequence $\widetilde{\mu}_{V_{k,n}}$, $n\geq k$
%\[
%\widetilde{\mu}_{V_{k,n}}(\cdot) = \Pro\Bigg[\sum_{j=1}^nZ_j C_j(V_{k,n}) \in \cdot \Bigg],\qquad n\in\N,
%\]
satisfies a weak LDP  on $\mathcal M_1(\R^k)$ at speed $n$ with rate function $\rate:\mathcal M_1(\R^k)\to[0,+\infty]$, then the sequence $\mu_{V_{k,n}}$, $n\geq k$, satisfies the same weak LDP.
%\begin{align*}
%\mu_{V_{k,n}}(\cdot)
%  & \stackrel{\dint}{=}
%     \Pro\Bigg[\sum_{j=1}^n U^{1/n}n^{1/p}\frac{Z_j}{\|Z\|_p} C_j(V_{k,n}) \in \cdot \Bigg],\qquad n\in\N,
%\end{align*}
%satisfies the same LDP, where $U\sim\Uni([0,1])$ is independent of all other random elements.
\end{lemma}
\begin{proof}
We shall use Proposition \ref{prop:basis topology} in combination with Lemma \ref{lem:levy-prokhorov lemma} to prove this result. %Let us recall that
%\begin{align*}
%\mu_{V_{k,n}}(\cdot)
%  & \stackrel{\dint}{=}
%     \Pro\Bigg[\sum_{j=1}^n U^{1/n}n^{1/p}\frac{Z_j}{\|Z\|_p} C_j(V_{k,n}) \in \cdot \Bigg],\qquad n\in\N,
%\end{align*}
%where $U\sim\Uni([0,1])$ is independent of all other random elements.
By Proposition \ref{prop:basis topology} it is enough to control the large deviation behavior of $\mu_{V_{k,n}}$, $n\geq k$, on a base of the topology of $\mathcal M_1(\R^k)$, which is given, e.g., by open balls in the L\'evy-Prokhorov distance. So for $r\in(0,\infty)$ and $\nu\in\mathcal M_1(\R^k)$, we consider the ball
\[
B_r(\nu) := \Big\{ \mu\in\mathcal M_1(\R^k)\,:\, \rho_{\textrm{LP}}(\mu,\nu)< r\Big\}.
\]
From Lemma \ref{lem:levy-prokhorov lemma}, we know that
\begin{equation}\label{eq:unif_Lemma_5.2}
\lim_{n\to\infty} \sup_{V \in\mathbb V_{k,n}}\rho_{\textrm{LP}}\big(\widetilde{\mu}_{V}, \mu_{V}\big) = 0,
\end{equation}
i.e., $\rho_{\textrm{LP}}\big(\widetilde{\mu}_{V}, \mu_{V}\big)$ converges to zero uniformly on $\mathbb V_{k,n}$ as $n\to\infty$; the uniformity of convergence is required in the following argument. For $n\in\N$ sufficiently large, we obtain
\begin{align}\label{eq:upper estimate on levy-prokhorov balls}
 \frac{1}{n} \log \Pro\big[\widetilde{\mu}_{V_{k,n}} \in B_{r/2}(\nu)\big] \leq \frac{1}{n} \log \Pro\big[\mu_{V_{k,n}} \in B_r(\nu)\big] \leq \frac{1}{n} \log \Pro\big[\widetilde{\mu}_{V_{k,n}} \in B_{3r/2}(\nu)\big].
\end{align}
Indeed, because of the uniform convergence stated in~\eqref{eq:unif_Lemma_5.2}, we obtain for each realization of $V_{k,n}$ that, as long as $n\in\N$ is sufficiently large,
$
\rho_{\textrm{LP}}\big(\widetilde{\mu}_{V_{k,n}}, \mu_{V_{k,n}}\big) < r/2
$
and~\eqref{eq:upper estimate on levy-prokhorov balls} follows from the triangle inequality for the L\'evy-Prokhorov metric.
%But this simply means that for any Borel set $D\subseteq \R^k$ and sufficiently large $n\in\N$, we have
%\[
%\mu_{V_{k,n}}(D) < \widetilde{\mu}_{V_{k,n}}(D_r) + r.
%\]
%Now $\mu_{V_{k,n}} \in B_r(\nu)$ means that $\nu(A)< \mu_{V_{k,n}}(A_r)+r$  for any Borel set $A\subseteq \R^k$. Hence, by the %estimate in the previous display,
%\[
%\nu(A) < \mu_{V_{k,n}}(A_r)+r < \widetilde{\mu}_{V_{k,n}}(A_{2r}) + 2r
%\]
%and so $\widetilde{\mu}_{V_{k,n}}\in B_{2r}(\nu)$, which proves \eqref{eq:upper estimate on levy-prokhorov balls}.
Therefore,
\begin{multline*}
\limsup_{n\to\infty}
\frac{1}{n} \log \Pro\big[\widetilde{\mu}_{V_{k,n}} \in B_{r/2}(\nu)\big]
\leq
\liminf_{n\to\infty}\frac{1}{n} \log \Pro\big[\mu_{V_{k,n}} \in B_r(\nu)\big]
\\
\leq
\limsup_{n\to\infty}\frac{1}{n} \log \Pro\big[\mu_{V_{k,n}} \in B_r(\nu)\big]
\leq
\liminf_{n\to\infty}\frac{1}{n} \log \Pro\big[\widetilde{\mu}_{V_{k,n}} \in B_{2r}(\nu)\big].
\end{multline*}
Clearly, the expressions are monotone in the radius $r$ and thus taking the infimum over $r>0$ and using
the assumption that $\widetilde{\mu}_{V_{k,n}}$ satisfies a weak LDP at speed $n$ with rate function $\rate$,
\[
-\rate(\nu)
\leq
\inf_{r\in(0,\infty)}\liminf_{n\to\infty}\frac{1}{n} \log \Pro\big[\mu_{V_{k,n}} \in B_r(\nu)\big]
\leq
\inf_{r\in(0,\infty)}\limsup_{n\to\infty}\frac{1}{n} \log \Pro\big[\mu_{V_{k,n}} \in B_r(\nu)\big]
\leq
 -\rate(\nu),
\]
and so Proposition \ref{prop:basis topology} yields the weak LDP for $\mu_{V_{k,n}}$ at speed $n$ with rate function $\rate$.
%The lower bound is obtained in a similar way. Indeed, it follows again from the uniform convergence established in Lemma %\ref{lem:levy-prokhorov lemma} that, for each realization and sufficiently large $n\in\N$,
%\[
%\rho_{\textrm{LP}}\big(\widetilde{\mu}_{V_{k,n}}, \mu_{V_{k,n}}\big) < \frac{r}{2}.
%\]
%Observing that if $\widetilde{\mu}_{V_{k,n}} \in B_{r/2}(\nu)$, then $\widetilde{\mu}_{V_{k,n}}(D) < \nu(D_{r/2})+r/2$ for each %Borel set $D\subseteq\R^k$, we obtain from the previous display that, for all Borel sets $A\subseteq\R^k$,
%\[
%\mu_{V_{k,n}}(A) < \widetilde{\mu}_{V_{k,n}}(A_{r/2})+\frac{r}{2} < \nu(A_r) + r.
%\]
%his shows that, for sufficiently large $n\in\N$,
%\[
%\liminf_{n\to\infty}\frac{1}{n} \log \Pro\big[\mu_{V_{k,n}} \in B_r(\nu)\big] \geq \liminf_{n\to\infty}\frac{1}{n} \log %Pro\big[\widetilde{\mu}_{V_{k,n}} \in B_{2r}(\nu)\big].
%\]
%As before, we obtain from the weak LDP for $\widetilde{\mu}_{V_{k,n}}$, $n\in\N$, that
%\[
%\inf_{r\in(0,\infty)}\liminf_{n\to\infty}\frac{1}{n} \log \Pro\big[\mu_{V_{k,n}} \in B_r(\nu)\big]
%\geq
%\inf_{r\in(0,\infty)}\liminf_{n\to\infty}\frac{1}{n} \log \Pro\big[\widetilde{\mu}_{V_{k,n}} \in B_{2r}(\nu)\big] = -\rate(\nu).
%\]
%Thus, putting upper and lower bound together, we obtain
%\[
%-\rate(\nu) = \inf_{r\in(0,\infty)}\limsup_{n\to\infty}\frac{1}{n} \log \Pro\big[\mu_{V_{k,n}} \in B_r(\nu)\big] = %\inf_{r\in(0,\infty)}\liminf_{n\to\infty}\frac{1}{n} \log \Pro\big[\mu_{V_{k,n}} \in B_r(\nu)\big]
%\]
%and so Proposition \ref{prop:basis topology} yields the weak LDP for $\mu_{V_{k,n}}$, $n\in\N$, at speed $n$ with rate function $\rate$.
\end{proof}

The following lemma is needed  in order to lift the weak LDP to a full one since on a compact space both notions coincide.

\begin{lemma}
There is $C\in(0,\infty)$ such that, for all $n\geq k$ and each $V\in \mathbb V_{k,n}$, we have
$$
\mu_{V}\in M_{C} :=\Bigg\{ \mu\in\mathcal M_1(\R^k)\,:\, \int_{\R^k} \|x\|_2\,\mu(\dint x) \leq C \Bigg\}.
$$
The set $M_C$ is compact in $\mathcal M_1(\R^k)$ for all $C\in(0,\infty)$. %Consequently, the  $\mu_{V_{k,n}}$, $n\in\N$,
\end{lemma}
\begin{proof}
%Let $C\in(0,\infty)$ and consider the set
%\[
%_{C} :=\Bigg\{ \mu\in\mathcal M_1(\R^k)\,:\, \int_{\R^k} \|x\|_2\,\mu(\dint x) \leq 2C \Bigg\}.
%\]
As was shown in the proof of Lemma~\ref{lem:levy-prokhorov lemma} (see, in particular,~\eqref{eq:estiamte expectation via cs} and the argument following it), we have
$$
\lim_{n\to\infty} \sup_{V\in \mathbb V_{k,n}} \E \big[\|\xi_n - \widetilde \xi_n\|_2\big] =0
\qquad
\text{ and }
\qquad
\E\big[ \|\widetilde \xi_n\|_2^2\big] = k\, \E[|Z_1|^2] =: B<\infty,
$$
where for each $V\in \mathbb V_{k,n}$
$$
\xi_n := \xi_n(V) =  \sum_{j=1}^nZ_jC_j(V) U^{1/n}\frac{n^{1/p}}{\|Z\|_p},
\qquad
\widetilde \xi_n :=\widetilde \xi_n(V) =  \sum_{j=1}^nZ_jC_j(V)
$$
are random vectors in $\R^k$ with distributions $\mu_{V}$ and $\widetilde \mu_{V}$, respectively. It follows from H\"older's inequality that $\E[\|\widetilde \xi_n\|_2]  \leq \sqrt B$. Hence, the triangle inequality $\|\xi_n \|_2 \leq \|\xi_n - \widetilde \xi_n\|_2 + \|\widetilde \xi_n\|_2$ implies that
$$
\limsup_{n\to\infty} \sup_{V\in \mathbb V_{k,n}} \E \big[\|\xi_n\|_2\big] \leq \sqrt B < \infty,
$$
which implies that for sufficiently large $C\in(0,\infty)$ and all $n\in\N$, $V\in \mathbb V_{k,n}$, we have $\mu_{V} \in M_C$.

The compactness of $M_C$ in the weak topology on $\mathcal M_1(\R^k)$ follows by a standard argument based on Prokhorov's theorem \cite{Proh1956} (see also \cite[Theorem 16.3]{Kallenberg} or \cite[Theorem 13.29]{Klenke2014}); in particular the weak compactness does not depend on the choice of $C\in(0,\infty)$ above . Let us show this.
%We start by showing that for each $c>0$ the set $K_c:= \{\nu \in \mathcal M(\R): \int_\R |x|^p \nu(\dint x) \leq c\}$ is compact in $\mathcal M(\R)$.
The closedness of $M_C$ follows by considering the cut-off $\min\{\|x\|_2, n\}$ and using the definition of the weak topology together with the monotone convergence theorem. It therefore just remains to show that $M_C$ is relatively compact in $\mathcal M_1(\R^k)$. To do this, we prove that $M_C$ is a tight family of measures and then use Prokhorov's theorem, which states that a tight family of probability measures is relatively compact in the weak topology. Given some $\varepsilon>0$, putting $C_\varepsilon:=C/\eps$ and using Markov's inequality, for every $\mu\in M_C$, we obtain
$$
\mu\Big(\big\{x\in\R^k\,:\,\|x\|_2>C_\varepsilon\big\}\Big) \leq {\int_{\R^k}\|x\|_2\mu(\dint x) \over C_\varepsilon}\leq {C\over C_\varepsilon} = \eps.
$$
Therefore, we have
\[
\sup_{\mu\in M_C}\mu\big(\big\{x\in\R^k\,:\,\|x\|_2\leq C_\varepsilon\big\}^c\big) \leq \varepsilon,
\]
which shows that the family $M_C$ of probability measures is tight. Thus, by Prokhorov's theorem and the fact that $M_C$ is closed in the weak topology, the set $M_C$ is compact in $\mathcal M_1(\R^k)$.
\end{proof}

%\begin{rmk}
The previous lemma together with Lemma \ref{lem:weak ldp for tilde mu implies weak ldp for mu} implies that under the assumption that $\widetilde{\mu}_{V_{k,n}}$ satisfies a weak LDP at speed $n$ with rate function $\rate$, the sequence $\mu_{V_{k,n}}$, $n\in\N$, satisfies a full LDP at the same speed with the same rate function; in particular, the rate function is good. Knowing this, Theorem~\ref{thm:ldp multidimensional projection unif distr pball} becomes a consequence of Theorem \ref{thm:ldp multidimensional projection_product_measures}. The latter theorem will be established in the next Section~\ref{sec:proof_theorem_D}, its proof being quite delicate.
%\end{rmk}

%{\color{red} This is the end of the proof of Theorem~\ref{thm:ldp multidimensional projection unif distr pball}. The rest of this section should be removed. }

% % % % % % % % % % % % % % % % % % % % % % % % % % % % % % % % % % % % % % %
\section{Proof of Theorem \ref{thm:ldp multidimensional projection_product_measures} -- the LDP for \texorpdfstring{$k$}{k}-dimensional projections of product measures}\label{sec:proof_theorem_D}
% % % % % % % % % % % % % % % % % % % % % % % % % % % % % % % % % % % % % % %
In this section we prove Theorem~\ref{thm:ldp multidimensional projection_product_measures}, whose proof consists of seven steps.
The general strategy is to prove an LDP on a  compact space $\mathbb W_k$  which will be introduced in Step~1. This LDP will then be transported to an LDP on the space $\mathcal M_1(\R^k)$ of probability measures by means of mapping $\Psi$ (defined in Step~2) which we shall show to be a homeomorphism in Steps 3, 4 (injectivity) and  5 (continuity). The LDP on $\mathbb W_k$ shall be proved in Step 6. In Step~7 we shall map this LDP to an LDP on $\mathcal M_1(\R^k)$.

\subsection{Step 1 -- Space of deviations}
We are going to define a space $\mathbb W_k$ (endowed with the topology of vague convergence) whose elements encode, in some sense, all possible deviations of the projected high-dimensional product measures. Let us first try to give some intuition. Large deviations of the projected product measure are caused by ``atypical'' realizations of the random Stiefel matrix $V_{k,n}$. In a ``typical'' realization, all columns are infinitesimal, while for atypical ones certain columns are of order $1$. The positions of these columns and their signs do not influence the shape of the projected measure, which is why we record these columns  (together with their negatives) as a symmetric point configuration in $\R^k$. The typical realizations of the random Stiefel matrix correspond to the empty configuration.

Let us be more precise. Fix $k\in\N$, the dimension of the space we project on.   We denote by $\mathbb W_k$ the set of all Borel measures $\mu$ on $[-1,1]^k\backslash\{0\}$ with the following properties:
\begin{itemize}
\item[($\mathbb W1$)] $\mu$ is symmetric, which means that it is invariant under the mapping $x\mapsto -x$.
\item[($\mathbb W2$)] $\mu$ is locally finite on $[-1,1]^k\backslash\{0\}$, meaning that $\mu(K)<\infty$ for every compact set $K\subset [-1,1]^k\backslash\{0\}$.
\item[($\mathbb W3$)] $\mu$ is integer-valued, meaning that $\mu(B)\in \N \cup \{0,\infty\}$ for every Borel set $B\subset [-1,1]^k\backslash\{0\}$.
\item[($\mathbb W4$)] The above three properties imply (see, e.g., \cite[Proposition~1.1.2 on p.~9]{K1978}) that $\mu$ admits a representation of the form
$$
\mu = \sum_{j=1}^{N}(\delta_{C_j} + \delta_{-C_j}),
$$
where $N=N(\mu)\in\N\cup\{0,\infty\}$  and $C_1,C_2,\ldots\in [-1,1]^k\backslash\{0\}$. If $N=0$, then $\mu=0$. If $N\in \N$ is finite, then $C_1,\ldots,C_N$ are some points in $[-1,1]^k\backslash\{0\}$. If $N=+\infty$, then $C_1,C_2,\ldots [-1,1]^k\backslash\{0\}$ is a sequence which converges to $0\in\R^k$.
\item[($\mathbb W5$)] Additionally, we require the $k\times\infty$ matrix $V(\mu)$ whose columns are the vectors $C_1,C_2,\ldots$ to have square summable rows, i.e., $V(\mu)\in \mathcal R_2^{k\times \infty}$, and to satisfy $\|V(\mu) V(\mu)^*\| \leq 1$. More precisely, $V(\mu)$ is defined as follows. If $N=0$, then $V(\mu)$ is the zero $k\times\infty$-matrix. If $N\in \N$ is finite, then we define $V(\mu)$ to be the $k\times\infty$ matrix with columns $C_1,\ldots,C_N$ filled up with infinitely many zero columns. Finally, if $N=\infty$, then $V(\mu)$ is the $k\times \infty$-matrix with columns $C_1,C_2,\ldots$.  Let us stress that the matrix $V(\mu)$ is defined up to a signed permutation of its non-zero columns only. However, as one readily checks, the $k\times k$-matrix $V(\mu)V(\mu)^*$ does not change under signed permutations of the columns of the matrix $V(\mu)$ and is therefore well defined.
\end{itemize}

We shall endow the set $\mathbb W_k$ with the topology of vague convergence of locally finite measures (see, e.g.,~\cite[Section~3.4]{R2008}). Recall that a sequence $(\mu_n)_{n\in\N}\subset \mathbb W_k$ of measures converges in the vague sense if and only if for each continuous, compactly supported function $f:[-1,1]^k\backslash\{0\}\to\R$, we have
\begin{equation}\label{eq:vague_conv_f}
\lim_{n\to\infty} \int_{[-1,1]^k\backslash\{0\}} f \,\dint\mu_n = \int_{[-1,1]^k\backslash\{0\}} f\,\dint\mu.
\end{equation}
It is known~\cite[Proposition~3.17]{R2008} that the vague topology is metrizable by a complete, separable metric. In the following lemma, we show that $\mathbb W_k$ is compact.
%As we shall show in Lemma~\ref{lem:W_k_compact}, the space $\mathbb W_k$ is compact.

\begin{lemma}\label{lem:W_k_compact}
The space $\mathbb W_k$ is compact.
\end{lemma}
\begin{proof}
Take a sequence $\mu_1,\mu_2,\ldots$ in $\mathbb W_k$. We have to show that it has a convergent subsequence.

\vspace*{2mm}
\noindent
\textsc{Step 1}. Let $E$ be the set of all points $x\in [-1,1]^k$ representable as $x= \lim_{n\to\infty} x_n$, where $x_n$ is an atom of $\mu_n$, for every $n\in \N$. We show that $E$ is countable. It suffices to check that $E\cap \B_2^k(0,\eps)^c$ is finite for every $\eps>0$, where $\B_2^k(0,\eps)^c=[-1,1]^k\backslash \B_2^k(0,\eps)$. Now, for every $\mu=\sum_{j=1}^N (\delta_{C_j}+\delta_{-C_j})\in \mathbb W_k$, we have
$\sum_{j=1}^N \|C_j\|_2^2 \leq k$, which holds because the condition $\|V(\mu) V(\mu)^*\| \leq 1$ implies $\sum_{j=1}^N |C_j(i)|^2 \leq 1$ for each $i\in\{1,\dots,k\}$. It follows that the number of $j$ with $\|C_j\|_2 \geq  \eps/2$ is at most $4k/\eps^2$. In particular, the number of atoms of $\mu_n$ outside $\B_2^k(0,\eps/2)$ is at most $8k/\eps^2$, for every $n\in \N$. Now assume that $E\cap \B_2^k(0,\eps)^c$ contains  $L > 8k/\eps^2$ different points $p_1= \lim_{n\to\infty} x_{1;n},\ldots,p_L=\lim_{n\to\infty} x_{L;n}$, with $x_{1;n},\ldots, x_{L;n}\in \B_2^k(0,\eps/2)^c$ being atoms of $\mu_n$, for $n>n_0$. By the pigeon-hole principle, for each $n>n_0$ two of the points $x_{1;n},\ldots, x_{L;n}$ must be equal. It follows that there exist two different $i,j\in \{1,\ldots, L\}$ such that the sequences $(x_{i;n})_{n\in \N}$ and $(x_{j;n})_{n\in \N}$ have infinitely many common terms. This implies $p_i=p_j$, which is a contradiction. Hence, $E\cap \B_2^k(0,\eps)^c$ is finite.

\vspace*{2mm}
\noindent
\textsc{Step 2}. Fix some $\eps>0$ with the property that $E$ does not contain points with $\|x\|_2 = \eps$. The number of atoms of $\mu_n$ with a norm $\geq \eps$ is bounded above by $2 k/ \eps^2$ (see argument in Step 1). By the pigeon-hole principle, along some subsequence of $n$'s, $\mu_n(\B_2^k(0,\eps)^c) = 2p$ stays constant, where we denote the atoms of $\mu_n$ with norm $\geq \eps$ by $\pm C_{1;n}, \ldots, \pm C_{p;n}$. Since $\B_2^k(0,\eps)^c$ is compact, we can again pass to a subsequence of $n$'s (and relabel the atoms, if necessary) along which we have $C_{j,n} \to C_j$ as $n\to\infty$,  for all $j=1,\ldots, p$. Note that $\|C_j\|_2 \neq \eps$ by our choice of $\eps$. By definition of vague convergence~\eqref{eq:vague_conv_f}, this implies that the restriction of $\mu_n$ to $\B_2^k(0,\eps)^c$ converges vaguely along the subsequence constructed above.

\vspace*{2mm}
\noindent
\textsc{Step 3}. By Step 1 we can find a decreasing sequence $\eps_r\downarrow 0$ such that $E$ does not contain points with $\|x\|_2 = \eps_r$, for all $r\in \N$. Applying Step~2 to $\eps= \eps_1$, we extract a subsequence $N_1\subset N$ along which the restrictions of $\mu_n$ to $\B_2^k(0,\eps_1)^c$ converge vaguely.  Applying  Step~2 to $\eps= \eps_2$, we find a subsequence $N_2\subset N_1$ along which the restrictions of $\mu_n$ to $\B_2^k(0,\eps_2)^c$ converge vaguely.
Doing this inductively, we obtain subsequences $N_1 \supset N_2 \supset \ldots$. Applying Cantor's diagonal argument, we get a subsequence $N_\infty \subset N$ along which the restrictions of $\mu_n$ to $\B_2^k(0,\eps_r)^c$ converge vaguely for every $r\in\N$. This implies that $\mu_n$ converges vaguely along $N_\infty$, thus establishing the compactness of $\mathbb W_k$.
\end{proof}

\begin{rmk}
In the special case when $k=1$, there is an alternative description of the space $\mathbb W_1$ which was used in~\cite{JKP2021}. Consider the set of all sequences $\alpha = (\alpha_1,\alpha_2,\ldots)\in \mathbb R^\infty$ such that  $\alpha_1\geq \alpha_2 \geq \ldots \geq 0$ and  $\|\alpha\|_2 \leq 1$ and endow it with the topology of coordinatewise convergence inherited from $\R^\infty$.
One can check that assigning $\alpha \mapsto \sum_{i\in \N: \alpha_i\neq 0} (\delta_{\alpha_i} + \delta_{-\alpha_i})$ defines a homeomorphism  between this space and $\mathbb W_1$. For $k\geq 2$, it is possible to order the atoms of $\mu\in \mathbb W_k$ in decreasing order of their norms, but there is no canonical ordering of atoms having equal norm, and also no canonical choice of the ``signs''. Therefore, it seems inconvenient to order the atoms for $k\geq 2$.
\end{rmk}

\begin{rmk}
The elements of $\mathbb W_k$ are in one-to-one correspondence with equivalence classes of  $k\times \infty$-matrices $V$ with square summable rows and $\|V V^*\| \leq 1$, where we call two such matrices equivalent, if they differ by a signed permutations of their columns.
However, one should be careful about the topology. Consider for simplicity the case $k=1$. Let $\mathbb B_2^\infty= \{x\in\ell_2: \|x\|_2\leq 1\}$ be the unit ball in the Hilbert space $\ell_2$ of square summable sequences endowed with the topology of coordinatewise convergence.  Call two sequences $\alpha',\alpha''\in \mathbb B_2^\infty$ equivalent if they differ by a signed permutation of coordinates. Let $W_1'$ be the space of equivalence classes endowed with the quotient topology. Then, there is a natural bijection between the elements of $W_1'$ and $\mathbb W_1$ which maps the equivalence class of $\alpha= (\alpha_1,\alpha_2,\ldots)$ to $\sum_{i\in \N: \alpha_i\neq 0} (\delta_{\alpha_i} + \delta_{-\alpha_i})$.   However, it is not a homeomorphism. In fact, the space $W_1'$ is not Hausdorff (whereas $\mathbb W_1$ is Polish). Indeed, one can check that any neighborhood of $(0,0,\ldots)$ in the quotient topology of $W_1'$ contains the element $(1,0,0,\ldots)$.
\end{rmk}

\subsection{Step 2 --  Definition of the map \texorpdfstring{$\Psi$}{Psi}}
Let $Y_1,Y_2,\ldots$ be non-Gaussian i.i.d.\ random variables with symmetric distribution (meaning that $Y_i$ has the same distribution as $-Y_i$) and finite variance $\sigma^2 := \E [Y_1^2] <\infty$. We define a map $\Psi: \mathbb W_k \to \mathcal M_1(\R^k)$ as follows: for $\mu=\sum_{j=1}^N (\delta_{C_j}+\delta_{-C_j})\in \mathbb W_k$, we put
\[
\Psi(\mu) := \text{Law}\,\Big(\sum_{j=1}^N C_jY_j + \sigma \left(\id_{k\times k} - V(\mu)V(\mu)^*\right)^{1/2} N_k\Big),
\]
where $N_k$ is a $k$-dimensional standard Gaussian random vector independent of the sequence $(Y_i)_{i=1}^\infty$. For example, $\Psi(0)$ is an isotropic Gaussian distribution on $\R^k$ with covariance matrix $\sigma^2 \id_{k\times k}$.
Let us argue that the mapping $\Psi$ is well-defined. First of all, as explained in the previous subsection, the term $\sigma (\id_{k\times k} - V(\mu)V(\mu)^*)^{1/2} N_k$  is well defined. Furthermore, it is clear that the law of $\sum_{j=1}^NC_jY_j$ is invariant under   signed permutations of $C_1,\ldots,C_N$ if $N<\infty$. On the other hand, if $N=\infty$, then the series $\sum_{j=1}^NC_jY_j$ converges a.s.\ (because its terms are independent and their $L^2$-norms are summable) and, moreover, the distribution of the sum is invariant under signed permutation of the summands, as we shall see in the proof of Lemma~\ref{lem:vague vs weak convergence}.  It also follows  that for every $\mu \in \mathbb W_k$, the probability measure $\Psi(\mu)$ is symmetric  and its covariance matrix is $\sigma^2 \id_{k\times k}$. Indeed, the covariance matrix of the random vector $\sum_{j=1}^N C_jY_j$ is just $\sigma^2 V(\mu)V(\mu)^*$.

In the following Sections \ref{sec:equality linear forms} and \ref{sec:injectivity}, we shall show that the map $\Psi$ is injective. Moreover, we shall prove that $\Psi$ is a homeomorphism between $\mathbb W_k$ and its image $\mathcal K_{k,Y_1} = \Psi (\mathbb W_k)$, which is a compact subset of $\mathcal M_1(\R^k)$ endowed with the topology of weak convergence.
After these preparatory results, our approach to prove Theorem~\ref{thm:ldp multidimensional projection_product_measures} is as follows. Given a random uniform matrix $V_{k,n}$ from the Stiefel manifold $\mathbb V_{k,n}$, we are interested in the random probability measure $\widetilde \mu_{V_{k,n}}$ defined as the image of the law of $(Y_1,\ldots,Y_n)$ on $\R^n$ under the random map $V_{k,n}:\R^n\to \R^k$ (cf. \eqref{eq:mu_V_tilde_def}).
Consider the following  random element taking values in $\mathbb W_k$:
$$
\eta_n = \sum_{j=1}^n \left(\delta_{C_j(V_{k,n})} + \delta_{-C_j(V_{k,n})}\right),
$$
where $C_1(V_{k,n}),\ldots, C_n(V_{k,n})$ are the columns of the matrix $V_{k,n}$.
Since $V_{k,n} V_{k,n}^* = \id_{k\times k}$, it follows from the very definition of the map $\Psi$ that $\Psi (\eta_n) = \widetilde \mu_{V_{k,n}}$. We shall prove that the sequence $\eta_n$, $n\geq k$, satisfies an LDP on the compact space $\mathbb W_k$. Using the homeomorphism $\Psi$, this LDP will then be transferred to an LDP for $\tilde \mu_{V_{k,n}}$ on $\mathcal K_{k,Y_1}$.

%\subsection{Compactness of $\mathbb W_k$}

% % % % % % % % % % % % % % % % % % % % % % % % % %
\subsection{Step 3 --  Equality in distribution of linear forms}\label{sec:equality linear forms}
% % % % % % % % % % % % % % % % % % % % % % % % % %
Our goal is to show that the map $\Psi : \mathbb W_k \to \mathcal M_1(\R^k)$ defined above is injective. As a first step we prove the following preliminary lemma, which implies injectivity in the case $k=1$.

\begin{lemma}\label{lem:linear_combinations}
Let $Y_1,Y_2,\ldots$ be non-Gaussian i.i.d.\ random variables with symmetric distribution and suppose that $\E |Y_1|^p <\infty$ for all $p\in \N$.
%with a moment generating function $\E \eee^{tY_1}$ being finite for all $t$ in a sufficiently small interval $(-\eps_0,\eps_0)$.
Let also $\xi'\sim N(0, \sigma'^2)$ and $\xi''\sim N(0,\sigma''^2)$ be Gaussian random variables independent of $Y_1,Y_2\ldots$.  Assume that for some sequences $\alpha:=(\alpha_i)_{i\in\N}\in\ell_2$ and $\beta:=(\beta_i)_{i\in\N}\in\ell_2$, we have
\begin{equation}\label{eq:infinite sum equality in distirbution}
\sum_{i\in\N}\alpha_iY_i + \xi' \stackrel{\dint}{=} \sum_{i\in\N} \beta_i Y_i + \xi''.
\end{equation}
Then, $\alpha$ and $\beta$ are equal up to a signed permutation of the entries, and $\sigma'^2=\sigma''^2$.
\end{lemma}
\begin{proof}
Without the terms $\xi'$ and $\xi''$ the lemma was proved by Marcinkiewicz~\cite{M1939}. The following proof is an adaptation of his method.
Since the distributions of the random variables on both sides of~\eqref{eq:infinite sum equality in distirbution} do not change upon applying to $\alpha$ and $\beta$ arbitrary signed permutations of the components, we may assume without loss of generality that $\alpha_1\geq \alpha_2\geq \ldots \geq 0$ and $\beta_1\geq \beta_2\geq \ldots \geq 0$.
Then, our aim is to show that $\alpha_i = \beta_i$ for all $i\in\N$, and $\sigma'^2= \sigma''^2$.
We prove this claim via a comparison of cumulants.
%Let $\varphi$ be the cumulant generating function of the random variable $Y_1$, i.e.,
Let $\psi$ be the logarithm of the characteristic function of $Y_1$, that is,
\[
\psi(t) = \log \E\Big[ e^{\ii t Y_1}\Big], \qquad |t|< \eps_0,
\]
%Note that $\varphi(t)$ is analytic on the strip $|\Re t|< \eps_0$. The cumulants of $Y_1$ are defined as the derivatives of $\varphi$ at $0$:
which is well-defined for $t\in (-\eps_0,\eps_0)$ with $\eps_0>0$ sufficiently small. The sequence of cumulants of $Y_1$ are defined by
$$
\kappa_k(Y_1) := \ii^{-k} \psi^{(k)}(0), \qquad k\in \N.
$$
Taking the log-characteristic functions of both sides of~\eqref{eq:infinite sum equality in distirbution} yields
$$
\sum_{i\in\N} \psi(\alpha_i t) - \frac 12 \sigma'^2 t^2
=
\sum_{i\in\N} \psi(\beta_i t)  - \frac 12 \sigma''^2 t^2.
$$
Due to the square integrability of $\alpha$ and $\beta$ and the estimate $\psi(x) =O(x^2)$, as $x\to 0$, both series are well-defined and converge uniformly on a sufficiently small interval around the origin. Moreover, using the finiteness of moments of $Y_1$, Marcinkiewicz~\cite[Lemme 5]{M1939} has shown that this equality can be differentiated term-wise any number $k\in \N$ of times leading to
\[
\kappa_k(Y_1) \sum_{i\in\N}\alpha_i^k  = \kappa_k(Y_1) \sum_{i\in\N} \beta_i^k , \qquad k\in \{3,4,\ldots\}.
\]
It is a classical result of Marcinkiewicz \cite{M1939} that the log-characteristic function of a non-Gaussian random variable cannot be a finite-degree polynomial (see also \cite[Theorem 7.3.4]{L1970} and \cite{L1958} for a generalization and an elementary proof). Since we excluded the Gaussian case,  the function $\psi$ is not a finite-degree polynomial in $t$ and thus $Y_1$ must have infinitely many non-zero cumulants. It follows that
%Assume that $k\in \{3,4,\ldots\}$ is such that $\kappa_k(Y_1)\neq 0$. Then,
\begin{equation}\label{eq:equality coefficients}
\sum_{i\in\N}\alpha_i^k  = \sum_{i\in\N} \beta_i^k \quad  \text{ for infinitely many } k\in\N.
\end{equation}
%Because, as we argued, there are infinitely many non-zero cumulants, Equation~\eqref{eq:equality coefficients} holds
It remains to argue that this implies $\alpha = \beta$. Recall that $\alpha_1\geq \alpha_2\geq \ldots \geq 0$. Assume that $\alpha_1>0$ since otherwise all $\alpha_i$ vanish and the statement becomes trivial. Consider the set $A_1:=\{i\in\N\,:\, \alpha_i = \alpha_1\}$; note that it is finite since
$\alpha\in\ell_2$ is a null-sequence.  We observe that
\[
\sum_{i\in\N} \alpha_i^k = \alpha_1^k \, \# A_1 + \sum_{i\in\N\setminus A_1} \alpha_i^k.
\]
And so, because we may send $k\to\infty$ along a subsequence (because \eqref{eq:equality coefficients} is valid for infinitely many $k\in\N$),
\[
\frac{\sum_{i\in\N} \alpha_i^k}{\alpha_1^k\,  \# A_1} = 1+ \frac{\sum_{i\in\N\setminus A_1}\alpha_i^k}{\alpha_1^k\, \# A_1} = 1+ \frac{1}{\# A_1} \sum_{i\in\N\setminus A_1}\Big(\frac{\alpha_i}{\alpha_1}\Big)^k \stackrel{k\to\infty}{\longrightarrow} 1,
\]
where the convergence in the last step follows from dominated convergence, since $\lim_{k\to\infty} (\alpha_i/\alpha_1)^k = 0$ componentwise for all $i\in \N \backslash A_1$ and $\sum_{i\in\N \backslash A_1} (\alpha_i/\alpha_1)^k \leq \sum_{i\in\N \backslash A_1} (\alpha_i/\alpha_1)^2 <\infty$ provides a summable majorant. We obtain
\[
\sum_{i\in\N} \alpha_i^k \sim \alpha_1^k \cdot \#A_1, \qquad k\to\infty.
\]
In a similar way, assume $\beta_1>0$, and define $B_1:=\{i\in\N\,:\, \beta_i = \beta_1 \}$, which is again a finite set. We get
\[
\sum_{i\in\N} \beta_i^k \sim \beta_1^k \cdot \#B_1, \qquad k\to\infty.
\]
By~\eqref{eq:equality coefficients}, this means that $\alpha_1^k\cdot \#A_1 \sim \beta_1^k\cdot \# B_1$ as $k\to\infty$,
i.e., $\lim_{k\to\infty} \frac{\alpha_1^k}{\beta_1^k} = \frac{\#B_1}{\#A_1} \not\in\{0,\infty\}$, which implies that $\alpha_1 = \beta_1$ and so $\#A_1=\#B_1$. Now, subtracting from both sides in~\eqref{eq:equality coefficients} the terms equal to $\alpha_1^k=\beta_1^k$ and repeating the procedure, we inductively obtain that $\alpha_i = \beta_i $ for all $i\in\N$. Finally, comparing the variances in~\eqref{eq:infinite sum equality in distirbution}, we obtain $\sigma'^2 = \sigma''^2$.
This completes the proof.
\end{proof}

\begin{rmk}\label{rem:finite_moments_assumption}
One may ask whether the requirement of finite moments can be removed from the statement of Lemma~\ref{lem:linear_combinations}. The answer is ``no'' (which is the reason why we require the finiteness of all moments in Theorem~\ref{thm:ldp multidimensional projection_product_measures}).  In deep of work of Linnik~\cite{linnik_linear_I_and_II}, \cite{linnik_linear_I}, \cite{linnik_linear_II} continued by Zinger~\cite{zinger1}, \cite{zinger2}, see also~\cite[Chapter~2]{kagan_linnik_rao_book}, the following question going back to Marcinkiewicz~\cite{M1939} has been investigated. Suppose that for i.i.d.\ symmetric random variables $Y_1,\ldots,Y_r$ and some vectors $a=(a_1,\ldots,a_r)$ and $b=(b_1,\ldots,b_r)$ that are not signed permutations of each other  we have
\begin{equation}\label{eq:linear_forms_equal_distr}
\sum_{k=1}^r a_k Y_k \stackrel{\dint}{=} \sum_{k=1}^r b_k Y_k.
\end{equation}
Does this imply that $Y_1$ is normal? Marcinkiewicz~\cite{M1939} proved that the answer is ``yes'' if we require \textit{all} moments of $Y_1$ to be finite. However, the finiteness of any fixed moment of  $Y_1$ is in general  not sufficient to conclude normality,  and counterexamples are given in~\cite[\S~55]{linnik_linear_I}. Moreover, a necessary and sufficient condition on $a$ and $b$ under which~\eqref{eq:linear_forms_equal_distr} implies normality of $Y_1$ has been discovered by Linnik~\cite{linnik_linear_I} under an additional assumption $\max\{a_1,\ldots,a_r\} \neq \max\{b_1,\ldots,b_r\}$ that has been subsequently removed by Zinger~\cite{zinger1}, \cite{zinger2}.
\end{rmk}

\begin{rmk}\label{rem:symmetry_assumption}
A simple modification of the above proof shows that for non-symmetric $Y_1$ the same conclusion holds, i.e.,\ $\alpha$ and $\beta$ are equal up to a \textit{signed} permutation of the entries. However, in the non-symmetric setting it is natural to ask whether a stronger conclusion holds that $\alpha$ and $\beta$ are equal up to a \textit{usual} permutation. To prove this stronger statement it would be sufficient to show that a non-symmetric random variable with finite moments has infinitely many non-zero cumulants of \textit{odd} degree. If this claim (which we are not able to prove) is true, then for non-symmetric random variables with zero mean the definition of $\mathbb W_k$ could be modified by removing the requirement of the symmetry of $\mu$, and all remaining steps of the proof of Theorem~\ref{thm:ldp multidimensional projection_product_measures} would apply.
\end{rmk}

\subsection{Step 4 --  Injectivity of the map \texorpdfstring{$\Psi$}{Psi}}\label{sec:injectivity}
The injectivity of the map $\Psi$ follows from the following multidimensional version of Lemma~\ref{lem:linear_combinations}, which can be considered a multidimensional Marcinkiewicz-type result (cf. \cite{M1939}).
\begin{lemma}\label{lem:linear_combinations_multidim}
Fix some $k\in \N$. Let $Y_1,Y_2,\ldots$ be non-Gaussian i.i.d.\ random variables having symmetric distribution  with variance $\sigma^2:= \E [Y_1^2]$ and $\E[|Y_1|^p] <\infty$ for all $p\in \N$.
%and a moment generating function $\E \eee^{tY_1}$ being finite for all $t$ in a sufficiently small interval $(-\eps_0,\eps_0)$.
Let also $\Xi'$ and $\Xi''$ be $k$-dimensional centered Gaussian random vectors independent of $Y_1,Y_2\ldots$.
If, for some
$$
\mu'  =  \sum_{j=1}^{n'} (\delta_{C_j'} + \delta_{-C_j'}) \in \mathbb W_k
\;\;
\text{ and }
\;\;
\mu'' = \sum_{j=1}^{n''} (\delta_{C_j''} + \delta_{-C_j''}) \in \mathbb W_k,
$$
we have
\begin{equation}\label{eq:infinite sum equality in distirbution_multidim}
\sum_{i=1}^{n'} C_i' Y_i +  \Xi' \stackrel{\dint}{=} \sum_{i=1}^{n''} C_i'' Y_i + \Xi'',
\end{equation}
then  $\mu'=\mu''$.
\end{lemma}
\begin{proof}
We need to show that the $k\times\infty$-matrices $V':= V(\mu')$ and $V'':= V(\mu'')$ corresponding to $\mu'$ and $\mu''$ are equal up to a signed permutation of columns.
We already proved the claim for $k=1$ in Lemma~\ref{lem:linear_combinations}. For arbitrary $k\in \N$, Lemma~\ref{lem:linear_combinations} implies that any row of $V'$ differs from a corresponding row of $V''$ by a signed permutation only. However, these permutations may be different for different rows, which is why an additional, more elaborate argument is needed.

Let us prove the lemma for arbitrary $k\in \N$ by induction. Assume that we proved the claim  for elements of $\mathbb W_{k-1}$. Our aim is to prove it for any $\mu',\mu''\in \mathbb W_k$. We are going to apply the induction assumption to the first $k-1$ rows of $V'$ and $V''$.  Let us take some column of $V'$ and write it in the form $(v,x)\in \R^k$ with $v\in \R^{k-1}$ and $x\in \R$. We call this column ``good'' if $v\neq 0$. A column is called ``bad'' if $v=0$ but $x\neq 0$. Otherwise, a column equals $0$.
Let us first remove all bad and zero columns and restrict our attention to good columns only.
Applying the induction assumption to the first $(k-1)$ rows of $V'$ and $V''$ we may assume that, after an appropriate signed permutation of the columns of $V'$ and $V''$, the first $k-1$ rows of the good parts of the matrices coincide. More precisely, the good columns of $V'$ can be written as $C_i' = C_i(V') = (w_i, x_i') \in \R^k$ and the good columns of $V''$ as $C_i'' = C_i(V'') = (w_i, x_i'')\in \R^k$, where $w_i\in \R^{k-1}\backslash\{0\}$ and $x_i',x_i''\in \R$.  There may be repetitions among the vectors $\pm w_1,\pm w_2,\ldots$. Let us take one of these vectors, call it $w$, observe that it is not $0$ (because we removed the bad columns), and consider all other vectors of the form $w_i$ that are equal to $\pm w$.  By square summability of the rows, the number of such vectors is finite. Applying the same signed permutation to the columns of $V'$ and $V''$, we may and will assume that $w := w_1=\ldots = w_\ell \neq 0$ and $w_j \neq \pm w$ for all $j>\ell$. Our aim is to prove that $(x_1',\ldots, x_\ell')$ is an unsigned permutation of $(x_1'',\ldots, x_\ell'')$. The idea is to apply a linear functional to both sides of~\eqref{eq:infinite sum equality in distirbution_multidim}, which is chosen such that the images of the first $\ell$ columns are separated from the images of the other columns. Then we shall apply the one-dimensional Lemma~\ref{lem:linear_combinations} to these functionals to conclude that the values of these functionals are signed permutations of each other, which, by the separation property, implies the claim.  Let $u\in \R^{k-1}$ be some vector satisfying $\langle u, w\rangle \neq \langle u, w_j\rangle$ and $\langle u, w\rangle \neq -\langle u, w_j\rangle$ for all $j>\ell$, as well as $\langle u, w\rangle > 0$. To prove that such a vector exists, note that the set of all $u\in \R^{k-1}$ for which one of the identities $\langle u, w\rangle = \pm \langle u, w_j\rangle$ or $\langle u, w\rangle = 0$ holds true is a union of countably many linear hyperplanes, and therefore cannot be equal to $\R^{k-1}$  because the hyperplanes are Lebesgue zero sets. We can take $u$ to be any vector outside this countable union. Replacing $u$ by $-u$, if necessary, we can assure the condition $\langle u, w\rangle > 0$. Since $w_j\to 0$ and hence also $\langle u, w_j\rangle \to 0$ as $j\to\infty$, we can even find a sufficiently small $r>0$ such that $\langle u, w\rangle > 2r$ and, additionally,
$$
\langle u, w_j \rangle \notin (\langle u, w\rangle - 2r, \langle u, w\rangle + 2r)
\;\;
\text{ and }
\;\;
-\langle u, w_j \rangle \notin (\langle u, w\rangle - 2r, \langle u, w\rangle + 2r)
\;\;
\text{ for all }
\;\;
j>\ell.
$$
Since both $x_j'$ and $x_j''$ converge to $0$ as $j\to\infty$ by square summability, we can find a sufficiently small $\eps>0$ such that
$$
\langle u,w \rangle + \eps x_m' \in (\langle u, w\rangle - r, \langle u, w\rangle + r)
\;\;
\text{ and }
\langle u,w \rangle + \eps x_m'' \in (\langle u, w\rangle - r, \langle u, w\rangle + r)
\;\;
\text{ for all }
\;\;
m\in \{1,\ldots, \ell\}
$$
and at the same time
$$
\pm (\langle u, w_j \rangle + \eps x_j') \notin (\langle u, w\rangle - r, \langle u, w\rangle + r)
\;\;
\text{ and }
\;\;
\pm (\langle u, w_j \rangle + \eps x_j'') \notin (\langle u, w\rangle - r, \langle u, w\rangle + r)
\;\;
\text{ for all }
\;\;
j>\ell.
$$
Now, let us consider the linear functional $L:\R^k \to\R$ defined by $L(v, x) = \langle u, v\rangle + \eps x$ for $v\in \R^{k-1}$ and $x\in \R$. Applying this functional to both sides of the distributional equality~\eqref{eq:infinite sum equality in distirbution_multidim}, we obtain
$$
\sum_{i=1}^\infty (\langle u, w_i\rangle  + \eps x_i') Y_i+\xi' \stackrel{\dint}{=} \sum_{i=1}^\infty (\langle u, w_i\rangle  + \eps x_i'') Y_i +\xi''
$$
for some centered Gaussian random variables $\xi'$ and $\xi''$ independent of $Y_1,Y_2,\ldots$. Applying to this identity Lemma~\ref{lem:linear_combinations}, we conclude  that the vectors
$$
(\langle u, w_i\rangle  + \eps x_i')_{i\in \N}
\;\;
\text{ and }
\;\;
(\langle u, w_i\rangle  + \eps x_i'')_{i\in \N}
$$
are signed permutations of each other. However, by the above construction, the first $\ell$ entries of these vectors belong to the interval $(\langle u, w\rangle - r, \langle u, w\rangle + r)\subset (0,\infty)$, whereas all other entries (as well as their negatives) are located outside this interval. It follows that the first $\ell$ entries of these vectors are equal up to an unsigned permutation. Consequently, $(x_1',\ldots,x_\ell')$ and $(x_1'',\ldots, x_\ell'')$ are equal up to an unsigned permutation. This means that the first $\ell$ columns of $V'$ are equal to the first $\ell$ columns of $V''$.  Since these considerations hold for every $w\in \R^{k-1}\backslash\{0\}$, the above argument proves that the good columns of $V'$ become equal to the good columns  of $V''$ after applying a suitable signed permutation.

Let us now take into consideration the bad columns, too.
%Assume first that the number of non-zero columns in $V'$ is finite.
The above argument shows\footnote{Note that the presence of bad columns does not influence the validity of the argument because we can choose $\eps\in (0,r)$ which ensures that $|\eps y|< r$ for every bad column $(0,y)$ since $|y|\leq 1$.} that the good columns of both matrices coincide, after a signed permutation.   That is,  we can write the good columns of $V'$ and $V''$ as $(w_j, x_j)$, $j\in J$, for some index set $J\subset \N$. Besides, $V'$ may have bad columns which are denoted by $(0, y_i')$, $i\in I'$, with $y_i'\neq 0$. Similarly, the bad columns of $V''$ are denoted by $(0, y_i'')$, $i\in I''$, with $y_i''\neq 0$. Lemma~\ref{lem:linear_combinations} applied to the $k$-th coordinate in~\eqref{eq:infinite sum equality in distirbution_multidim} yields the following statement: for every $c\in\R$, the number of occurrences of $\pm c$ in the last row of $V'$ is the same as in the last row of $V''$.  If $c\neq 0$, then the number of occurrences of $\pm c$ in the entries belonging to the good columns is finite  (by square summability) and this number is the same for $V'$ and $V''$. Since the good columns of $V'$ and $V''$ coincide, the number of occurrences of $\pm c$ in the \textit{bad} columns in the last row of $V'$ and $V''$ is the same. That is, the number of times  a bad column of the form  $(0,\pm c)$ appears in $V'$ is the same as for $V''$. This holds for every $c\in \R\backslash\{0\}$. After applying a suitable signed permutation, the good and the bad columns of $V'$ become equal to the good and the bad columns of $V''$. The remaining columns, if there are any,  are $0$.  Hence, $\mu' = \mu''$.
\end{proof}

As an immediate consequence of the above lemma we record the following result.

\begin{cor}\label{cor:nu_injective}
The map $\Psi : \mathbb W_k \to \mathcal M_1(\R^k)$ is injective.
\end{cor}
\subsection{Step 5 -- Vague convergence in \texorpdfstring{$\mathbb W_k$}{W\_k} and weak convergence in \texorpdfstring{$\mathcal M_1(\R^k)$}{M\_1(Rk)}}
% % % % % % % % % % % % % % % % % % % % % % % % % %
%Recall that for $k\in\N$,
%\[
%W_k = \Bigg\{ \mu:=\sum_{j=1}^{N}(\delta_{C_j} + \delta_{-C_j})\,:\, N\in\N\cup\{0,\infty\}, \mu \text{ symmetric locally finite on %} [-1,1]^k\setminus\{0\},\, (C_j)_{j\in\N}=V \in \widetilde{W}_k  \Bigg\},
%\]
%where the space $\widetilde{W}_k$ of real $k\times \infty$ matrices is given by
%\[
%\widetilde{W}_k = \Big\{ V=(V(i,j))_{i=1,j=1}^{k,\infty}\in \R^{k\times \infty}\,:\,\|(V(i,j))_{j\in\N}\|_2 <\infty \,\forall %i\in\{1,\dots,k\}, \, \|VV^*\|\leq 1\Big\}
%\]
%and endowed with the topology of pointwise convergence.
%We will now show that vague convergence in the space $W_k$ of point processes and weak convergence of the corresponding laws as %defined in \eqref{eq:rep_law_sum_p_gauss_product_meas} are equivalent.

The next lemma states that $\Psi:\mathbb W_k \to \mathcal M_1(\R^k)$ is a homeomorphism onto its image $\mathcal K_{k, Y_1}:= \Psi(\mathbb W_k)$, which is a compact subset of $\mathcal M_1(\R^k)$.

\begin{lemma}\label{lem:vague vs weak convergence}
Let $k\in\N$ be fixed.
Assume that $Y_1,Y_2,\ldots$ are non-Gaussian i.i.d.\ random variables having symmetric distribution  with variance $\sigma^2:= \E [Y_1^2]$ and $\E[|Y_1|^p] <\infty$ for all $p\in \N$. Recall that the map $\Psi: \mathbb W_k \to \mathcal M_1(\R^k)$ is defined as follows: for $\mu=\sum_{j=1}^N (\delta_{C_j}+\delta_{-C_j})\in \mathbb W_k$, we put
\[
\Psi(\mu) := \text{Law}\,\Big(\sum_{j=1}^N C_jY_j + \sigma \left(\id_{k\times k} - V(\mu)V(\mu)^*\right)^{1/2} N_k\Big),
\]
where $N_k$ is a $k$-dimensional standard Gaussian random vector independent of the sequence $(Y_i)_{i\in\N}$.
Then a sequence $\mu_1,\mu_2,\ldots \in \mathbb W_k$ converges vaguely to $\mu\in \mathbb W_k$ if and only if $\Psi(\mu_n)$ converges weakly to $\Psi(\mu)$ in $\mathcal M_1(\R^k)$, as $n\to\infty$.
\end{lemma}

\begin{proof}
Assume $\mu_n \to \mu$ vaguely in $\mathbb W_k$ as $n\to\infty$. We prove the weak convergence of $\Psi(\mu_n)$ to $\Psi(\mu)$ via characteristic functions. Let us denote by $\varphi_Y(s) = \E[e^{isY_1}]$ the characteristic function of $Y_1$. Recall that the characteristic function of $N_k$ is given by $\varphi_{N_k}(t) = e^{-\frac 12 \langle t,t\rangle }$ and observe that if $M$ is a real symmetric $k\times k$ matrix, then the characteristic function of $MN_k$ is given by
\[
\varphi_{N_k}(M^*t) = e^{-\frac{1}{2}\langle M^*t,M^* t \rangle}=e^{-\frac{1}{2}\langle t,MM^* t \rangle} = e^{-\frac{1}{2}\langle t,M^2 t \rangle}.
\]
The characteristic function $\varphi$ of $\Psi(\mu)$ is, for any $t\in\R^k$, given by
\begin{align*}
\varphi(t) & = \int_{\R^k} e^{i \langle t,y  \rangle } \,\Psi(\mu)(\dint y) \cr
& = \E\Bigg[e^{i\big\langle t, \sum_{j=1}^N C_jY_j + \sigma(\id_{k\times k} - V(\mu)V(\mu)^*)^{1/2}N_k  \big\rangle} \Bigg] \cr
& = \E\Big[e^{i\langle t, \sigma(\id_{k\times k}-V(\mu)V(\mu)^*)^{1/2}N_k \rangle}\Big] \prod_{j=1}^N \E\Big[e^{i\langle t,C_jY_j \rangle}\Big] \cr
& = e^{-\frac{1}{2}\langle t, \sigma^2(\id_{k\times k}-V(\mu)V(\mu)^*)t \rangle}  \prod_{j=1}^N \varphi_Y\big(\langle t,C_j\rangle\big).
\end{align*}
Note that the previous arguments are clear for $N<\infty$ and if $N=\infty$, then the step from second to third line holds, because the series $\sum_{j=1}^N C_jY_j$ converges almost surely (and thus weakly) by the $L_2$-version of the bounded martingale convergence theorem (see, e.g., \cite[Theorem 11.10]{Klenke2014}) as all rows of the matrix $V(\mu)=(C_j)_{j=1}^N$ belong to $\ell_2$ and $Y_1$ has variance $\sigma^2<\infty$. More precisely, since $Y_1,Y_2,\dots$ are independent and centered, for any $1\leq i \leq k$,
\[
 \sup_{m\in\N} \E\Big[\Big| \sum_{j=1}^mY_jC_j(i) \Big|^2\Big]
   = \sup_{m\in\N} \Var\Big[\sum_{j=1}^mY_jC_j(i)\Big] = \sigma^2 \sum_{j=1}^\infty |C_j(i)|^2  < \infty,
\]
and so we can apply L\'evy's continuity theorem to obtain the desired equality. Similarly, for any $n\in\N$, the characteristic function $\varphi_n$ of $\Psi(\mu_n)$ is given by
\[
\varphi_n(t) = e^{-\frac{1}{2}\langle t, \sigma^2(\id_{k\times k}-V(\mu_n)V(\mu_n)^*)t \rangle}  \prod_{j=1}^{K_n} \varphi_Y\big(\langle t,C_j^{(n)}\rangle\big),
\]
where $V(\mu_n)$ is the matrix with columns $C_j^{(n)}\in\R^k$ with
\[
\mu_n = \sum_{j=1}^{  K_n  }\left(\delta_{C_j^{(n)}} + \delta_{-C_j^{(n)}}\right),\qquad K_n \in\N\cup\{0,\infty\}.
\]
We shall now show that, as $n\to\infty$, $\varphi_n(t)\to\varphi(t)$ for any $t\in\R^k$. Fix $t\in\R^k$ and let $\varepsilon>0$ be such that our measure $\mu=\sum_{j=1}^N (\delta_{C_j}+\delta_{-C_j})\in \mathbb W_k$ satisfies
\[
\mu\big(\partial \B_2^k(0,\varepsilon)\big) = 0,
\]
i.e., no point $C_j\in\R^k$, $j\in\N$, lies on the Euclidean sphere of radius $\varepsilon$ around zero. Note that because the rows of our matrix $V(\mu)=(C_j)_{j\in\N}$ belong to $\ell_2$, there are at most finitely many points outside of the ball $\B_2^k(0,\varepsilon)$ (otherwise we would find a row which is not in $\ell_2$), i.e., we may assume that for some $m=m(\varepsilon)<\infty$
\[
\pm C_1,\dots, \pm C_m \notin \B_2^k(0,\varepsilon)
\]
and, for all $j>m$, $\pm C_j \in \B_2^k(0,\varepsilon)$. We now apply \cite[Proposition 3.13, p. 144]{R2008} with $K=\B_2^k(0,\varepsilon)$ there, which yields a formulation of vague convergence in terms of convergence of the points that define the respective point measures. We obtain the following: for all $n\geq N(\varepsilon) \in\N$, the number of points $\pm C_j^{(n)}$, $j\in\N$, defining the point measure $\mu_n$, which lie outside of $\B_2^k(0,\varepsilon)$ is also equal to $2m$. Moreover, we may relabel those points as $\pm C_1^{(n)},\dots,\pm C_m^{(n)}$ such that, for all $j\in\{1,\dots,m\}$, we have the pointwise convergence $\pm C_j^{(n)} \to \pm C_j$ in $\R^k$ as $n\to\infty$. This implies that
\begin{equation}\label{eq:product 1 to m}
\prod_{j=1}^m \varphi_Y\big(\langle t,C_j^{(n)}\rangle\big) \stackrel{n\to\infty}{\longrightarrow} \prod_{j=1}^m \varphi_Y\big(\langle t,C_j\rangle\big).
\end{equation}
Since
\[
\varphi(t) = e^{-\frac{1}{2}\langle t, \sigma^2(\id_{k\times k}-V(\mu)V(\mu)^*)t \rangle}  \prod_{j=1}^m \varphi_Y\big(\langle t,C_j\rangle\big)\prod_{j=m+1}^N \varphi_Y\big(\langle t,C_j\rangle\big)
\]
and
\[
\varphi_n(t) = e^{-\frac{1}{2}\langle t, \sigma^2(\id_{k\times k}-V(\mu_n)V(\mu_n)^*)t \rangle}  \prod_{j=1}^m \varphi_Y\big(\langle t,C_j^{(n)}\rangle\big)\prod_{j=m+1}^{ K_n} \varphi_Y\big(\langle t,C_j^{(n)}\rangle\big)
\]
it is left to establish the convergence of the remaining parts defining $\varphi_n$. Since $\E[|Y_1|^p] <\infty$ for all $p\in \N$, the characteristic function $\varphi_Y$ of $Y_1$ is also $p$-times differentiable for all $p\in\N$. Since the random variable $Y_1$ is symmetric (and thus centered), we have
\[
\varphi_Y(s) = \sum_{k=0}^2\frac{\varphi_Y^{(k)}(0)}{k!} s^k + R_3(s) = 1-\sigma^2\frac{ s^2}{2} + O(s^4)
\]
for $s\to 0$. Therefore, as $s\to 0$,
\[
\log \varphi_Y(s) = -\sigma^2\frac{s^2}{2} +O(s^4).
\]
This means that for every $n\geq N(\eps)$,
\begin{align*}
\log \prod_{j=m+1}^{K_n} \varphi_Y\big(\langle t,C_j^{(n)}\rangle\big) & = \sum_{j=m+1}^{K_n} \log \varphi_Y\big(\langle t,C_j^{(n)}\rangle\big)
= \sum_{j=m+1}^{K_n} \left(-\frac{\sigma^2}{2}\langle t,C_j^{(n)}\rangle^2 +O\Big(\langle t,C_j^{(n)}\rangle^4\Big)\right).
\end{align*}
Note that it follows from the Cauchy-Schwarz inequality and the fact that the rows of the matrix $V(\mu_n)=(C_j^{(n)})_{j=1}^{\infty}$ are square summable (which implies $\|C_j^{(n)}\|_2\to 0$ as $j\to\infty$) that
\[
\langle t,C_j^{(n)}\rangle^4 \leq \|t\|_2^4\|C_j^{(n)}\|_2^4 \stackrel{j\to\infty}{\longrightarrow} 0
\]
as $j\to\infty$, and so the constant implicit in the $O$-terms is uniform, because we may work with an $\varepsilon>0$ above such that we only consider $j\geq m$ with $m=m(\varepsilon)$ sufficiently large. Thus, we have
\begin{align*}
\log \prod_{j=m+1}^{K_n} \varphi_Y\big(\langle t,C_j^{(n)}\rangle\big)
& =
-\frac{\sigma^2}{2}\Bigg(\sum_{j=m+1}^{K_n} \langle t,C_j^{(n)}\rangle^2\Bigg) +O\Bigg(\sum_{j=m+1}^{K_n}\|C_j^{(n)}\|_2^4\Bigg) \cr
& =
-\frac{\sigma^2}{2}\Bigg(\sum_{j=1}^{K_n} \langle t,C_j^{(n)}\rangle^2 - \sum_{j=1}^m \langle t,C_j^{(n)}\rangle^2\Bigg) +O\Bigg(\sum_{j=m+1}^{K_n}\|C_j^{(n)}\|_2^4\Bigg) \cr
&=
-\frac{\sigma^2}{2} \langle V(\mu_n)V(\mu_n)^*t,t\rangle +\frac{\sigma^2}{2}\sum_{j=1}^m \langle t,C_j^{(n)}\rangle^2 +O\Bigg(\sum_{j=m+1}^{K_n}\|C_j^{(n)}\|_2^4\Bigg),
\end{align*}
where we used that $\sum_{j=1}^{K_n}\langle t,C_j^{(n)}\rangle^2 = \langle V(\mu_n)V(\mu_n)^*t,t\rangle$. Therefore, we obtain
\begin{align*}
&\log \Bigg(e^{-\frac{1}{2}\langle t, \sigma^2(\id_{k\times k}-V(\mu_n)V(\mu_n)^*)t \rangle}  \prod_{j=m+1}^{K_n} \varphi_Y\big(\langle t,C_j^{(n)}\rangle\big)\Bigg)\\
& = \log \Bigg(e^{-\frac{1}{2}\langle t, \sigma^2(\id_{k\times k}-V(\mu_n)V(\mu_n)^*)t \rangle}\Bigg)\\
&\quad -\frac{\sigma^2}{2} \langle V(\mu_n)V(\mu_n)^*t,t\rangle +\frac{\sigma^2}{2}\sum_{j=1}^m \langle t,C_j^{(n)}\rangle^2 +O\Bigg(\sum_{j=m+1}^{K_n}\|C_j^{(n)}\|_2^4\Bigg) \cr
& = -\frac{\sigma^2}{2}\langle t,t \rangle +\frac{\sigma^2}{2}\sum_{j=1}^m \langle t,C_j^{(n)}\rangle^2 +O\Bigg(\sum_{j=m+1}^{K_n}\|C_j^{(n)}\|_2^4\Bigg).
\end{align*}
A similar computation shows that
\begin{align*}
\log\Bigg(e^{-\frac{1}{2}\langle t, \sigma^2(\id_{k\times k}-V(\mu)V(\mu)^*)t \rangle}  \prod_{j=m+1}^N \varphi_Y\big(\langle t,C_j\rangle\big)\Bigg)
& = -\frac{\sigma^2}{2}\langle t,t \rangle +\frac{\sigma^2}{2}\sum_{j=1}^m \langle t,C_j\rangle^2 +O\Bigg(\sum_{j=m+1}^N\|C_j\|_2^4\Bigg).
\end{align*}
Now we deal with the error terms. For some absolute constant $c\in(0,\infty)$, because for all $j>m$, $\pm C_j \in \B_2^k(0,\varepsilon)$, we have
\[
O\Bigg(\sum_{j=m+1}^N\|C_j\|_2^4\Bigg) \leq c \sum_{j=m+1}^N\|C_j\|_2^4 = c \sum_{j=m+1}^N\|C_j\|_2^2 \|C_j\|_2^2 \leq c \sum_{j=m+1}^N\varepsilon^2\|C_j\|_2^2 \leq c\varepsilon^2 k,
\]
where in the last step we used that the rows of $V(\mu)=(C_j)_{j=1}^N$ have $\ell_2$-norm at most $1$. Similarly, and uniformly in $n$,
\[
O\Bigg(\sum_{j=m+1}^{K_n}\|C_j^{(n)}\|_2^4\Bigg) \leq \widetilde{c} \varepsilon^2 k,
\]
where $\widetilde{c}\in(0,\infty)$ is another absolute constant. Therefore, we obtain that for some absolute constant $C\in(0,
\infty)$
\begin{align*}
-2C\varepsilon^2 k & \leq \liminf_{n\to\infty}\, \log \Bigg(\frac{e^{-\frac{1}{2}\langle t, \sigma^2(\id_{k\times k}-V(\mu_n)V(\mu_n)^*)t \rangle}  \prod_{j=m+1}^{K_n} \varphi_Y\big(\langle t,C_j^{(n)}\rangle\big)}{e^{-\frac{1}{2}\langle t, \sigma^2(\id_{k\times k}-V(\mu)V(\mu)^*)t \rangle}  \prod_{j=m+1}^N \varphi_Y\big(\langle t,C_j\rangle\big)}\Bigg) \cr
& \leq \limsup_{n\to\infty}\, \log \Bigg(\frac{e^{-\frac{1}{2}\langle t, \sigma^2(\id_{k\times k}-V(\mu_n)V(\mu_n)^*)t \rangle}  \prod_{j=m+1}^{K_n} \varphi_Y\big(\langle t,C_j^{(n)}\rangle\big)}{e^{-\frac{1}{2}\langle t, \sigma^2(\id_{k\times k}-V(\mu)V(\mu)^*)t \rangle}  \prod_{j=m+1}^N \varphi_Y\big(\langle t,C_j\rangle\big)}\Bigg) \leq 2C\varepsilon^2 k.
\end{align*}
Taking the convergence \eqref{eq:product 1 to m} into account, we thus obtain
\begin{align*}
e^{-2C\varepsilon^2 k} & \leq \liminf_{n\to\infty} \Bigg(\frac{e^{-\frac{1}{2}\langle t, \sigma^2(\id_{k\times k}-V(\mu_n)V(\mu_n)^*)t \rangle}  \prod_{j=1}^{K_n} \varphi_Y\big(\langle t,C_j^{(n)}\rangle\big)}{e^{-\frac{1}{2}\langle t, \sigma^2(\id_{k\times k}-V(\mu)V(\mu)^*)t \rangle}  \prod_{j=1}^N \varphi_Y\big(\langle t,C_j\rangle\big)}\Bigg) \cr
& \leq \limsup_{n\to\infty}\Bigg(\frac{e^{-\frac{1}{2}\langle t, \sigma^2(\id_{k\times k}-V(\mu_n)V(\mu_n)^*)t \rangle}  \prod_{j=1}^{K_n} \varphi_Y\big(\langle t,C_j^{(n)}\rangle\big)}{e^{-\frac{1}{2}\langle t, \sigma^2(\id_{k\times k}-V(\mu)V(\mu)^*)t \rangle}  \prod_{j=1}^N \varphi_Y\big(\langle t,C_j\rangle\big)}\Bigg) \leq e^{ 2C\varepsilon^2 k}.
\end{align*}
Since the latter chain of inequalities holds for any $\varepsilon>0$ small enough such that $\mu\big(\partial \B_2^k(0,\varepsilon)\big) = 0$, considering a sequence of such $\varepsilon$ tending to $0$, the desired pointwise convergence of the characteristic functions follows, i.e., $\lim_{n\to\infty}\varphi_n(t)=\varphi(t)$.

Note in passing that the above proof shows absolute convergence of the series
$\sum_{j=1}^N \log \varphi_Y (\langle t,C_j\rangle)$ (if $N=\infty$),
which proves that the distribution of the sum $\sum_{j=1}^N C_j Y_j$ stays invariant under arbitrary signed permutations of the summands.

We have shown that $\Psi: \mathbb W_k \to \mathcal M_1(\R^k)$ is continuous. Since this map is also injective by Corollary~\ref{cor:nu_injective} and $\mathbb W_k$ is compact by Lemma~\ref{lem:W_k_compact}, we conclude that $\Psi$ is a homeomorphism onto its image (which is also compact). Here we  used that a continuous, bijective mapping between a compact space and a Hausdorff space has continuous inverse and  hence is a homeomorphism.
\end{proof}

As a consequence of Lemma~\ref{lem:vague vs weak convergence} we record the following result.

\begin{cor}\label{cor:nu_homeo}
The map $\Psi : \mathbb W_k \to \mathcal M_1(\R^k)$ is a homeomorphism between $\mathbb W_k$ and its image $\mathcal K_{k, Y_1} = \Psi (\mathbb W_k)$, which is a compact subset of $\mathcal M_1(\R^k)$ endowed with the topology of weak convergence.
\end{cor}

\subsection{Step 6 -- LDP on \texorpdfstring{$\mathbb W_k$}{W\_k}}
The next proposition states an LDP on the space $\mathbb W_k$ for the columns of the random Stiefel matrix. As we shall argue in the next step, it implies the LDP stated in Theorem~\ref{thm:ldp multidimensional projection_product_measures} by the contraction principle (e.g.,\cite[Theorem 4.2.1]{DZ2010}) applied to the homeomorphic mapping $\Psi$ .

\begin{proposition}\label{prop: LDP_for_nu_n}
Let $V_{k,n}$ be uniformly distributed on the Stiefel manifold $\mathbb V_{k,n}$. Denote the columns of $V_{k,n}$ by $C_1(V_{k,n}),\ldots,C_n(V_{k,n})$ and consider the following  random element in $\mathbb W_k$:
$$
\eta_n = \sum_{j=1}^n \left(\delta_{C_j(V_{k,n})} + \delta_{-C_j(V_{k,n})}\right).
$$
Then, the sequence $\eta_n$, $n\geq k$, satisfies an LDP on the compact space $\mathbb W_k$ with speed $n$ and a good rate function $\mathbb J: \mathbb W_k\to[0,+\infty]$ defined by
\begin{align}\label{eq:J_inform_funct_def}
\mathbb J(\mu) =
\begin{cases}
-\frac 12 \log \det (\id_{k\times k} - V(\mu)V(\mu)^*) & :\, \|V(\mu)V(\mu)^*\|<1 \\
+\infty &:\, \|V(\mu)V(\mu)^*\|=1.
\end{cases}
\end{align}
\end{proposition}
\begin{proof}
We shall again draw on Proposition~\ref{prop:basis topology} and work on a base of the topology. Hence, we start by describing an explicit base for the vague topology on $\mathbb W_k$. We denote by $\mathbb B(x, r)$  the open ball (in the metric space $[-1,1]^k$ endowed with the induced Euclidean metric) which has radius $r$ and is centered at $x\in \R^k$. Let $\bar {\mathbb B}(x,r)$ be the  closure of $\mathbb B(x, r)$.

\vspace*{2mm}
\noindent
\textit{Base of topology.} Take some $\mu\in \mathbb W_k$. We shall construct a basis of open neighborhoods of $\mu$ as follows. Take some number $r>0$ and assume that there are no atoms of $\mu$ on the sphere $\{x\in \R^k: \|x\|_2 = r\}$. The restriction of $\mu$ to the complement of the  ball $\bar {\mathbb  B}(0,r)$ can be written as
$$
\mu|_{[-1,1]^k \backslash \bar {\mathbb  B}(0,r)} = \sum_{i=1}^L \left(m_i \delta_{D_i} + m_i \delta_{-D_i}\right)
$$
with the following convention: the total number of atoms of $\mu$ in $\{\|x\|_2>r\}$ is $2m$ with $m:= m_1+\ldots + m_L$. The atoms are located at positions $\pm D_1,\ldots, \pm D_L\in [-1,1]^k \backslash \bar {\mathbb B}(0,r)$, for some $L\in \mathbb N_0$, and the multiplicity of the atom at $D_i$ is equal to $m_i\in \N$, as is the multiplicity of the atom at $-D_i$. Let now $\rho>0$ be so small that the $2L$ balls $\bar{\mathbb B}(D_i, \rho)$ and $\bar{\mathbb B}(-D_i, \rho)$, for $i=1,\ldots, L$, are pairwise disjoint and contained in $[-1,1]^k \backslash \bar {\mathbb B}(0,r)$. Then, we denote by $W_{r,\rho}(\mu)$ the set of all $\mu'\in \mathbb W_k$ such that the following two conditions are satisfied:
\begin{itemize}
\item $\mu'$ has $m_i$ atoms in the ball $\mathbb B(D_i, \rho)$ and $m_i$ atoms in $\mathbb B(-D_i, \rho)$, for every $i=1,\ldots, L$.
\item $\mu'$ has no other atoms in $[-1,1]^k \backslash {\mathbb B}(0,r)$, that is, $\mu'([-1,1]^k \backslash {\mathbb B}(0,r)) = 2m= \mu ([-1,1]^k \backslash {\mathbb B}(0,r))$.
\end{itemize}
Note that the number of the atoms of $\mu'$ in the ball $\mathbb B(0, r)$ (as well as their positions) may be arbitrary. The sets of the form $W_{r,\rho}(\mu)$ with $r>0$ and $\rho>0$ satisfying the conditions listed above form a base of open neighborhoods of $\mu$ for the vague topology on $\mathbb W_k$. This is a well-known characterization of vague convergence; see~\cite[Proposition~3.13]{R2008}.

\vspace*{2mm}
\noindent
\textit{Upper bound.}
In the following we shall verify conditions of Proposition~\ref{prop:basis topology} for the base of topology described above. We start with the (simpler) upper bound. Take some $\mu\in \mathbb W_k$. Our aim is to prove that
\begin{equation}\label{eq:upper_bound_theo_D_statement}
\inf_{r>0, \rho>0} \limsup_{n\to\infty} \frac 1 {n} \log \Pro \left[\eta_n \in W_{r,\rho}(\mu)\right] \leq - \mathbb J(\mu).
\end{equation}
If $\mu=0$, then $\mathbb J(\mu)=0$ and there is nothing to prove. In the following let $\mu\neq 0$. For sufficiently small $r>0$ and $\rho>0$, we would like to bound from above the probability that $\eta_n \in W_{r,\rho}(\mu)$, where $W_{r,\rho}(\mu)$ is defined as above. Recall that $V_{k,n}$ denotes the random Stiefel matrix and let $A_{m}\in \R^{k\times m}$ be the $k\times m$-matrix with columns $C_1(V_{k,n}), \ldots, C_{m}(V_{k,n})$. By the union bound,
\begin{equation}\label{eq:union_bound}
\Pro \left[\eta_n \in W_{r,\rho}(\mu)\right]
\leq
2^L \binom {n}{m_1,\ldots, m_L, n-m_1-\ldots-m_L} \Pro[A_{m} \in O_{r,\rho}(\mu)],
\end{equation}
where $O_{r,\rho}(\mu)$ is the set of $k\times m$-matrices $A\in \R^{k \times m}$ whose columns, denoted by $C_1(A), \ldots C_m(A)$, satisfy the following conditions:
\begin{itemize}
\item $C_i(A)\in \mathbb B(D_1, \rho)$ for all $i=1,\ldots m_1$,
\item $C_i(A)\in \mathbb B(D_2, \rho)$ for all $i=m_1+1,\ldots, m_1+m_2$,
\item $\ldots$
\item $C_i(A)\in \mathbb B(D_L, \rho)$ for all $i=m_1+\ldots+ m_{L-1}+1,\ldots, m_1+\ldots + m_L$.
\end{itemize}
With other words, the set $O_{r,\rho}(\mu)$ is the Cartesian product of Euclidean balls $\otimes_{i=1}^L (\mathbb B(D_i, \rho))^{m_i}$.
The density of the random matrix $A_{m}$ with respect to the Lebesgue measure on $\R^{k\times m}$ is known from~\eqref{eq:V_k_n_decomposition} and~\eqref{eq:V_k_n_submatrix_density} to be
\begin{equation}\label{eq:V_k_n_submatrix_density_repeat}
f_n(A) = \frac{\Gamma_k(\frac{n}{2})}{\pi^{\frac{km}{2}}\Gamma_k(\frac{n-m}{2})}\det\Big(\id_{k\times k} - AA^*\Big)^{\frac{n-m-k-1}{2}},
\qquad
A\in\R^{k\times m},
\qquad
\|AA^*\| \leq 1,
\end{equation}
where $\Gamma_k$ is the multivariate Gamma function given by~\eqref{eq:gamma_function_generalized}.  Combining~\eqref{eq:union_bound} and~\eqref{eq:V_k_n_submatrix_density_repeat}, we arrive at the estimate
$$
\Pro \left[\eta_n \in W_{r,\rho}(\mu)\right] \leq \kappa(n) \int_{O_{r,\rho}(\mu) \cap \{A\in \R^{k\times m}:\, \|AA^*\|\leq 1\}} \det\Big(\id_{k\times k} - AA^*\Big)^{\frac{n-m-k-1}{2}} \dint A,
$$
where $\kappa(n)$ is a certain explicit factor satisfying $\lim_{n\to\infty} \frac 1n \log \kappa(n) = 0$ and is therefore negligible at logarithmic speed $n$.  Estimating the integral by the supremum, we get
$$
\Pro \left[\eta_n \in W_{r,\rho}(\mu)\right] \leq \kappa(n) \cdot \vol_{k\times m}(O_{r,\rho}(\mu))\cdot \left(\sup_{A\in O_{r,\rho}(\mu), \, \|AA^*\|\leq 1} \det\Big(\id_{k\times k} - AA^*\Big)\right)^{\frac{n-m-k-1}{2}},
$$
where $\vol_{k\times m}$ denotes the Lebesgue volume in $\R^{k\times m}$. Taking the logarithm, dividing by $n$, and taking the $\limsup_{n\to\infty}$, we arrive at
$$
\limsup_{n\to\infty} \frac 1n \log \Pro \left[\eta_n \in W_{r,\rho}(\mu)\right]
\leq \frac 12 \sup_{A\in O_{r,\rho}(\mu),\,  \|AA^*\|\leq 1}  \log \det\Big(\id_{k\times k} - AA^*\Big),
$$
where we used that $O_{r,\rho}(\mu)$ has volume not depending on $n$ and, in particular,
$$
\limsup_{n\to\infty} \frac 1n \log \vol_{k\times m}(O_{r,\rho}(\mu)) = 0.
$$
We are now able to complete the proof of the upper bound~\eqref{eq:upper_bound_theo_D_statement}.
Let $\eps>0$ be given. Let first the number of atoms of $\mu$ be infinite. For a sufficiently small $r>0$ (meaning that $m$, which is a function of $r$, is sufficiently large), the $k\times m$-matrix $A_{m}$ is such that $\det (\id_{k\times k} - AA^*)$ is within distance $\eps/2$ from $\det (\id_{k\times k} - V_{k,n}V_{k,n}^*)$; see~\eqref{eq:rate_stiefel_supremum}. If the total number of atoms of $\mu$ is finite, then we can even achieve the exact equality of both determinants by choosing $r$ smaller than the smallest norm of an atom of $\mu$.  Now, in both cases, we use the continuity of the function $A\mapsto \det (\id_{k\times k} - AA^*)$ to choose $\rho>0$ so small that $\det (\id_{k\times k} - AA^*)$ is within distance $\eps/2$ from $\det (\id_{k\times k} - A_mA_m^*)$ for every $A\in O_{r,\rho}(\mu)$. By the triangle inequality, we obtain
$$
\sup_{A\in O_{r,\rho}(\mu),\,  \|AA^*\|\leq 1} \det\Big(\id_{k\times k} - AA^*\Big) \leq  \eps + \det (\id_{k\times k} - V_{k,n}V_{k,n}^*).
$$
Since $\eps>0$ was arbitrary, this completes the proof of~\eqref{eq:upper_bound_theo_D_statement}.

\vspace*{2mm}
\noindent
\textit{Lower bound.}
We now prove the lower bound of Proposition~\ref{prop:basis topology}.  Take some $\mu\in \mathbb W_k$. Our aim is to prove that for all $r>0$, $\rho>0$,
\begin{equation}\label{eq:lower_bound_theo_D_statement}
\liminf_{n\to\infty} \frac 1 {n} \log \Pro \left[\eta_n \in W_{r,\rho}(\mu)\right] \geq - \mathbb J(\mu).
\end{equation}
If $\|V(\mu) V(\mu)^*\| = 1$, then $\mathbb J(\mu) = +\infty$ by~\eqref{eq:J_inform_funct_def} and there is nothing to prove. Hence, we let in the following $\|V(\mu) V(\mu)^*\| < 1$. Using the notation introduced in the proof of the upper bound, we observe that in order for the event $\eta_n \in W_{r,\rho}(\mu)$ to occur, it is sufficient that the first $m$ columns of $V_{k,n}$ (which form the random matrix $A_m$) belong to the balls $\mathbb B(D_i, \rho)$  (where $i=1,\ldots, L$ and the $i$-th ball is counted $m_i$ times), whereas all other columns have a norm less than $r$. More precisely, we  can write
\begin{align*}
\Pro \left[\eta_n \in W_{r,\rho}(\mu)\right]
&\geq
\Pro\left[A_m \in O_{r,\rho}(\mu), \sup_{i=1,\ldots,n-m} \|C_{m+i}(V_{k,n})\|_2 < r\right]\\
&=
\int_{O_{r,\rho}(\mu) \cap \{A\in \R^{k\times m}:\, \|AA^*\|\leq 1\}}
f_n(A) \, \Pro\left[\sup_{i=1,\ldots,n-m} \|C_{m+i}(V_{k,n})\|_2 < r \Big| A_m = A\right] \dint A,
\end{align*}
where $f_n$ is the Lebesgue density of $A_m$ given in~\eqref{eq:V_k_n_submatrix_density_repeat} and $\dint A$ refers to the integration with respect to Lebesgue measure on $\R^{k\times m}$. Assume that $n-m \geq k$, which is justified since later we let $n\to\infty$. Fix some matrix $A\in \R^{k\times m}$ with $\|AA^*\|\leq 1$. Conditionally on $A_m = A$ (that is, conditionally on the first $m$ columns of the Stiefel matrix $V_{k,n}$), the remaining columns $C_{m+1}(V_{k,n}), C_{m+2}(V_{k,n}),\ldots$ form a $k\times (n-m)$-matrix denoted by $\widetilde A_m$ whose distribution is that of $(1-A A^*)^{1/2} \widetilde V_{k, n-m}$, where $\widetilde V_{k, n-m}$ is a random, uniform $k\times (n-m)$-Stiefel matrix. This fact is due to Khatri~\cite[Lemma~2]{Khatri1970a} who proved that $A_m$ and $\widetilde V_{k, n-m}:= (1-A A^*)^{-1/2}\tilde A_m$ are independent, the first matrix being inverse $t$-distributed, while the second one being uniform on $\mathbb V_{k, n-m}$; see also~\cite[Theorem~8.2.2]{GN2000}. The columns of $(1-A A^*)^{1/2} \widetilde V_{k, n-m}$ are just $(1-A A^*)^{1/2}C_1(\widetilde V_{k, n-m}),(1-A A^*)^{1/2}C_2(\widetilde V_{k, n-m}),\ldots$. Hence, we have
\begin{align*}
\Pro \left[\eta_n \in W_{r,\rho}(\mu)\right]
&\geq
\int_{O_{r,\rho}(\mu) \cap \{A\in \R^{k\times m}:\, \|AA^*\|\leq 1\}}
f_n(A) \, \Pro\left[\sup_{i=1,\ldots,n-m}
\|(1-A A^*)^{1/2} C_{i}(\widetilde V_{k,n-m})\|_2 < r \right] \dint A\\
&\geq
\int_{O_{r,\rho}(\mu) \cap \{A\in \R^{k\times m}:\, \|AA^*\|\leq 1\}}
f_n(A) \, \Pro\left[\sup_{i=1,\ldots,n-m}
\|C_{i}(\widetilde V_{k,n-m})\|_2 < r \right] \dint A
\end{align*}
because the operator norm of $(1-A A^*)^{1/2}$ is at most $1$. We now claim that
\begin{equation}\label{eq:proof_D_lim_probab_sup_norm_column}
\lim_{n\to\infty}
\Pro\left[\sup_{i=1,\ldots,n-m}
\|C_{i}(\widetilde V_{k,n-m})\|_2 < r \right]
= 1.
\end{equation}
To prove this claim, observe that by the argument of Section~\ref{subsec:inverted_t_and_LDP_for_submatrix}, see in particular~\eqref{eq:V_k_n_submatrix_density} with $\ell= 1$, the density of the first column $C_{1}(\widetilde V_{k,n-m})$ is given by
\begin{equation}\label{eq:density_first_column_stiefel}
x\mapsto
%\frac{\Gamma_k(\frac{n-m}{2})}{\pi^{\frac{k}{2}}\Gamma_k(\frac{n-1}{2})}
%\det\Big(\id_{k\times k} - xx^*\Big)^{\frac{n-m-k-2}{2}} \mathbb 1_{\{\|x\|_2 <1\}}
%=
\frac{\Gamma_k(\frac{n-m}{2})}{\pi^{\frac{k}{2}}\Gamma_k(\frac{n-m-1}{2})}
(1 - x^2)^{\frac{n-m-k-2}{2}} \mathbb 1_{\{\|x\|_2 <1\}},
\qquad
x\in \R^{k\times 1}\equiv \R^k,
\end{equation}
where we used that the $k\times k$-matrix $xx^*$ appearing in~\eqref{eq:V_k_n_submatrix_density} is $\|x\|_2^2$ times the orthogonal projection onto the line spanned by $x$ and therefore $\det(\id_{k\times k} - xx^*) = 1-\|x\|_2^2$. The normalizing factor in~\eqref{eq:density_first_column_stiefel} grows polynomially in $n$, hence for every fixed $r\in (0,1)$ and every $\eps>0$, we have the exponential decay
$$
\Pro\big[\|C_{1}(\widetilde V_{k,n-m})\|_2 >r\big] = O\big((1-r^2)^{\frac {1-\eps}2 n}\big),
\qquad n\to\infty.
$$
The union bound completes the proof of~\eqref{eq:proof_D_lim_probab_sup_norm_column}.
It follows that for sufficiently large $n\in \N$, the probability in~\eqref{eq:proof_D_lim_probab_sup_norm_column} is larger than $1/2$ and we can write
\begin{align*}
\Pro \left[\eta_n \in W_{r,\rho}(\mu)\right]
&\geq
\frac 12 \int_{O_{r,\rho}(\mu) \cap \{A\in \R^{k\times m}:\, \|AA^*\|\leq 1\}}
f_n(A)  \dint A\\
&\geq
\frac{\Gamma_k(\frac{n}{2})}{2\pi^{\frac{km}{2}}\Gamma_k(\frac{n-m}{2})} \int_{O_{r,\rho}(\mu) \cap \{A\in \R^{k\times m}:\, \|AA^*\|\leq 1\}}
\det\Big(\id_{k\times k} - AA^*\Big)^{\frac{n-m-k-1}{2}}\dint A.
\end{align*}
Denoting the factor in front of the integral by $\kappa'(n)$, we observe that $\lim_{n\to\infty} \frac 1n \log \kappa'(n) = 0$. Recall that the condition $A\in O_{r,\rho}(\mu)$ means that the $m$ columns of $A$ belong to the balls $B(D_i, \rho)$, $i=1,\ldots, L$, where the $i$-th ball appears in the list $m_i$ times. For sufficiently large $n\in\N$, we have $\rho>1/n$ and hence $A\in O_{r,1/n}(\mu)$ implies that $A\in O_{r,\rho}(\mu)$ and also that $\|AA^*\|<1$ (the latter since $\|V(\mu) V(\mu^*)\| <1$). Estimating the integrand by its minimum, we obtain
\begin{align*}
\Pro \left[\eta_n \in W_{r,\rho}(\mu)\right]
&\geq
\kappa'(n)\cdot \vol_{k\times m} (O_{r,1/n}(\mu)) \cdot \left(\inf_{A\in O_{r,1/n}(\mu)} \det \Big(\id_{k\times k} - AA^*\Big)\right)^{\frac{n-m-k-1}{2}}.
\end{align*}
The volume of $O_{r,1/n}(\mu)= \otimes_{i=1}^L (\mathbb B(D_i, 1/n))^{m_i}$ decays polynomially in $n$. Taking the logarithm, dividing by $n$ and letting $n$ to infinity and using the continuity of the function $A\mapsto \det (\id_{k\times k} - AA^*)$, we get
$$
\liminf_{n\to\infty} \Pro \left[\eta_n \in W_{r,\rho}(\mu)\right]
\geq \frac 12 \log  \det (\id_{k\times k} - V_m(\mu) V_m(\mu)^*),
$$
where $V_m(\mu)$ is the $k\times m$- matrix formed by the first $m$ columns of $V(\mu)$ (which are $D_1,\ldots, D_L$ with multiplicities $m_1,\ldots, m_L$).  The right-hand side of the above inequality is $\geq \frac 12 \log  \det (\id_{k\times k} - V(\mu) V(\mu)^*)$, as we have shown in Lemma~\ref{lem:monotone_determinant_in_ell} of Section~\ref{subsec:dawson gaertner approach}, and the proof of the lower bound~\eqref{eq:lower_bound_theo_D_statement} is complete.

\vspace*{2mm}
\noindent
\textit{Completing the proof.} Taking together the upper and lower bounds~\eqref{eq:upper_bound_theo_D_statement} and~\eqref{eq:lower_bound_theo_D_statement}, we can apply Proposition~\ref{prop:basis topology} thereby obtaining a weak LDP for $\eta_n$ on the space $\mathbb W_k$. Since this space is compact by Lemma~\ref{lem:W_k_compact}, the weak LDP already implies the full LDP.
\end{proof}

\subsection{Step 7 -- LDP on \texorpdfstring{$\mathcal M_1(\R^k)$}{M\_1(Rk)}}
We can now complete the proof of Theorem~\ref{thm:ldp multidimensional projection_product_measures}. The random measure $\widetilde \mu_{V_{k,n}}$ appearing there is nothing else but $\Psi(\eta_n)$. Transforming the LDP for $\eta_n$ on the space $\mathbb W_k$ stated in Proposition~\ref{prop: LDP_for_nu_n} to the compact space $\mathcal K_{k, Y_1}= \Psi(\mathbb W_k)\subset \mathcal M_1(\R^k)$ by the map $\Psi$ (which is a homeomorphism between $\mathbb W_k$ and $\mathcal K_{k, Y_1}$, as we have shown in Corollary~\ref{cor:nu_homeo}), we arrive at Theorem~\ref{thm:ldp multidimensional projection_product_measures} since the function $\rate$ appearing there is $\rate(\nu) = \mathbb J(\Psi^{-1}(\nu))$.

\subsection*{Acknowledgment}
We are grateful to Gerold Alsmeyer for drawing our attention to the works of J. V.\ Linnik related to Lemma~\ref{lem:linear_combinations}.
Zakhar Kabluchko has been supported by the German Research Foundation under Germany’s Excellence Strategy EXC 2044 – 390685587, \textit{Mathematics M\"unster: Dynamics - Geometry - Structure} and by the DFG priority program SPP 2265 \textit{Random Geometric Systems}.
Joscha Prochno is supported by the Austrian Science Fund (FWF) Project P32405 \textit{Asymptotic Geometric Analysis and Applications} and by the FWF Project F5513-N26 which  is  a  part  of  the  Special Research  Program  \emph{Quasi-Monte  Carlo  Methods:  Theory  and  Applications}.

%\section{Appendix}

\bibliographystyle{plain}
\bibliography{ldp}

\bigskip
\bigskip
	
	\medskip
	
	\small

	\noindent \textsc{Zakhar Kabluchko:} Faculty of Mathematics, University of M\"unster, Orl\'eans-Ring 10,
		48149 M\"unster, Germany
		
	\noindent
		{\it E-mail:} \texttt{zakhar.kabluchko@uni-muenster.de}
	
		\medskip
	
	\noindent \textsc{Joscha Prochno:} Faculty of Computer Science and Mathematics,
	University of Passau, Dr.-Hans-Kapfinger-Strasse 30, 94032 Passau, Germany
	
	\noindent
	{\it E-mail:} \texttt{joscha.prochno@uni-passau.de}

\end{document}